\RequirePackage[l2tabu, orthodox]{nag}
\documentclass[a4paper,10pt]{amsart}

\usepackage[utf8]{inputenc}
\usepackage[T1]{fontenc}
\usepackage{tikz-cd} 
\usepackage[ruled,vlined]{algorithm2e}

\usepackage{amsmath,amsfonts,amssymb,amsxtra,setspace,xspace,graphicx,lmodern,psfrag,color,latexsym,bbm,comment,amsthm,enumerate,enumitem,pifont,mathrsfs,caption,subcaption,microtype,dsfont, bm,nicefrac}
\usepackage[foot]{amsaddr}
\usepackage{tkz-euclide}

\usepackage[a4paper]{geometry}

\usepackage[hidelinks,colorlinks=true]{hyperref}\hypersetup{urlcolor=blue, citecolor=blue}

\usepackage{mathtools}

\usepackage[capitalise]{cleveref}
\usepackage{autonum}
\numberwithin{equation}{section}
\makeatletter
 \def\@textbottom{\vskip \z@ \@plus 17pt}
 \let\@texttop\relax
\makeatother
\usepackage[acronym,nonumberlist,nogroupskip]{glossaries}
\usepackage{array,booktabs}

\newtheorem{thm}{Theorem}[section]
\newtheorem{cor}[thm]{Corollary}
\newtheorem{lem}[thm]{Lemma}
\newtheorem{prop}[thm]{Proposition}

\theoremstyle{plain}
\theoremstyle{definition}
\newtheorem{assumption}[thm]{Assumption}

\newtheorem{defn}[thm]{Definition}
\theoremstyle{remark}
\newtheorem{rem}[thm]{Remark} 

\Crefname{lem}{Lemma}{Lemmas}

\def\Xint#1{\mathchoice
{\XXint\displaystyle\textstyle{#1}}%
{\XXint\textstyle\scriptstyle{#1}}%
{\XXint\scriptstyle\scriptscriptstyle{#1}}%
{\XXint\scriptscriptstyle\scriptscriptstyle{#1}}%
\!\int}
\def\XXint#1#2#3{{\setbox0=\hbox{$#1{#2#3}{\int}$ }
\vcenter{\hbox{$#2#3$ }}\kern-.6\wd0}}

\def\dashint{\Xint-}

\renewcommand{\glossarysection}[2][]{} %% removes glossaries from toc

\DeclarePairedDelimiter{\abs}{\lvert}{\rvert}

\DeclarePairedDelimiter{\bra}{(}{)}
\DeclarePairedDelimiter{\pra}{[}{]}
\DeclarePairedDelimiter{\set}{\{}{\}}

\DeclareMathAlphabet{\mathup}{OT1}{\familydefault}{m}{n}
\newcommand{\dx}[1]{\mathop{}\!\mathup{d} #1}

\newcommand{\cM}{\ensuremath{\mathcal M}}

\newcommand{\N}{{\mathbb N}}
\newcommand{\R}{{\mathbb R}}

\newcommand{\Z}{{\mathbb Z}}

\newcommand{\eps}{{\varepsilon}}
\newcommand{\grm}{g}
\newcommand{\al}{\alpha}

\newcommand{\at}{\bra*{A^T}}

\newcommand{\intI}[1]{\int_{0}^1 #1 \dx{x}}

\renewcommand{\tilde}{\widetilde}
\renewcommand{\bar}{\overline}
\renewcommand{\eps}{\varepsilon}
\allowdisplaybreaks[4] % page breaks in long equations

\usetikzlibrary{arrows, arrows.meta, decorations.pathmorphing}

\author{Benjamin Gess$^{\dagger,\ddagger}$, Rishabh S. Gvalani$^\ddagger$, Florian Kunick$^\ddagger$, and Felix Otto$^\ddagger$}
\address[$\dagger$]{Universit\"at Bielefeld}
\email{bgess@math.uni-bielefeld.de}
\address[$\ddagger$]{Max-Planck-Institut f\"ur Mathematik in den Naturwissenschaften}
\email{gvalani@mis.mpg.de, kunick@mis.mpg.de,  otto@mis.mpg.de}

\date{\today}

\tikzset{snake it/.style={decorate, decoration=snake}}
\title[The thin-film equation with thermal noise]{Thermodynamically consistent and positivity-preserving discretization 
of the thin-film equation with thermal noise}
\begin{document}

\maketitle

\begin{abstract}

In micro-fluidics not only does capillarity dominate but also
thermal fluctuations become important. On the level of the lubrication
approximation, this leads to a quasi-linear fourth-order parabolic equation
for the film height $h$ driven by space-time white noise.
The (formal) gradient flow structure of its deterministic counterpart, the so-called thin-film equation, 
which encodes the balance between driving capillary and limiting viscous forces,
provides the guidance for the thermodynamically consistent introduction
of fluctuations. We follow this route on the level of a spatial discretization 
of the gradient flow structure, i.e., on the level of a discretization
of energy functional and dissipative metric tensor. 

\smallskip

Starting from an energetically conformal finite-element 
(FE) discretization, we point out that the numerical mobility function introduced by Gr\"un and Rumpf can be interpreted as a discretization of the metric tensor in
the sense of a mixed FE method with lumping. While this discretization was devised
in order to preserve the so-called entropy estimate, we use this to show that
the resulting high-dimensional stochastic differential equation (SDE)
preserves pathwise and pointwise strict positivity, 
at least in case of the physically relevant mobility function
arising from the no-slip boundary condition.

\smallskip

As a consequence, and opposed to more naive discretizations
of the thin-film equation with thermal noise, 
the above discretization is not in need
of an artificial condition at the boundary of the configuration space orthant $\{h>0\}$
(which admittedly could also be avoided by modelling a disjoining pressure). As a consequence,
this discretization gives rise to a consistent invariant measure, namely
a discretization of the Brown\-ian excursion (up to the volume constraint),
and thus features an entropic repulsion. The price to pay over more naive discretizations
is that when writing the SDE in It\^o's form, which is the basis for the
Euler-Mayurama time discretization, a correction term appears.

\smallskip

We perform various numerical experiments to compare the behavior and performance of our discretization to that of the more naive finite difference discretization of the equation. Among other things, we study numerically the invariance and entropic repulsion of the invariant measure and provide evidence for the fact that the finite difference discretization touches down almost surely while our discretization stays away from the $\partial\set*{h > 0}$.
\end{abstract}

\setcounter{tocdepth}{1}
\tableofcontents
\section{Introduction}

The thin-film equation models the evolution of the height $h$ of a liquid film 
over a solid flat substrate, as driven by capillarity\footnote{surface tension} and
limited by viscosity. In the considered
regime of small slope ($|\partial_xh|\ll 1$) and due to the no-slip boundary 
condition at the liquid-solid interface, viscous dissipation is so strong
that the liquid's inertia can typically be neglected.
Hence the dynamics are determined by a quasi-static balance between capillary
and viscous forces. The lubrication approximation, which is based on a modulated
Poiseuille Ansatz for the fluid velocity, leads to a fourth-order
parabolic equation with a mobility that cubically degenerates in the film height.

\medskip

In this paper, we are interested in the thin-film equation driven by the noise
that models thermal fluctuations. That noise takes the form of a conservative
white noise with a multiplicative non-linearity. 
The specific form of the multiplicative non-linearity -- it is given by the square root 
of the mobility --
formally arises from the fluctuation-dissipation principle, see \cite[(4)]{Davidovitch}.
The fluctuation-dissipation principle amounts to a linearized version of the property
of detailed balance, which in turn amounts to reversibility of the invariant measure
on path space.
While there exist elements of a well-posedness theory for (spatially) more regular 
forms of the noise in the mathematical literature, see \cite{FG18}, \cite{GG20} and \cite{DGGG20} and the next section for a detailed discussion, the stochastic partial differential equation (SPDE) we are interested in is
expected to require a renormalization, and is theoretically uncharted. 
However, at least in $1+1$-space dimensions\footnote{which means that the profile is
constant in one direction, so that the space variable $x$ is one-dimensional} 
as considered in this paper,
the invariant measure (on configuration space) of the SPDE does not require a renormalization. 
In this paper we ignore the issue of renormalization and focus on spatial\footnote{by which we mean the physical space variable $x$,
and not the state-space variable $h$} discretizations of this SPDE.

\medskip

The main issue is that the configuration space $\{h>0\}$, which after discretization has the
structure of an orthant, obviously has a boundary. The related preservation of positivity\footnote{often
in form of preservation of non-negativity if the interest was in film spreading 
and (partial) wetting} has been at the core of the analysis of the deterministic thin-film equation, both on the continuum level 
\cite{BF90, Beretta, DPGG98} and others, and on the level of spatial
discretization \cite{GruenRumpf, Bertozzi}. We refer to the end of the section for a more in-depth overview.
The preservation of strict positivity is intimately related to what is called the entropy estimate,
i.~e.~the existence of a Lyapunov functional on configuration space that blows up when $h$ 
approaches zero.
This Lyapunov functional depends on the mobility, 
and thereby arises from kinetics and dissipation,
and thus is actually unrelated to the notion of entropy in thermodynamic equilibrium theory. 
In fact, the blowing up of the entropy as 
$h \downarrow 0$ is a consequence of a sufficiently strong degeneracy of the mobility.
Of course, both in the discrete and the continuum case, such a touch-down can be
suppressed by introducing a disjoining pressure. However, this feature comes with an additional
(vertical) length scale of molecular size, and which one 
thus would like to avoid resolving. In this paper, we therefore disregard this 
energetic mechanism preventing touch-down, and just focus on the above-mentioned kinetic mechanism.

\medskip

In case of the thin-film equation with thermal noise, which in its discretized version
describes a drift-diffusion process on the high-dimensional orthant $\{h>0\}$, the
question is even more pressing: Does the process reach the boundary or is the degeneracy
of the mobility as $h\downarrow0$, which translates into a degeneracy of the diffusion 
near the boundary of $\{h>0\}$, strong enough to prevent reaching the boundary?
The fact that the boundary may be reached has been already recognized in \cite{Davidovitch}, 
where also an (uncontrolled) fix has been proposed.
For a rigorous analysis of a given discretization, we need a multi-dimensional
version of a Feller test. One main insight
of this paper is that such a Feller test can be carried out with help of the entropy
mentioned above. It shows that for the physical mobility considered in this paper, and in the case of $1+1$-dimensions,
the numerical mobility, which was introduced in \cite[Section 5]{GruenRumpf} in order to prevent touch-down
in the deterministic case, does also prevent touch-down in the presence of thermal noise (cf.~\cref{stentest}). However, in~\cref{sec:ldp} we provide evidence, through analysis of the path-space rate functional of the continuum stochastic thin-film equation, that the absence of touch-down maybe an artifact of discretization - for the continuum system touch-down is unlikely only for $m\geq 8$ (cf.~\cref{m<8}).

\medskip

The use of entropy estimates to construct non-negative solutions to the (deterministic) thin-film equation goes back to the original work \cite[p.190, (4.12)]{BF90}, proving the existence of non-negative solutions for mobility exponents $1<m<4$ (see \cref{mobilitym}) and preservation of positivity for $m\ge 4$. Subsequently, these estimates were refined by means of so-called $\alpha$-entropy estimates in \cite[p.182, Proposition 2.1]{Beretta} and \cite[p.99, (4.8) - (4.13)]{BP96}, which allowed to deduce the preservation of positivity for $m\ge \frac{7}{2}$. A generalization of the existence of non-negative solutions to multiple space dimension was given in \cite{G95} and extended to a wider range of mobility exponents in \cite[p.324, Proposition 2.2]{DPGG98}. Localized forms of $\alpha$-entropy estimates were subsequently introduced in \cite[Section 4]{B96} in $1+1$ dimensions and \cite[p.422, Theorem 3.1]{BDPGG98} in higher space dimensions and in \cite{DPGG01} used to prove upper bounds on the propagation of the support of solutions. Backward weighted entropy estimates have been introduced in \cite[Section 3]{F13} and \cite[p.3142, Lemma 11]{F14} to prove lower bounds on propagation rates. Also in the context of stochastic thin-film equations (with spatially regular noise) entropy estimates have been used in order to derive a-priori estimates and the existence of non-negative solutions \cite[p.423, Proposition 4.3]{FG18} and \cite[p.20, Lemma 4.3]{DGGG20}.

\medskip

As has been already mentioned, for the discretized thin-film equation the use of entropy estimates, which rely on an appropriate discretization of the mobility, dates back to \cite[Section 5]{GruenRumpf} in the case of a finite element discretization, and to \cite[p.529, Proposition 3.1]{Bertozzi} in the case of a finite difference discretization. In the discrete case the corresponding entropy estimates have a stronger effect yielding positivity already for $m \geq 2$ in case of the two aforementioned discretizations. In this paper we transfer the discretization and entropy estimate of \cite{GruenRumpf} to the stochastic setting and get positivity for the scheme for $m \geq 3$ (cf. \eqref{stentest}).

%%%%%%%%%%%%%%%%%%%%%%%%%%%%%%%%%%%%%%%%%%%%%%%%%%%%%%%%%%%%%%%%%%%%%%%%%%%%%%%%%%%%%%%%%

\section{State of the art}

In \cite[Section 2.3]{Gruen}, 
the authors make the ansatz of an (infinite-dimensional) SDE in It\^o form 
with a drift term given by\footnote{just, i.~e.~there is no It\^o correction term} 
the deterministic thin-film operator, see \cite[(36)]{Gruen}, and seek a noise term such that
the process satisfies detailed balance with respect to the associated Gibbs measure, see \cite[(21)]{Gruen}.
They carry this out on the level of a finite-difference discretization in space, 
based on centered finite differences, see \cite[p.1269]{Gruen} 
which allows to use a local numerical mobility function, see \cite[(29b)]{Gruen}.
Thanks to this simple structure\footnote{where there is no difference between the It\^o and
Stratonovich form} they find that this is the case, provided the multiplicative noise
involves the exact square root of the numerical mobility function, see \cite[(33)]{Gruen}.
However in this case, it is easy to see that the process does touch-down (cf. Section \ref{sec:cdd}).

\medskip

When it comes to actual simulations, \cite{Gruen} departs from this somewhat academic 
spatial discretization:
They treat the noise term, which due to its conservative and multiplicative nature
has the structure of a scalar conservation law with nonlinear and heterogeneous (in fact, rough)
drift, via a finite volume discretization with an upwind scheme, see \cite[(63),(64)]{Gruen}.
The upwind scheme preserves non-negativity. 
For the deterministic term, they however use the numerical mobility introduced in
\cite{GruenRumpf}, see \cite[(B.3)]{Gruen}, 
which is rather based on a lumped finite element interpretation, see \cite[p.1275]{Gruen}.
Again, at least on the purely deterministic level, this ensures non-negativity.
Using two different, and nonlocal, numerical mobility functions however destroys the structure
of exact detailed balance. The authors acknowledge this deficiency, see \cite[p.1278]{Gruen},
mentioning that the deviation from detailed balance is vanishing (of first order) in
the grid size. However, it is well-known that in the case of a singular SPDE, 
two different spatial discretizations, while both nominally first-order consistent, 
may lead to order-one different solutions (cf. \cite{HairerMaasAoP2012}).

\medskip

In~\cite[Sections 2 and 4]{Pavliotis}, the authors repeat the derivation of the infinite-dimensional SDE of~\cite{Gruen}, but obtain it in the limit of fully correlated noise in the wall-normal direction for the long-wave/lubrication approximation of the so-called fluctuating hydrodynamics equations (see~\cite[§88, (88.6)-(88.18)]{LLP80}). Following~\cite{Gruen}, the authors make, essentially, an identical observation, that a finite difference discretization of the associated stochastic thin-film equation is formally reversible with respect to the associated Gibbs measure if and only if the multiplicative noise is given by the square root of the associated mobility.  

\medskip

Again, for the purposes of numerical simulations,~\cite{Pavliotis} departs from the finite-difference discretization and instead proposes a spectral collocation method. The idea is to carry out the differentiation operations by decomposing the solution in terms of the eigenfunctions of the covariance operator of the noise, while treating the numerical mobility in a similar manner to the finite-difference discretization (see~\cite[Section 5.1, (71)-(72b)]{Pavliotis}). While this may have some structural advantages, it suffers from the drawback that it is unclear, and possibly untrue, that the spectral discretization satisfies detailed balance. Furthermore, it is also unclear if this scheme preserves the positivity of the film height.

\medskip

In recent years, the existence of probabilistically weak solutions to the stochastic thin-film equation has been considered in a sequence of works. In all of these works the noise term is spatially regularized. In \cite{FG18}, the authors constructed weak solutions for the case of quadratic mobility, relying on a conjoining-disjoining pressure term, and noise interpreted in It\^o sense. In \cite{Cornalba2018} more general mobilities were treated depending on a non-conservative source term. Both works require the initial condition to be strictly positive. For quadratic mobility and noise in Stratonovich sense, this restriction was lifted in \cite{GG20}. The case of cubic mobility without additional conjoining-disjoining pressure term was recently treated in \cite{DGGG20}. Recently, these results were extended to $2+1$ dimensions in \cite{MG21} and \cite{S21}. 

\section{The thin-film equation as a formal gradient flow}
The gradient flow structure of the thin-film equation is folklore by now (cf. \cite[p.2092 ff.]{O98}); we recall it for the reader's convenience.
In $1+1$ dimensions the equation takes the form
\begin{align}\label{tfe}
    \partial_t h + \partial_x \bra{ M(h) \partial^3_x h} = 0,  \quad  (t,x) \in (0,\infty) \times \R  \, ,
\end{align}
where $h$ is the film height and $M$ is called the mobility. In the following discussion, we tacitly think of $h > 0$ -- this paper does not address partial wetting, which would require more modelling assumptions at the contact line, like the equilibrium contact angle, possibly in conjunction with additional dissipation. Equation \eqref{tfe} is based on a lubrication approximation of a fluids equation, like Darcy or Stokes (cf. \cite{GO03, KM15}) and is a fourth order and possibly degenerate parabolic partial differential equation. The mobility $M(h)$ depends on the dissipation mechanism (e.g. Stokes vs. Darcy) and the boundary condition (e.g. no-slip vs. Navier) for the fluid velocity. Often, it is assumed that the mobility follows a power law, i.e. $M(h) \propto h^m$ for some $m \geq 0$. For example, Stokes with no-slip boundary conditions gives rise to $M(h) \propto h^3$ and this is also the most relevant case. Stokes with Navier slip leads to $M(h) \propto h^2$ for $h$ below the slip length, and Darcy yields $M(h) \propto h$.

\smallskip

In this paper, we make the convenient assumption that the solution $h$ of \eqref{tfe} is 1-periodic. Since we clearly have conservation of mass, i.e.
\begin{align}
    \frac{\dx}{\dx{t}} \int_{0}^1 h \dx{x} = 0,
\end{align}
we choose as the configuration space 
\begin{align}\label{configspace}
    \mathcal{M} := \set*{ h:\R \to \R : h \textrm{ 1-periodic}, h > 0, \int_0^1 h \dx{x} = 1}.
\end{align}
The thin-film equation on $\mathcal{M}$ is driven by capillarity in the form of the Dirichlet energy
\begin{align}\label{energy}
    E(h) := \frac{1}{2} \int_0^1 \bra*{\partial_x h}^2  \dx{x}
\end{align}
and limited by viscosity as described by the metric tensor \footnote{for which, by polarization, it is enough to specify the quadratic part} \footnote{Note that $\partial_x j + \dot{h} = 0$ determines $j$ up to an additive constant so that the infimum is taken on a single parameter. We opted for this representation because it extends verbatim to the higher dimensional case and will play a crucial role in the discretization.}
\begin{align}\label{mettens}
    g_h\bra*{\dot{h}, \dot{h}} := \inf_{j}\set*{ \int_0^1 \frac{j^2}{M(h)} \dx{x} : \partial_x j + \dot{h} = 0  }
\end{align}
where $\dot{h} \in T_h \mathcal{M}$, and the tangent space is given by
\begin{align}\label{tangentspace}
    T_h \mathcal{M} = \set*{\dot{h} :\R \to \R : \dot{h} \textrm{ 1-periodic,} \int_0^1 \dot{h} \dx{x} = 0}.
\end{align}
For $M(h) = h$, this metric tensor corresponds to the infinitesimal metric in the $2-$Wasserstein distance (cf. \cite[p.384, (35)-(36)]{BB00} and \cite[p.111]{geometryFO01}).

\medskip

Hence, it is natural to expect that the thin-film equation has the structure of a gradient flow, i.e. that \eqref{tfe} can formally be written as
\begin{align}
    %\frac{\dx}{\dx{t}} h
    \partial_t h = - \nabla E(h).
\end{align}
This can be understood in the following way. The energy functional $E$ gives rise to a differential defined as
\begin{align}\label{differential}
    \mathrm{diff}E|_h.\dot{h} := \frac{\dx}{\dx{s}}\Bigr|_{s = 0} E\bra*{h + s\dot{h}}
\end{align}
for $h \in \mathcal{M}$ and $\dot{h} \in T_h \mathcal{M}$, and we can define a gradient via the Riemannian structure for all $h \in \mathcal{M}$ as the unique element $\nabla E(h) \in T_h \mathcal{M}$ satisfying
\begin{align}\label{gradient}
    \mathrm{diff}E|_{h}. \dot{h} = g_h\bra*{\nabla E(h), \dot{h}} 
\end{align}
for all $\dot{h} \in T_h\mathcal{M}$. Hence, the gradient flow formulation $\partial_t h = - \nabla E(h)$
means that we have
\begin{align}\label{variational}
    \mathrm{diff}E|_{h}.\dot{h} + g_h\bra*{\partial_t h, \dot{h}} = 0
\end{align}
for all $h \in \mathcal{M}$ and $\dot{h} \in T_h\mathcal{M}$. More precisely, by considering the Euler--Lagrange equation for \eqref{mettens}, we have
\begin{align}\label{mettens2}
    g_h\bra*{\dot{h}, \dot{h}} = \intI{M(h) (\partial_x f)^2},
\end{align}
where the $1-$periodic $f$ is such that $\dot{h} + \partial_x\bra*{M(h) \partial_x f} = 0$. By polarization of \eqref{mettens2} and integration by parts we indeed obtain \eqref{variational}:
\begin{align}
    g_h\bra*{\partial_t h, \dot{h}} = \intI{\dot{h} \partial_x^2 h } \stackrel{\eqref{energy}, \eqref{differential}}{=} - \mathrm{diff}E|_h.\dot{h}.
\end{align}
Choosing $\dot{h} = \partial_t h$ in \eqref{variational} we recover the energy dissipation identity characteristic of gradient flows
\begin{align}\label{energydiss}
    \frac{\dx}{\dx{t}} E(h) = - g_{h}(\partial_t h, \partial_t h) = - \intI{M(h) (\partial^3_x h)^2} \leq 0 \, .
\end{align}
Often, the energy has further contributions next to the one coming from capillarity (cf. \eqref{energy}) giving for instance rise to a disjoining pressure. In fact, the choice of the energy functional will not be important for \cref{coord} and \cref{grsection} and so if not otherwise stated we will not further specify $E$.

\medskip

However, following \cite[p.188, (4.3)]{BF90} we define the function $s$ as a solution to the equation $s'' = \frac{1}{M}$ and then for $E$ being the Dirichlet energy this yields another Lyapunov functional
\begin{align}\label{entropy}
    S(h) := \intI{s(h)}
\end{align}
called entropy in the mathematical literature, and the following entropy estimate
\begin{align} \label{entest}
    \frac{\dx}{\dx{t}} S(h) = - \intI{(\partial_x^2 h)^2} \leq 0
\end{align}
holds. This estimate will play a major role in \cref{sec:pos}.

\medskip

The preservation of positivity
can also be interpreted geometrically in the sense that the evolution on the configuration space $\mathcal{M}$ does not touch its boundary $\partial \mathcal{M}$.

\section{Thermodynamically consistent introduction of fluctuations}\label{sec:therm}
\subsection{Invariant measure on configuration space and the associated reversible dynamics}\label{subsec:eqmeasure}
In agreement with the standard equilibrium thermodynamics, we postulate that the invariant measure on configuration space of the stochastic dynamics is given by the Gibbs measure
% Starting from the viewpoint of the deterministic thin-film equation as a gradient flow governed by capillarity and limited by viscosity, our general philosophy is that also the evolution of the law of the stochastic thin-film equation on a short-time scale is completely described by the kinetics, i.e. the metric tensor $g$ and on a long-time scale by the energetics, i.e. the energy functional $E$. Formally, the energy functional gives rise to a Gibbs measure:
\begin{align}\label{invmeasure}
    \dx\nu(h) = \frac{1}{Z} e^{-\beta E(h)} \dx h
\end{align}
for some $\beta > 0$, which up to the Boltzmann factor is the inverse temperature, and a normalization constant $Z$. Here one thinks of $\dx h$ as a uniform measure on the configuration space $\mathcal{M}$. In the special case where the energy functional is the Dirichlet energy (cf. \eqref{energy}), the measure \eqref{invmeasure} looks similar to the classical Wiener measure. This relation, though, is not quite correct due to the following three reasons. First of all, we are on a periodic domain and, secondly, we have the additional constraint $\int_{0}^1 h \dx{x} = 1$. Finally, the restriction to the orthant $\set*{h > 0}$ is the major difference. 

\medskip

Hence we have to think of \eqref{invmeasure} as a
% a shifted (so as to satisfy the volume constraint) 
Gaussian measure conditioned to be non-negative, i.e.
\begin{align}\label{invmeasure2}
    \dx\nu(h) = \frac{1}{Z} \mathbbm{1}\set*{h > 0} \dx\mu(h)
\end{align}
where $\mu$ is the so-called Gaussian free field, i.e. the stationary Gaussian measure with covariance operator given by $\bra*{-\beta\partial_{x}^2}^{-1}$ and conditioned on the spatial average being $1$.
We will refer to the measure $\nu$ on $\mathcal{M}$ as the conservative Brownian excursion due to its reminiscence to the classical Brownian excursion from stochastic analysis. Notice, however, that unlike in the case of the classical Brownian excursion, the set $\set*{h \geq 0} $ we are conditioning on is not a null set with respect to the measure $\mu$. In other words, the conservative Brownian excursion \eqref{invmeasure2} is absolutely continuous with respect to the Gaussian free field, and it is well known that the latter is supported on $C^{\frac{1}{2}-}$-functions, and hence so is $\nu$.
% It is quite straightforward to see that the measure \eqref{invmeasure2} has the \green{same support} as the Brownian excursion
% which in turn has the same pathwise regularity as the Brownian motion (). 
% would refer to a Lebesgue measure on the infinite-dimensional configuration space $\mathcal{M}$ which of course does not exist but o %This measure also features an entropic repulsion which means that it is supported on the orthant $\set*{h > 0}$. 

\medskip

In the case of zero Dirichlet boundary data, the Brownian bridge conditioned to non-negative functions $\dx\tilde\nu(h) = \frac{1}{Z} \mathbbm{1}_{\set*{h \geq 0}} \dx\mu(h)$ corresponds to the law of the Brownian excursion, which in turn is the law of the 3d Bessel bridge (cf. \cite[p.205, Theorem 3]{Z01}). As a consequence, the transience of the 3d Brownian motion implies that $\tilde\nu$ is supported on positive functions. This repulsive effect of the boundary $\partial \mathcal{M}$ is called entropic repulsion.
Entropic repulsion in discrete systems and interface models has been analyzed, for example, in \cite{DG00}. Brownian excursion with fixed average has been realized as an invariant measure of an SPDE in \cite{Z08}. 
% In the case of the stochastic thin-film equation, this degenerate mobility is complemented by the entropic repulsion indicated above. Since entropic repulsion is a purely stochastic phenomenon, it could offer an alternative approach to proving strict positivity of solutions to the stochastic thin-film equation. 
%this can be seen as follows: by Theorem 1.5 on p.42 in \cite{RB92} the Brownian excursion has the same support as a three dimensional Bessel process which is the radial part of a three dimensional Brownian motion by Proposition 3.21 on p. 159 in \cite{KS88};

\medskip

We note in passing that in $2+1$-dimensions, the Gaussian measure would be related to
the two-dimensional Gaussian free field, so that in view of the latter's ultraviolet
logarithmic divergence, the conditioning on $h > 0$ is 
(borderline) singular; hence the nature of the Gibbs measure is unclear in this case.

\medskip

We now turn to the stochastic dynamics.
We follow the standard Ansatz that the time evolution of the law $\nu_{t}$ -- which we will assume to be absolutely continuous with respect to the invariant measure $\nu$ -- of the stochastic thin-film equation is described by the following Fokker--Planck equation in variational form, i.e. we have
\begin{align}\label{fpevar}
       \frac{\dx}{\dx{t}} \int_{\mathcal{M}} \zeta \dx{\nu_{t}} = - \frac{1}{\beta} \int_{\mathcal{M}} g\bra*{\nabla \zeta, \nabla f_t} \dx{\nu}
       %\rho_0 = \rho
\end{align}
for all sufficiently nice test functions $\zeta$ and where $f_t := \frac{\dx \nu_t}{\dx \nu}$. It is obvious from \eqref{fpevar} that $\nu$ is indeed invariant. The symmetry of the so-called Dirichlet form on the r.h.s. of \eqref{fpevar} implies that the generator $\mathcal{L}$, which is defined as the representation of the Dirichlet form w.r.t. $L^2(\dx \nu)$, is symmetric. This in turn yields that the stochastic process is reversible, meaning that the invariant measure on path space is invariant under reversing the time direction.
%As is well known this leads to the following PDE formulation of the Fokker-Planck equation
%\begin{align}
%    \begin{cases}
%        \partial_t \rho_t(h) = \epsilon \Delta \rho_t(h) + \nabla \cdot \bra*{\nabla E(h) \rho_t(h)} \ \mathrm{on} \ \mathcal{M} \\
%        \rho_0 = \rho
%    \end{cases}
%\end{align}
%where $\nabla$ denotes the Riemannian gradient defined in \ref{gradient}, $\nabla \cdot$ is the Riemannian divergence defined as being the %$L^2(dh)$-adjoint of the Riemannian gradient, and $\Delta = \nabla \cdot \nabla$ is the Laplace-Beltrami operator on $\mathcal{M}$.
As we will see later, this ansatz will ensure that the dynamics obey the detailed balance condition known from thermodynamics. 

\subsection{Renormalization of the thin-film equation with thermal noise}
In \cite[(4)]{Davidovitch} it has been suggested that the thin-film equation with thermal noise is given by 
\begin{align}\label{stfe}
    \partial_t h + \partial_x \bra*{M(h)\partial_x^3 h} = \partial_x \bra*{\sqrt{M(h)} \xi}
\end{align}
where $\xi$ denotes space-time white noise. In the course of this paper, it will become apparent that \eqref{stfe} arises from \eqref{fpevar}.
% In view of \eqref{finaldstfematrix} one would expect that the SPDE corresponding the law $\nu_t$ which evolves according to~\eqref{fpevar} should have a similar structure. 
% This is a so called singular SPDE, i.e. there are nonlinear terms which are not well-defined in the classical sense. 
First, we explain why equation \eqref{stfe} is singular
as an SPDE which means that there are nonlinear terms which are not well-defined a priori in a classical sense. This is in contrast to versions of the thin-film equation driven
by a less singular (and thus less physical)
noise than white noise, for which a well-posedness theory exists, see the discussion in Section 2. 
% As we have mentioned already, by thermodynamic consistency, the stationary measure formally
% is the Gibbs measure for the
% surface energy in its thin-film version $E(h)=\intI{(\partial_x h)^2}$. 
% Specifying to the torus as physical domain, 
% this measure can be considered as the stationary Gaussian measure with covariance operator 
% $(-\partial_x^2)^{-1}$, restricted to the (convex) state space of $h\ge 0$ and $\intI{ h}=1$.

\medskip

As a consequence of the characterization of the invariant measure on configuration space in \cref{subsec:eqmeasure},
we expect typical solutions $h$ of the thin-film equation with thermal noise
to have spatial regularity in the H\"older class $C^{\frac{1}{2}-}$ and not better. 
Hence, the product $M(h)\partial_x^3h$ appearing in the thin-film operator
is the product of a function in $C^{\frac{1}{2}-}$ and a distribution in the
negative H\"older space\footnote{see a couple of sentences below for a definition} 
$C^{-\frac{5}{2}-}$ and thus ill-defined (and more than just border-line 
since $(\frac{1}{2}-)+(-\frac{5}{2}-)
=-2-$). 

\medskip

Moreover, we encounter a similar difficulty in the multiplicative noise term
that formally is given by
$\partial_x(\sqrt{M(h)}\xi)$:
%with space-time white noise $\xi$: 
Since the effective dimension for our fourth-order parabolic 
operator in one space dimension is $4+1=5$, $\xi$ is in the negative H\"older class
$C^{-\frac{5}{2}-}$ (which can be defined as $\partial_tC^{\frac{3}{2}-}$ 
$+\partial_x^3C^{\frac{1}{2}-}$, where space-time H\"older spaces are defined
w.~r.~t.~to the anisotropic fourth-order parabolic Carnot-Carath\'eodory norm).
Hence the product $\sqrt{M(h)}\xi$ has the same singular
nature as the product $M(h)\partial_x^3h$. This similarity in the degree of singularity
is reminiscent of quasi-linear second-order equations (cf. \cite{OW19}).
We stress that these difficulties are unrelated to the 
degeneracy\footnote{meaning that $M(0)=0$} of $M$.

\medskip

Hence, the thin-film equation with thermal noise is in need 
of a renormalization, a pressing and attractive topic for the theory of singular SPDE.
In this paper, we do not further
address this issue for several reasons: 1) In 1+1-space dimensions,
as mentioned above, the invariant measure is not in need of a renormalization. Hence
the situation is better than in case of the well-studied stochastic quantization equation\footnote{which comes in form of the Allen-Cahn equation driven by space-time white noise}.
The 
invariant measure for the latter equation\footnote{also known as $\phi^4$ model in quantum field theory}
is in need of a renormalization for space dimensions $\ge 2$ (and renormalizable
in dimensions $<4$).
2) In this paper, we focus on structural properties of spatial discretizations that
can be rigorously addressed without a well-posedness theory for the continuum limit.
3) A simple but typical scaling argument
suggests that our problem is renormalizable in 1+1-space dimensions. 
Indeed, zooming in on small length and time scales through
\begin{align}\label{rescaling}
    x=\ell\hat x,\quad t=\ell^4\hat t,\quad h=1+\ell^\frac{1}{2}\hat h,\quad
    \xi=\ell^{-\frac{5}{2}}\hat \xi,
\end{align}
where the rescaling of $\xi$ is such that $\hat\xi$ is another instance of space-time white noise,
and where $1$ could be replaced by any positive constant,
the equation \eqref{stfe} turns into
\begin{align}
    \partial_{\hat t}\hat h
    +\partial_{\hat x}\big(M(1+\ell^\frac{1}{2}\hat h)\partial_{\hat x}^3\hat h\big)
    =\partial_{\hat x}\big(\sqrt{M(1+\ell^\frac{1}{2}\hat h)}\hat\xi\big),
\end{align}
from which we learn that on small scales, the non-linearity fades away
\footnote{this discussion obviously ignores additional difficulties that may arise
from the degeneracy of the mobility}. A similar computation shows that in 2+1-space dimensions the stochastic thin-film equation is critical, i.e. the rescaling \eqref{rescaling} leaves the equation \eqref{stfe} invariant and hence the nonlinear terms persist on small scales.

\medskip

There is a fourth point that we would like to make. Although at first sight the singular nature of the equation is very far from borderline, it is better than expected in some specific cases. As is common in the deterministic
rigorous treatment, one could rewrite the non-linearity in the thin-film operator in a less singular way: 
% in case of linear mobility
% as $h\partial_x^3h$ $=\partial_x^3\frac{1}{2}h^2-\partial_x\frac{3}{2}(\partial_xh)^2$.
% However, the term $(\partial_xh)^2$ is still ill-defined, like in case of the
% KPZ equation (cf. \cite[Thm. 15.1, p. 223]{FH14}).
\begin{align}
    M(h)\partial_x^3 h = \partial_x^3 \bar{M}(h) - \frac{3}{2} \partial_x \bra*{M'(h) \bra*{\partial_x h}^2} + \frac{1}{2} M''(h) \bra*{\partial_x h}^3
\end{align}
where $\bar{M}$ is the antiderivative of $M$. Of course the terms $\bra*{\partial_x h}^2$ and $\bra*{\partial_x h}^3$ are still singular but if we choose the following ansatz for renormalization which is inspired by the $\phi^4$-model 
\begin{align}
    \bra*{\partial_x h}^2 \to \bra*{\partial_x h}^2 - C , \quad \bra*{\partial_x h}^3 \to \bra*{\partial_x h}^3 - 3C \partial_x h
\end{align}
% where $C$ is a divergent constant, then we notice that since
the divergent constant $C$ drops out since by the chain rule 
% \begin{align}
%     -3C \partial_x M'(h) + 3CM''(h)\partial_x h = 0
% \end{align}
% and thus
\begin{align}
    &-\frac{3}{2}\partial_x\bra*{M'(h) \bra*{\bra*{\partial_x h}^2 - C }} + \frac{1}{2} M''(h) \bra*{\bra*{\partial_x h}^3 - 3C \partial_x h} \\
    &= - \frac{3}{2}\partial_x \bra*{M'(h) \bra*{\partial_x h}^2} + \frac{1}{2}  M''(h) \bra*{\partial_x h}^3.
\end{align}
While this argument suggests that the non-linearity $M(h) \partial_x^3h$ is less singular than expected, we now argue that the non-linearity $\sqrt{M(h)}\xi$ can be completely avoided in case of linear mobility, i.e. $M(h) = h$.
It is well known (cf. \cite[ p.74, Theorem 2.18]{CV03}) that for linear mobility under the change of variables $h \mapsto X$ where $X$ is the inverse distribution function of $h$, i.e.
\begin{align}\label{Xdef}
    z = \int_0^{X(z)} h(x) \dx{x},
\end{align}
the metric tensor transforms as
\begin{align}
    g_h\bra*{\dot{h}, \dot{h}} = \int_{0}^1 \dot{X}^2 \dx z = g_X\bra*{\dot{X}, \dot{X}}.
\end{align}
% Taking the derivative twice with respect to $z$ of \eqref{Xdef} yields
% \begin{align}\label{formulah}
%     1 = h(X(z)) \frac{\dx}{\dx z} X(z)
% \end{align}
% as well as
% \begin{align}\label{formuladxh}
%     0 = \partial_xh(X(z)) \bra*{\frac{\dx}{\dx z}X(z)}^2 + h(X(z)) \frac{\dx^2}{\dx z^2} X(z).
% \end{align}
% Multiplying \eqref{formuladxh} with $h(X(z))^2$ and invoking \eqref{formulah} we end up with
% \begin{align}\label{formuladxh2}
%     \partial_x h(X(z)) = - h(X(z))^3  \frac{\dx^2}{\dx z^2} X(z).
% \end{align}
% Hence we compute for the Dirichlet energy
% \begin{align}
%     E(h)&:= \frac{1}{2} \int_0^1 (\partial_x h)^2 \dx x = \frac{1}{2} \int_0^1 (\partial_x h(X(z)))^2 \frac{\dx}{\dx z}X(z) \dx z \\
%     &\stackrel{\eqref{formuladxh2}}{=} \frac{1}{2} \int_0^1 \bra*{h(X(z))^3 \frac{\dx^2}{\dx z^2} X(z)}^2 \frac{\dx}{\dx z}X(z) \dx z \\
%     &\stackrel{\eqref{formulah}}{=} \frac{1}{2} \int_0^1\frac{\bra*{\frac{\dx^2}{\dx z^2} X(z)}^2}{\bra*{\frac{\dx}{\dx z}X(z)}^5} \dx z \\
%     &=: E(X).
% \end{align}
The Dirichlet energy transforms according to
\begin{align}
     E(X) = \frac{1}{2} \int_0^1\frac{\bra*{\frac{\dx^2}{\dx z^2} X(z)}^2}{\bra*{\frac{\dx}{\dx z}X(z)}^5} \dx z.
\end{align}
% Moreover, for some $\delta X$ we compute
% \begin{align}
%     \mathrm{diff}E|_X.\delta X &= \frac{1}{2} \int_0^1 2 \frac{\frac{\dx^2}{\dx z^2} X(z)}{\bra*{\frac{\dx}{\dx z}X(z)}^5} \frac{\dx^2}{\dx z^2}(\delta X(z)) - 5 \frac{\bra*{\frac{\dx^2}{\dx z^2} X(z)}^2}{\bra*{\frac{\dx}{\dx z}X(z)}^6} \frac{\dx}{\dx z}(\delta X(z)) \dx z \\
%     &= \int_0^1 \bra*{\frac{\dx^2}{\dx z^2}\bra*{\frac{\frac{\dx^2}{\dx z^2} X(z)}{\bra*{\frac{\dx}{\dx z}X(z)}^5}} + \frac{5}{2} \frac{\dx}{\dx z}\bra*{\frac{\bra*{\frac{\dx^2}{\dx z^2} X(z)}^2}{\bra*{\frac{\dx}{\dx z}X(z)}^6}}} \delta X(z) \dx z.
% \end{align}
Hence the deterministic dynamics amount to the $L^2$-gradient flow of $E$, which is seen to assume the form
\begin{align}
    \partial_t X = \frac{1}{4} \partial_z^3 \bra*{\partial_z X}^{-4} - \frac{5}{8} \partial_z \bra*{\partial_z \bra*{\partial_z X}^{-2}}^2
\end{align}
and then \eqref{fpevar} can be seen to translate into
\begin{align}\label{lagstfe}
    \partial_t X = \frac{1}{4} \partial_z^3 \bra*{\partial_z X}^{-4} - \frac{5}{8} \partial_z \bra*{\partial_z \bra*{\partial_z X}^{-2}}^2 + \xi
\end{align}
where $\xi$ is space-time white noise. The first term on the right hand side of \eqref{lagstfe} is well-defined since $\partial_z X$ behaves like $h$ (cf. \eqref{Xdef}), and a non-linearity in the H\"older continuous $h$ is still harmless.
For the second term on the right hand side of \eqref{lagstfe}, we notice that it is a ``KPZ-like'' term followed by a derivative. Since the renormalization constant for the KPZ equation does not depend on the space variable (cf. \cite[p.223, Theorem 15.1]{FH14}) we might expect that in this case it is annihilated by the outer derivative.  Thus, one may expect that the leading order counter terms are zero and one obtains only higher order counter terms.
% This heuristic is confirmed by numerical simulations even for $m=3$.
\medskip

 We will comment further on the possible structure of a renormalizing counter term in~\cref{rem:itocorrection}, once we introduce both our discretization and the central difference discretization. Furthermore, in the numerical experiments performed in \cref{conv2} we observe that the two-point\footnote{in time} distribution functions of the two discretizations we are considering in this paper converge to the same object. This provides some numerical evidence for the our guess that equation \eqref{stfe} is less singular than expected, even for $M(h) = h^3$.

\section{Discretization}
A numerical treatment requires a discretization. From the Fokker--Planck equation
in its variational form (4.2) we learn that it is determined by the
triple $({\mathcal M},g,E)$, which hence we need to discretize. For the
function space ${\mathcal M}$, we choose a Finite Element discretization. More precisely, we fix $N \in \N$ and denote the equidistant partition of the torus by $\set*{x_i}_{i = 1, \dots, N}$. Then we denote by $P_1$ the space of $1$-periodic, continuous, and piecewise linear (with respect to the equidistant partition) functions and we set
% Since $\mathcal{M}$ is infinite-dimensional and we want to introduce coordinates as seen in Section 4.1, we discretize in the spatial variable such that we end up with a Fokker--Planck equation on a finite-dimensional manifold. To this end, we fix $N \in \N$ and denote the equidistant partition of the unit interval $[0, 1]$ by $\set*{x_i}_{i = 1, \dots, N}$.
% %, i.e. the partition satisfies the property that $x_i + x_j = x_{i + j}$
% %for $i, j \in \set*{1, \dots, N}$ and where we understand addition modulo $N$. 
% Then we denote by $P_1$ the space of $1$-periodic, continuous, and piecewise linear (with respect to the equidistant partition) functions. Abstractly, a spatial discretization of a gradient flow $\partial_t h = - \nabla E(h)$ is given by three ingredients. First, we need a finite-dimensional configuration space $\mathcal{M}_N$, which for us will always be 
\begin{align}
    \mathcal{M}_N := \mathcal{M} \cap P_1 = \set*{h \in P_1 : h > 0, \intI{h} = 1},
\end{align}
which then comes with a canonical tangent bundle $T\mathcal{M}_N$. 
% Secondly, an energy functional on $\mathcal{M}_N$, which for us will always be the restriction of the energy $E$ (see~\eqref{energy}) to $\mathcal{M}_N$ which amounts to an energetically conformal discretization and 
For the functional $E$, we make a conformal Ansatz by restricting to ${\mathcal M}_N$. 
This gives rise to a discretized conservative Brownian excursion $\nu_N$ according to \eqref{invmeasure}. 
Finally, we need to specify a metric tensor on $T\mathcal{M}_N \otimes T\mathcal{M}_N$.
A natural discretization of the metric tensor would be its restriction to the space $\mathcal{M}_N$. However, we will not consider this discretization in this paper for reasons explained in \cref{distancetoboundary}.
\section{Introducing coordinates}\label{coord}
In \cref{subsec:eqmeasure} we have already seen how a gradient flow structure, as determined by a Riemannian manifold $(\mathcal{M}, g)$ and a function $E$, gives rise to a stochastic process via the Fokker--Planck equation (cf. \eqref{fpevar}). In this section, we aim to write this process in It\^o form. To this end, we need to introduce coordinates. Let $\mathcal{M}$ be a differentiable Riemannian manifold with boundary, equipped with a Riemannian metric $g$, and assume that we have a global chart
\begin{align}
    (\varphi^{\al})_{\al} : \mathcal{M} \to \Delta,
\end{align}
where $\Delta$ is an open subset of $\R^N$ with coordinates enumerated by $\al=1, \dots, N$. 
Moreover, we think of $\mathcal{M}$ as equipped with a probability measure $\nu$. Then, these data give rise to a Fokker--Planck equation in variational form (cf. \eqref{fpevar}) which describes the time evolution of the probability measure $\nu_t$ which we assume to be absolutely continuous with respect to $\nu$.
% \begin{align}\label{fpeman}
%     \frac{\dx}{\dx{t}} \int_{\mathcal{M}} \zeta \dx \nu_{t} = - \frac{1}{\beta} \int_{\mathcal{M}} g\bra*{\nabla \zeta, \nabla f_t} \dx \nu
% \end{align}
% for all sufficiently nice test functions $\zeta$ where $f_t := \frac{\dx \nu_t}{\dx \nu}$. 
Hence \eqref{fpevar} gives rise to a Markovian stochastic process on $\mathcal{M}$ of which $\nu$ is the invariant measure. By the symmetry of the right hand side of \eqref{fpevar}, the resulting process on path space is reversible. 

\medskip

The chart $(\varphi^{\al})_{\al}$ allows to pull back functions from $\Delta$ to $\mathcal{M}$ and thus to push forward measures from $\mathcal{M}$ to  $\Delta$. For notational convenience we will not distinguish between $\zeta \circ \varphi$ and $\zeta$, between $f_t \circ \varphi$ and $f_t$, and between $\varphi \# \nu_t$ and $\nu_t$, and will write $h^{\al}$ instead of $\varphi^{\al}(h)$. A quick calculation shows that Radon-Nikodym derivatives transform like functions; in particular, the relation $\dx \nu_t = f_t \dx \nu$ lifts from $\mathcal{M}$ to $\Delta$.
By the usual duality, we define the gradient of $\varphi^{\al}$ as the unique element $\nabla \varphi^{\al}(h) \in T_h \mathcal{M}$ satisfying
\begin{align}
    \mathrm{diff}\varphi^{\al}|_{h}.\dot{h} = g_h\bra*{\nabla \varphi^{\al}(h), \dot{h}}
\end{align}
for all $\dot{h} \in T_h \mathcal{M}$ (cf. \eqref{gradient}).
While here, we think of the metric tensor as a bilinear form on tangent vectors, it is now convenient to consider its dual, a bilinear form on co-tangent vectors like differentials. The coordinate representation of this dual metric tensor is given by
\begin{align}\label{dualmetric}
    g^{\al \al'}(h) = \mathrm{diff}\varphi^{\al}|_{h}. \nabla \varphi^{\al'}(h).
\end{align}
The upper indices indicate the 2 contra-variant nature of the dual metric tensor. In fact, seen as a matrix, it is the inverse of the metric tensor $g_{\alpha\alpha'}(h)$ (cf. \eqref{ident}). 
% \begin{rem}
%     Note that in coordinates the dual metric is the inverse of the metric in coordinates (see \eqref{ident}). Thus we refer to the dual metric in coordinates by upper indices.
% \end{rem}
Then by \ref{gradpairing}, we get
\begin{align}
    g\bra*{\nabla \zeta, \nabla f_t} = g^{\al \al'} \partial_{\al} \zeta \partial_{\al'} f_t.,
\end{align}
where from now on we will use the Einstein convention of summing over repeated indices if not otherwise stated. 
Hence, we end up with the Fokker--Planck equation in variational form on $\Delta$, i.e.
\begin{align}\label{dfpevar}
    \frac{\dx}{\dx{t}} \int_{\Delta} \zeta \dx \nu_t &= - \frac{1}{\beta} \int_{\Delta} g^{\al \al'} \partial_{\al} \zeta   \partial_{\al'} f_t \dx \nu
\end{align}
for all sufficiently nice test functions $\zeta$.
Without much loss of generality, we assume that $\nu$ is given by
\begin{align}\label{SDEbar}
    \dx \nu(h) = \frac{1}{Z_{\beta}} e^{-\beta E(h)} \dx h
\end{align}
for some function $E : \Delta \to \R$ where $\dx h$ denotes the Lebesgue measure on $\Delta$. For brevity we set $\rho_{\infty} := \frac{1}{Z_{\beta}} e^{-\beta E}$.
%\begin{align}
%    \int_{\Omega} \hat{\zeta}(\hat{h}) \dx \hat{\nu} (\hat{h}) = \int_{\mathcal{M}} \zeta(h) \frac{1}{Z_{\beta}} e^{-\beta E(h)} \dx h = \int_{\Omega} %\hat{\zeta}(\hat{h}) \frac{1}{Z_{\beta}} e^{-\beta \hat{E}(\hat{h})} \dx \hat{h}
%\end{align}
%and hence
%\begin{align}
%    \dx \hat{\nu}(\hat{h}) = \frac{1}{Z_{\beta}} e^{-\beta \hat{E}(\hat{h})} \dx \hat{h}
%\end{align}
%where $\hat{E}$ is such that $E = \hat{E} \circ \varphi$.
Then we apply the divergence theorem which yields the following equation for $f_t$
\begin{align}\label{feq}
    \begin{cases}
        \rho_{\infty} \partial_t f_t  = \frac{1}{\beta}  \partial_{\al} \bra*{g^{\al \al'} \rho_{\infty} \partial_{\al'} f_t }  \ &\mathrm{in} \ \Delta, \\
       n_{\al} g^{\al \al'} \partial_{\al'} f_t   = 0\ &\mathrm{on} \ \partial \Delta, \\
    \end{cases}
\end{align}
where $n = (n_{\al})_{\al}$ denotes the outer normal of the boundary $\partial \Delta$. Moreover, considering the probability density $\rho_t$ defined through
\begin{align}
    \rho_t := f_t \rho_{\infty}
\end{align}
we see by \eqref{feq} and the Leibniz rule that $\rho_t$ solves the Fokker--Planck equation
\begin{align}\label{fpebar}
    \begin{cases}
        \partial_t \rho_t = \partial_{\al}\bra*{g^{\al \al'} \bra*{\frac{1}{\beta}  \partial_{\al'} \rho_t + \rho_t \partial_{\al'} E }} \ &\mathrm{in} \ \Delta, \\
        n_{\al} g^{\al \al'} \bra*{  \frac{1}{\beta}\partial_{\al'} \rho_t + \rho_t \partial_{\al'} E } = 0 \ &\mathrm{on} \ \partial \Delta. \\
    \end{cases}
\end{align}
Note that \eqref{fpebar} can be seen as a continuity equation for the probability density with the probability flux $J(\rho)$ being defined in components as
\begin{align}
    J^{\al}(\rho) :=  g^{\al \al'} \bra*{\frac{1}{\beta} \partial_{\al'} \rho + \rho \partial_{\al'} E }.
\end{align}
Then not only is $\rho_{\infty}$ the stationary solution of \eqref{fpebar} but in fact we have that
\begin{align}\label{fluxvan}
    J(\rho_{\infty}) \equiv 0
\end{align}
which corresponds to the so-called detailed balance condition (cf. \cite[p.119, (4.97)]{bookPavliotis}). 
% More precisely, if we define the Fokker--Planck operator
% \begin{align}
%     \mathcal{L}^*\rho := partial_{\al}\bra*{g^{\al \al'} \bra*{\frac{1}{\beta}  \partial_{\al'} \rho_t + \rho_t \partial_{\al'} E }}
% \end{align}
% then the detailed balance condition implies that the generator $\mathcal{L}$ is a self adjoint operator in the space $L^2(\dx \nu)$.
% The Fokker--Planck operator, i.e. the right hand side in \ref{fpehat}, can be rewritten as
% \begin{align}
%     &\partial_{\al} \bra*{\frac{1}{\beta} g^{\al \al'}(h) \partial_{\al'} \rho_t(h) + g^{\al \al'}(h) \partial_{\al'} E(h) \rho_t(h)} \\
%     &= \partial_{\al} \bra*{\frac{1}{\beta}  \partial_{\al'}\bra*{g^{\al \al'}(h) \rho_t(h)} - \partial_{\al'} g^{\al \al'}(h) \rho_t(h) + g^{\al \al'}(h) \partial_{\al'} E(h) \rho_t(h)}.
% \end{align}
Instead of describing the evolution of the law through \eqref{fpebar}, we can use the duality between measures and continuous functions to compute the evolution of observables $u$ of the process. Indeed, by computing the formal adjoint of~\eqref{fpebar}, we can read off the following backward Kolmogorov equation:
% For $T > 0$ we set $u_t := f_{T-t}$ for $f_t$ in \eqref{feq} and notice that $u_t$ solves the backward Kolmogorov equation 
\begin{align}\label{bke}
        \partial_t u_t &= \frac{1}{\beta} \partial_{\al} \bra*{g^{\al \al'} \partial_{\al'} u_t} - \partial_{\al} u_t g^{\al \al'} \partial_{\al'} E \\
        &= \frac{1}{\beta} g^{\al \al'} \partial_{\al \al'} u_t + \partial_{\al} u_t \bra*{\frac{1}{\beta} \partial_{\al'} g^{\al' \al} - g^{\al \al'} \partial_{\al'} E} 
\end{align}
in $\Delta$ equipped with the boundary conditions on $\partial \Delta$
\begin{align}\label{bkeb}
    n_{\al} g^{\al \al'}  \partial_{\al'} u_t  = 0.
\end{align}
Note that the right hand side of~\eqref{bke} is the generator of the associated diffusion process. Thus, we can use~\eqref{bke} to identify the stochastic process $h^{\al}_t$ arising from \eqref{dfpevar}. Indeed, its drift is given by  $-\bra*{\frac{1}{\beta} \partial_{\al'} g^{\al' \al} - g^{\al \al'} \partial_{\al'} E}(h_t)$ and its diffusion matrix by $\frac{1}{\beta} g^{\al \al'}(h_t)$. This gives rise to the following stochastic differential equation in It\^o form (cf. \cite[p.126, Theorem 7.3.3 and p.152, Theorem 8.4.3]{BO03})
\begin{align}\label{itosdebar}
        \dx h^{\al}_t = \bra*{ - g^{\al \al'}(h_t) \partial_{\al'}E(h_t) + \frac{1}{\beta}  \partial_{\al'} g^{\al' \al}(h_t) } \dx t + \sigma^{\al}_{\al'}(h_t) \sqrt{ \frac{2}{\beta}} \dx W^{\al'}_t,
\end{align}
where $\sigma^{\al}_{\al'}$ denotes any matrix satisfying $g^{\al \al'} = \sum_{\al''=1}^N \sigma^{\al}_{\al''} \sigma^{\al'}_{\al''}$, and $W_t$ is a standard Wiener process. Furthermore, the no-flux boundary conditions in \eqref{bkeb} correspond to reflecting boundary conditions in \eqref{itosdebar} (cf. \cite[p.222, Theorem 7.1]{IW89}).
The main purpose of this subsection was to elucidate the emergence of the It\^o-correction term $\frac{1}{\beta}  \partial_{\al'} g^{\al' \al}$.

\section{The Gr\"un--Rumpf metric}\label{grsection}
In \cite[Section 5]{GruenRumpf} the authors have introduced a discretization of the deterministic thin-film equation in a way such that a discrete version of the entropy estimate \eqref{entest} holds; see Lemma \eqref{entestgr}. They propose a finite element discretization and in particular introduce a specific discretization of the mobility. As it turns out, the latter can be interpreted as a mixed finite element discretization with lumping of the metric tensor \eqref{mettens}; see Definition \eqref{grmetric}. At the same time, the authors of \cite{Bertozzi} have considered a finite difference discretization of \eqref{tfe} with a similar discretization of the mobility as \cite{GruenRumpf} that preserves the entropy estimate.
% This corresponds to a spatial discretization of the gradient flow structure. 
% More precisely, our starting point is an energetically conformal finite-elements discretization. 
% Hence the function
% \begin{align}
%     \pi^N : P_1 \to \R^N, f \mapsto \bra*{f(x_1), \dots, f(x_N)}^T
% \end{align}
% is bijective and we will freely identify any $f \in P_1$ with its associated vector $\pi^N(f) \in \R^N$.
%This gives rise to the following definition.
%\begin{defn}
%    For $i \in \set*{1, \dots, N}$ we denote by $\varphi_i : \R \to \R$ the unique piecewise linear and continuous function that satisfies
%    \begin{align}
%        \varphi_i(x_j) = \delta_{ij}
%    \end{align}
%    for all $j \in \set*{1, \dots, N}$. We will often refer to the functions $\set*{\varphi_i}_{i = 1, \dots, N}$ as the hat functions or hat basis.
%\end{defn}

\medskip

As has been discussed in the last section, the space $\mathcal{M}_N$ is the configuration space for the discretized stochastic thin-film equation. Any function in $P_1$, and thus also any $h \in \mathcal{M}_N$, is uniquely determined by its values at the nodal points $\set*{x_i}_{i = 1, \dots, N}$. 
%If we denote as usual for $i = 1, \dots, N$ by $e_i$ the $i$-th unit vector of $\R^N$, then we define $\hat{\varphi}^i$ as the associated element in %$P_1$, i.e. $\pi^N(\hat{\varphi}^i) = e_i$. 
This gives rise to a natural chart
\begin{align}
    \bra*{\varphi^i}_i : \mathcal{M}_N \to \Delta_N
\end{align}
where for $h \in \mathcal{M}_N$ we have
\begin{align}\label{hatbasis}
    h = \varphi^i(h) \hat{\varphi}_i.
\end{align}
Here $\Delta _N$ is the $N$-simplex defined as
\begin{align}
    \Delta_N := \set*{h \in \R^N : h^i > 0, \frac{1}{N} \sum_{i = 1}^N h^i = 1}
\end{align}
and for $i = 1, \dots, N$ we define $\hat{\varphi}_i$ to be the unique piecewise linear and continuous function such that we have
\begin{align}
    \hat{\varphi}_i(x_j) = \delta_{ij}.
\end{align}
The family $\bra*{\hat{\varphi}_i}_{i = 1, \dots, N}$ is of course known as the hat basis in finite elements.
\begin{figure}[h!]
    \begin{tikzpicture}

    \coordinate[label = below:$x_{i-1}$] (A) at (0,0);

    \node at (A)[circle,fill,inner sep=1pt]{};
    
    \coordinate[label = below:$x_{i}$] (B) at (1,0);

    \node at (B)[circle,fill,inner sep=1pt]{};
    
    \coordinate[label = below:$x_{i+1}$] (C) at (2,0);

    \node at (C)[circle,fill,inner sep=1pt]{};

    \coordinate (F) at (2.5,0);
    
    \coordinate (E) at (-0.5,0);
    
    \draw (E) -- (B) -- (F);
    
    \coordinate [label = above:$1$] (D) at (1,1);
    
    \node at (D)[circle,fill,inner sep=1pt]{};
    
    \draw[color=blue] (A) -- (D) -- (C);

    \coordinate[label = below:$\hat{\varphi}_i$] (G) at (0,1);
    \end{tikzpicture}
    \caption{An element $\hat{\varphi}^i$ of the hat basis.}
\end{figure}
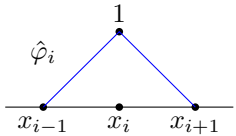
As in Section 4.1 we will write $h^i$ instead of $\varphi^i(h)$. We denote by $P_0$ the space of piecewise constant functions and we note that $T_h \mathcal{M}_N := T_h \mathcal{M} \cap P_1$. 

\medskip

We now turn to the discretization of \eqref{mettens}; in suitable coordinates it amounts to a reinterpretation of the metric considered in Section 5 of \cite{GruenRumpf}, see \cref{recover}. 
% We will give the definition first in the non-periodic case where the discretization scheme is the same except that the partition $\set*{x_i}_{i \in \Z}$ is countably infinite
For the discretization of \eqref{mettens}, following the strategy of first discretizing and then periodizing leads to a simpler result, and we shall follow it here. Hence \cref{grmetric} is phrased with the unit torus replaced by $\R$\footnote{with the abuse of keeping the notation $\mathcal{M}_N$}.
\begin{defn}[\emph{Gr\"un-Rumpf} metric]\label{grmetric}
    Let $h \in \mathcal{M}_N$ and $\dot{h} \in T_h \mathcal{M}_N$. We define a metric tensor on $T_h \mathcal{M}_N \otimes T_h \mathcal{M}_N$ via
    \begin{align}
        &\grm_h(\dot{h}, \dot{h}) := \\
        &\inf_{j} \set*{\int_{\R} \frac{j^2}{M(h)} \dx{x}: j \in P_0, \int_{\R}j \partial_x \zeta \dx{x} = \frac{1}{N} \sum_{i \in \Z} \dot{h}^i \zeta^i \ \, \forall \zeta \in P_1 \ \mathrm{compactly \ supported}}.
    \end{align}
\end{defn}
\begin{rem}
    As mentioned earlier \eqref{grmetric} is a mixed finite element discretization with lumping of \eqref{mettens}. By a mixed discretization, we mean that we are not just discretizing the configuration space but also the space of fluxes, i.e. we require $j \in P_0$. Moreover, lumping means that instead of the $L^2$-inner product $\int_{\R} \dot{h} \zeta \dx{x}$ we use the $\ell^2$-inner product $\frac{1}{N} \sum_{i \in \Z} \dot{h}^i \zeta^i$. 
\end{rem}
 \begin{figure}[h!]
     \begin{tikzpicture}

     \coordinate[label = below:{$x_{\al-} =  x_{i}$}] (A) at (0,0);

    \node at (A)[circle,fill,inner sep=1pt]{};
    
     \coordinate[label = below:{$x_{\al+} = x_{i+1}$}] (B) at (3,0);

     \node at (B)[circle,fill,inner sep=1pt]{};
    
    \draw[color=black] (A) -- (B);

    \coordinate[label = below:$I_{\al}$] (G) at (1.5,0.75);
     \end{tikzpicture}
    \caption{Relation of the intervals $\bra*{I_{\al}}_{\al}$ and the nodal points $\set*{x_i}_{i}$.}\label{fig:ali}
\end{figure}
Now we come to the choice of coordinates. In order to obtain a simpler expression of the metric tensor it is better to introduce another basis than the hat basis.
For any $\al$ let $\bar{\varphi}_{\al} \in P_1$ be given by (see Figure \eqref{fig:ali})
\begin{align}\label{zigzag}
    \bar{\varphi}_{\al} := N^\frac{3}{2} \bra*{ \hat{\varphi}_{\al+} - \hat{\varphi}_{\al-}}.
\end{align}
We call the family $\bra*{\bar{\varphi}_{\al}}_{\al = 1, \dots, N}$ the \emph{zigzag} basis \footnote{As is easily seen it holds that $\sum_{\al=1}^N \bar{\varphi}_{\al} = 0$ and thus the \emph{zigzag} basis is not really a basis. This issue is resolved by requiring that $\sum_{\al = 1}^N h^{\al} = 0$. }. Now we can introduce another set of coordinates given by the chart
\begin{align}
    \bra*{\varphi^{\al}}_{\al} : \mathcal{M}_N \to \R^N
\end{align}
where\footnote{The image of $\bra*{\varphi^{\al}}_{\al}$ is also affine linear.}
\begin{align}\label{barbasis}
    h = \varphi^{\al}(h) \bar{\varphi}_{\alpha} + 1.
\end{align}
Here $1$ denotes the constant function with that value.
Again for simplicity, instead of writing $\varphi^{\al}(h)$ we write $h^{\al}$.

% \begin{rem}
%     The choice of amplitude \eqref{zigzag} is reminiscent of that of a wavelet basis.
% \end{rem}
We note that by the relation \eqref{barbasis}, for every $h \in \mathcal{M}_N$ the induced basis on $T_h \mathcal{M}_N$ is given by the \emph{zigzag} basis.
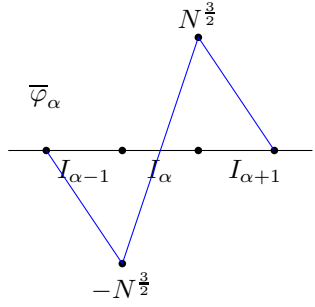
\begin{figure}[h!]
    \begin{tikzpicture}

    \coordinate (A) at (0,0);

    \node at (A)[circle,fill,inner sep=1pt]{};
    
    \coordinate (B) at (1,0);

    \node at (B)[circle,fill,inner sep=1pt]{};
    
    \coordinate (C) at (2,0);

    \node at (C)[circle,fill,inner sep=1pt]{};
    
    \coordinate (D) at (3,0);
    
    \node at (D)[circle,fill,inner sep=1pt]{};
    
    \coordinate (F) at (3.5,0);
    
    \coordinate (E) at (-0.5,0);
    
    \draw (E) -- (A)--node[below]{$I_{\al-1}$} (B) -- node[below]{$I_\al$} (C) -- node[below]{$I_{\al+1}$}(F);
    
    \coordinate [label = below:$-N^\frac{3}{2} $] (G) at (1,-1.5);
    
    \node at (G)[circle,fill,inner sep=1pt]{};
    
    \coordinate [label = above:$N^\frac{3}{2} $] (H) at (2,1.5);
    
    \node at (H)[circle,fill,inner sep=1pt]{};
    
    \draw[color=blue] (A) -- (G) -- (H) -- (D);

    \coordinate[label = below:$\bar{\varphi}_{\al}$] (G) at (0,1);
    \end{tikzpicture}
    \caption{An element $\bar{\varphi}_{\al}$ of the \emph{zigzag basis}.}
\end{figure}
Hence in these coordinates, the metric tensor \eqref{grmetric} takes the form
\begin{align}
    \grm_{\al \al'}(h) := \grm_h\bra*{\bar{\varphi}_{\al}, \bar{\varphi}_{\al'}}.
\end{align}
Note that for any $\al$ we have
\begin{align}\label{lhsgr}
    \frac{1}{N} \sum_{i \in \Z} \bra*{\bar{\varphi}_{\al}}^i \zeta^i = \sqrt{N} \bra*{\zeta^{\al +} - \zeta^{\al -}}.
\end{align}
Similarly, we compute
\begin{align}\label{rhsgr}
    \int_{\R} j \partial_x \zeta \dx{x} = \sum_{\al' \in \Z} j_{\al'} \bra*{\zeta^{\al'+} - \zeta^{\al' -}}
\end{align}
where $j = \sum_{\al \in \Z} j_{\al} \mathbbm{1}_{I_{\al}}$.
Hence we see that an admissible choice is $j = \sqrt{N} \mathbbm{1}_{I_{\al}}$ and since any other choice only differs by an additive constant this is already the optimal choice
and this yields
\begin{align}
    \grm_{\al \al'}(h) = \dashint_{I_{\al}} \frac{1}{M(h)} \dx{x} \ \delta_{\al \al'}.
\end{align}
%for all $\al, \al' \in \set*{1, \dots, N}$.
As in Section 4.1 we denote the dual metric associated to \eqref{grmetric} by $\bra*{\grm^{\al \al'}(h)}_{\al, \al'}$  and, since $\grm^{\al \al''}(h)\grm_{\al'' \al'}(h) = \delta^{\al}_{\al'}$ (see \eqref{ident}) we have
\begin{align}\label{grminverse}
    \grm^{\al \al'}(h) = \bra*{\dashint_{I_{\al}} \frac{1}{M(h)} \dx{x} }^{-1} \delta^{\al \al'}.
\end{align}
%for all $\al, \al' = 1, \dots, N$. 
Having derived this discretization, we again impose a periodic data structure on the discrete level.
\begin{rem}\label{recover}
    On every interval $I_{\al}$ the expression \eqref{grminverse} is the harmonic mean of the mobility $M(h)$ and thus we recover the discretization proposed in \cite[Section 5]{GruenRumpf}. 
\end{rem}
As mentioned in the last section, the discretization of the energy is just the restriction of $E$ to the space $P_1$. 
%For this choice of discretization we choose the (inverse of the) chart
%\begin{align}
%  \tilde{\varphi}^{-1} :  h \mapsto A^T h + \mathbbm{1}.
%\end{align}
Then according to \eqref{itosdebar} this specific discretization gives rise to the following SDE 
\begin{align}\label{grsdebar}
    \dx h^{\al}_t = \bra*{ - \grm^{\al \al'}(h_t) \partial_{\al'}E(h_t) + \frac{1}{\beta} \partial_{\al'} \grm^{\al' \al}(h_t) } \dx t + \sigma^{\al}_{\al'}(h_t) \sqrt{\frac{2}{\beta}} \dx W^{\al'}_t.
\end{align}
\begin{defn}\label{derA}
    Restricting the derivative $\partial_x$ to $P_1$ yields a linear operator $\partial_x : P_1 \to P_0$. We denote the matrix representation of this linear operator with respect to the hat basis on $P_1$ and the basis $\bra*{\mathbbm{1}_{I_{\al}}}_{\al}$ on $P_0$ by $A = \bra*{A^{\al}_i}^{\al}_i$, i.e. we have
    \begin{align}\label{forwarddif}
        A^{\al}_i b^i = N\bra*{b^{\al+} - b^{\al-}}
    \end{align}
    for all vectors $\bra*{b^i}^i$.
    Moreover, its transpose is given by
    \begin{align}
        \bra*{A^T}^{i}_{\al} = A^{\al}_i
        %\delta^{i l} A^{\al'}_{l} \delta_{\al' \al}. 
    \end{align}
\end{defn}
Now we pass from $\al$-coordinates to $i$-coordinates. To this end, we compute
\begin{align}
    h \stackrel{\eqref{barbasis}}{=} h^{\al} \bar{\varphi}_{\al} + 1 \stackrel{\eqref{zigzag}}{=} N^\frac{3}{2} h^{\al} \bra*{\hat{\varphi}_{\al+} - \hat{\varphi}_{\al-}}  + 1 = \sqrt{N} h^{\al} \bra*{A^T}^i_{\al} \hat{\varphi}_i  + 1.
\end{align}
Thus by \eqref{hatbasis} we obtain the formula
\begin{align}\label{bartohat}
    h^i = \sqrt{N} \at^i_{\al} h^{\al} + 1.
\end{align}
%for all $i \in \set*{1, \dots, N}$. 
Then \eqref{parder} and the chain rule yield
\begin{align}\label{hattobar}
    \partial_{\al} = \sqrt{N} \at^i_{\al} \partial_i.
\end{align}
%for $\al \in \set*{1, \dots, N}$. 
% We compute that
% \begin{align}
%     \partial_{\al'} E_N(h_t) = A_{\al' j} \partial_j E(h_t) = A_{\al' j}A^T_{j \al''} A_{\al'' k} h^k_t
% \end{align}
Hence, by applying \eqref{bartohat} to \eqref{grsdebar} and \eqref{hattobar} only to the first drift term, we end up with the following SDE in $i$-coordinates
% \[ N \partial_j \at^i_{\al}  g^{\al \al'}(h_t)  \at^{j}_{\al'}
% \left\{
\begin{align}\label{finaldstfe}
    \dx h^i_{t} =& \bra*{-N \at^i_{\al} \grm^{\al \al'}(h_t) \at^j_{\al'} \partial_j E(h_t) + \frac{\sqrt{N}}{\beta}  \at^{i}_{\al} \partial_{\al'} g^{\al' \al}(h_t) }\dx t \\
    & + \at^i_{\al}\sigma^{\al}_{\al'}(h_t) \sqrt{ \frac{2N}{\beta}} \dx W^{\al'}_t  
\end{align}
% \right.
% \]
subject to reflecting boundary conditions.
It is easy to see that the It\^o-correction term in the discrete thin-film equation with thermal noise \eqref{finaldstfe} does in general not vanish, see \eqref{itocorcomp} for the case $M(h) = h^3$. 

\section{Positivity of the scheme}\label{sec:pos}
As it turns out, the \emph{Gr\"un--Rumpf} metric is the right discretization in order to preserve positivity. From now on it will be important that the energy functional is the Dirichlet energy \eqref{energy}. In view of \cref{derA}, the restriction of $E$ to $\mathcal{M}_N$ assumes the form
\begin{align}
    E(h) = \frac{1}{2N} \sum_{\al = 1}^N \bra*{\bra*{Ah}^{\al}}^2 =\frac{1}{2N} h^{j} A^{\al}_j \delta_{\al \al'} A^{\al'}_{k} h^{k}
\end{align}
and hence
\begin{align}\label{gradE}
    \partial_j E(h) = \frac{1}{N} A^{\al}_j \delta_{\al \al'} A^{\al'}_k h^k.
\end{align}
Plugging this in the first drift term of \eqref{finaldstfe} yields
\begin{align}
    -\at^i_{\al} g^{\al \al'}(h_t) \at^j_{\al'} A^{\al''}_j \delta_{\al'' \al'''} A^{\al'''}_k h^k_t.
\end{align}
Instead of viewing $\partial_j E$ as a covector it makes sense to regard it as a vector. 
To this end, we contract the metric $g^{\al \al'}$ with respect to the ambient Euclidean metric, i.e.
\begin{align}
    g^{\al \al'} = g^{\al}_{\gamma} \delta^{\gamma \al'}
\end{align}
and this yields
\begin{align}\label{compdrift}
g^{\al \al'} \at^j_{\al'} A^{\al''}_j \delta_{\al'' \al'''} A^{\al'''}_k h^k_t 
= g^{\al}_{\al'} A^{\al'}_j \at^j_{\al''} A^{\al''}_k h^k_t 
= g^{\al}_{\al'} \bra*{A A^T A h_t}^{\al'}_k.
    % &-\at^i_{\al} g^{\al}_{\gamma}(h_t) \delta^{\gamma \al'} \at^j_{\al'} A^{\al''}_j \delta_{\al'' \al'''} A^{\al'''}_k h^k_t \\
    % &= -\at^i_{\al} g^{\al}_{\gamma}(h_t) \delta^{\gamma \al'} \at^l_{\al'} \delta_{l j} \delta^{j l'}  A^{\al''}_{l'} \delta_{\al'' \al'''} A^{\al'''}_k h^k_t \\
    % &= -\at^i_{\al} g^{\al}_{\al'}(h_t) A^{\al'}_j \at^j_{\al''} A^{\al''}_k h^k_t.
\end{align}
Furthermore we specify $\sigma^{\al}_{\al'}(h)$ to be the square-root of $g^{\al}_{\al'}(h)$ and from now on we will write $\sqrt{g}^{\al}_{\al'} := \sigma^{\al}_{\al'}$. 
% Moreover, concerning the It\^o-correction term, we write
% \begin{align}
%     \partial_j \at^i_{\al}  g^{\al \al'}(h_t)  \at^{j}_{\al'} &= \partial_l \delta^{lk} \at^i_{\al} g^{\al}_{\gamma}(h_t) \delta^{\gamma \al'} \at^{j}_{\al'} \delta_{jk} \\
%     &= \partial_l \delta^{l k} \at^i_{\al} g^{\al}_{\al'}(h_t) A^{\al'}_k.
% \end{align}
% where $\xi$ denotes Gaussian space-time white noise. Moreover, we will introduce the divergence in $i$-coordinates and set
% \begin{align}
%     {D \cdot \Sigma}^{i} := \sum_{j = 1}^N \partial_{j} \Sigma^{i}_j
% \end{align}
% for some matrix field $\Sigma = \bra*{\Sigma^{i}_j}^{i}_j$.
%N \partial_l \delta^{l k} \at^i_{\al} g^{\al}_{\al'}(h_t) A^{\al'}_k

\medskip

Combining \eqref{finaldstfe} and \eqref{compdrift} we end up with the following SDE
\begin{align}\label{finaldstfe2}
    \dx h^i_{t} =& \bra*{-\at^i_{\al} g^{\al}_{\al'}(h_t) A^{\al'}_j \at^j_{\al''} A^{\al''}_k h^k_t + \frac{\sqrt{N}}{\beta}  \at^i_{\al} \partial_{\al'} g^{\al' \al}(h_t)}  \dx t \\
    & + \at^i_{\al} \sqrt{g}^{\al}_{\al'}(h_t) \sqrt{ \frac{2N}{\beta} } \dx W^{\al'}_t.
\end{align}
Introducing the abbreviations $G^{-1}(h) := \bra*{g^{\al}_{\al'}(h)}^{\al}_{\al'}$ and $\sqrt{G}^{-1}(h) := \bra*{\sqrt{g}^{\al}_{\al'}(h)}^{\al}_{\al'}$ as well as the (rescaled) divergence-operator in $\al$-coordinates  
\begin{align}
    \bra*{\bar{D} \cdot \Sigma}^{\al} := \frac{1}{\sqrt{N}} \partial_{\al'} \Sigma^{\al' \al}
\end{align}
for some matrix field $\Sigma = \bra*{\Sigma^{\al' \al}}^{\al' \al}$
%N D\cdot \bra*{ A^T g^{-1}(h_t) A}
we see that \eqref{finaldstfe2} can be written in matrix form as
\begin{align}\label{finaldstfematrix}
    \dx h_t &= \bra*{-A^T G^{-1}(h_t) A A^T A h_t + \frac{N}{\beta} A^T \bar{D} \cdot G^{-1}(h_t) } \dx{t} 
    + A^T \sqrt{G}^{-1}(h_t) \sqrt{\frac{2N}{\beta}} \dx W_t.
\end{align}
The following table provides the connection to the continuum case:
\begin{center}
	\begin{tabular}{c |  c } 
		
		discrete & continuum \\ [0.5ex] 
		\midrule
		
		$G^{-1}(h)$ & $M(h)$, see \eqref{grminverse}           \\ [1 ex]
		
		$\sqrt{G}^{-1}(h)$ &$ \sqrt{M(h)}$  \\ [1 ex]
		
		$A$ & $\partial_x$, see \cref{derA}  \\ [1 ex]
		
		$A^T$ & $ -\partial_x $ \\ [1 ex]
		
		$\sqrt{N} \frac{dW_t}{dt}$ & $ \xi$. \\  [1 ex]
		
	\end{tabular}
\end{center}
For the last claim let $f_1(t), \dots, f_N(t)$ be compactly supported. A quick computation shows that
\begin{align}
    \mathbb{E}\pra*{\bra*{\int_0^{\infty} \frac{1}{N} f_{\al}(t) \sqrt{N} \frac{\dx W^{\al}}{\dx{t}} \dx{t}}^2} = \frac{1}{N} \mathbb{E}\pra*{\bra*{\int_0^{\infty} \frac{\dx}{\dx{t}} f_{\al}(t) W^{\al}_t \dx{t}}^2} = \frac{1}{N} \sum_{i = 1}^N \int_0^{\infty} f^2_i(t) \dx{t}.
\end{align}
Thus we obtain the following continuum analogs of \eqref{finaldstfematrix}:
\begin{center}\label{table}
	\begin{tabular}{c |  c } 
		
		discrete & continuum \\ [0.5ex] 
		\midrule
		$A^T G^{-1}(h) A A^T A h$ & $\partial_x \bra*{M(h) \partial_x^3 h}$   \\ [1 ex]
		
		%$\frac{1}{\beta} \sqrt{N} A^T \bar{D}\cdot g^{-1}(h)$ &  \\ [1 ex]
		
		$A^T \sqrt{G}^{-1}(h) \sqrt{\frac{2N}{\beta}} \frac{\dx W_t}{\dx{t}}$ & $\partial_x\bra*{\sqrt{M(h)} \sqrt{\frac{2}{\beta}} \xi}$.  \\ [1 ex]
		
	\end{tabular}
\end{center}
This confirms the form \eqref{stfe} of the SPDE. We will comment on the continuum form of the It\^o correction term~\eqref{finaldstfematrix} in~\eqref{itocorrection}.

\medskip

We will now turn our discussion to the entropy $S$. Recall that $s$ is chosen such that $s'' = \frac{1}{M}$. For $h \in \Delta_N$ we write $s(h) := \bra*{s(h^i)}^{i}$.
The choice of the metric tensor \eqref{grmetric} is based on the fact that it satisfies the crucial identity  
\begin{align}\label{gridentity}
    %g^{\al}_{\al'}(h) A^{\al'}_i G'(h^i) = A^{\al}_i h^i
    G^{-1}(h) A s'(h) = Ah
\end{align}
which is the discrete analog of
\begin{align}
    M(h) \partial_x s'(h) = \partial_x h.
\end{align}
By formally letting $\beta \to \infty$ in \eqref{finaldstfematrix}, we recover the \emph{Gr\"un--Rumpf} discretization of the deterministic thin-film equation 
\begin{align}\label{dtfegr}
        \frac{\dx}{\dx{t}} h_{t} = - A^T G^{-1}(h_t)AA^TAh_t.
\end{align}
In \cite{GruenRumpf} the authors have used the identity \eqref{gridentity} to show the following entropy estimate
\begin{prop}\cite[p.129, Lemma 5.1]{GruenRumpf}\label{entestgr}
    Let $h_t$ be a solution to \eqref{dtfegr}. We define the discrete entropy\footnote{Notice that the discrete entropy is the lumped version of \eqref{entropy}} via
    \begin{align}
        S(h) := \frac{1}{N} \sum_{i = 1}^N s(h^i)
    \end{align}
    where $s$ is chosen such that $s'' = \frac{1}{M}$. Then we have the identity
    \begin{align}
        \frac{\dx}{\dx{t}} S(h_t) = - \frac{1}{N} \sum_{i = 1}^N \bra*{\bra*{A^T A h_t}^i}^2.
        %= - \frac{1}{N} \sum_{i = 1}^N \bra*{\at^i_{\al} A^{\al}_j h^j}^2  \leq 0.
    \end{align}
\end{prop}
% For some $1 \leq m < \infty$ we make the following assumption on the mobility $M(h)$:
Recall that we are particularly interested in the case $M(h) = h^3$.
\begin{assumption}
    We assume that for some $0 \leq m < \infty$ we have
    \begin{align}
         L := \sup_{h \in \bra*{0, \infty}} \frac{M(h)}{h^m} < \infty.
         \label{mobilitymeq}
    \end{align}
    % \quad
    % \begin{enumerate} 

    %     \item $0 < \lim_{h \to 0} \frac{M'(h)}{h^{m-1}} < \infty$
    % \end{enumerate}
    \label{mobilitym}
\end{assumption}
From now on, for $m \geq 2$ we specifically set
\begin{align}
    s(h) := \int^{\infty}_h \int^{\infty}_{h'} \frac{1}{M(h'')} \dx h''.
\end{align}
Using Proposition \eqref{entestgr} it is easy to see that if the mobility satisfies assumption \eqref{mobilitym} the deterministic scheme preserves positivity for $m \geq 2$. 
\begin{rem}\label{distancetoboundary}
    In terms of the configuration space positivity means that the flow $h_t$ does not touch the boundary of the manifold $\mathcal{M}_N$ but stays in the open orthant $\set*{h > 0}$. In fact, one can show that the distance (induced by the metric tensor \eqref{grmetric}) between the boundary and the interior of $\mathcal{M}_N$ is finite if and only if $m < 3$, see \eqref{Nequaltwo} for the case of $N= 2$. Hence, by energy dissipation, any gradient flow with respect to the metric tensor \eqref{grmetric} preserves positivity for $m \geq 3$. In case of the Dirichlet energy as the energy functional the entropy estimate \eqref{entestgr} upgrades this threshold to $m \geq 2$. 
    
    On the other hand, it can be seen that the restriction of the metric tensor \eqref{mettens} to $T\mathcal{M}_N \otimes T\mathcal{M}_N$ induces a distance that is finite to the boundary iff $m < 5$.
\end{rem}
The main result in this paper transfers the entropy estimate \eqref{entestgr} to the stochastic setting.
\begin{thm}\label{stentest}
    Let $h_t$ be a solution to \eqref{finaldstfe} such that the initial condition $h_0$ satisfies $\mathbb{E}\pra*{S(h_0)} < \infty$ and the mobility $M$ satisfies \cref{mobilitym} for $m \geq 3$, then the following identity holds
        \begin{align}\label{enteq}
        \mathbb{E}\pra*{S(h_t)} + \int_{0}^t \mathbb{E}\pra*{\frac{1}{N} \sum_{i=1}^N \bra*{ \bra*{A^T A h_r}^i}^2} \dx{r} = \mathbb{E}\pra*{S(h_0)} + \frac{2N^3}{\beta} t.
    \end{align}
    % , then we have for all $t \geq 0$
    % \begin{align}\label{sentdiss}
    %     \mathbb{E}\pra*{S(h_t)} = \mathbb{E}\pra*{S(h_0)} - \int_{0}^t \mathbb{E}\pra*{\frac{1}{N} \sum_{i=1}^N \bra*{ \bra*{A^T A h_r}^i}^2} \dx{r} + \frac{2N^3}{\beta} t.
    % \end{align}
    % If, moreover, for $p \geq 1$ we have that $\mathbb{E}\pra*{\bra*{S(h_0)}^p}^{\frac{1}{p}} < \infty$ 
    Let $T > 0$. If, moreover, for $p < \infty$ we have that $\mathbb{E}\pra*{S^p(h_0)} < \infty$, then 
    \begin{align}\label{sentest2}
    \mathbb{E}\pra*{\bra*{\sup_{0 \leq r \leq T} S(h_r)}^p}^{\frac{1}{p}} \leq 
        \begin{cases}
             C \bra*{\bra*{\mathbb{E}\pra*{S^p(h_0)}^\frac{1}{p} + \frac{N^3 T}{\beta}}^{\frac{m-3}{m-2}} + \frac{N^{3 + \frac{1}{m-2}}T}{\beta}}^{\frac{m-2}{m-3}}\ &\mathrm{for} \ m > 3 \\
             C \bra*{ \mathbb{E}\pra*{S^p(h_0)}^{\frac{1}{p}}+1} e^{C \frac{N^4 T}{\beta} } \ &\mathrm{for} \ m =3
        \end{cases}
    \end{align}
    for some constant $C$ only depending on $p$, $m$ and $L$.
\end{thm}
\begin{proof}
    % and $\partial_{ij}Ent(h) = \frac{1}{N} G''(h^i) \delta_{ij}$.
    By assumption, the process $h_t$ satisfies the SDE
    \begin{align}
        \dx h_t = b(h_t) \dx t + \sigma(h_t) \dx W_t
    \end{align}
    where the drift $b$ and the diffusion matrix $\sigma$ are given according to \eqref{finaldstfematrix}. We set $\mathcal{M}_N^R := \set*{h \in \mathcal{M}_N : S(h) \leq R}$ for some $R$. Notice that thanks to $m \geq 2$ it holds that $h \in \mathcal{M}_N^R$ implies that $h$ is strictly bounded away from $0$. It is clear that there exist Lipschitz extensions $\bar{b}$ of $b$ and $\bar{\sigma}$ of $\sigma$ to all of $\R^{N+1}$ such that 
    \begin{align}\label{ext}
          \bar{b}|_{\mathcal{M}_N^R} = b|_{\mathcal{M}_N^R} \ \mathrm{and} \ \bar{\sigma}|_{\mathcal{M}_N^R} = \sigma|_{\mathcal{M}_N^R}
    \end{align}
    as well as a smooth extension $\bar{S}$ of the entropy $S$
    such that
    \begin{align}\label{ext2}
        \bar{S}|_{\mathcal{M}_N^R} = S|_{\mathcal{M}_N^R}.
    \end{align}
    Then we consider the process
    \begin{align}
        \dx \bar{h}_t = \bar{b}(\bar{h}_t) \dx t + \bar{\sigma}(\bar{h}_t) \dx W_t.
    \end{align}
    We apply It\^o's formula (cf. \cite[p.222, Theorem 3.3]{yorrevuz} and \cite[p.67, Lemma 3.2]{bookPavliotis}) to $\bar{S}(\bar{h}_t)$ which yields 
     \begin{align}\label{baritoform}
        \bar S(\bar{h}_{t}) = \bar S(h_0) + \int_0^{t }  \bar{\mathcal{L}} \, \bar S(\bar{h}_{s}) \dx{s} 
        + \sqrt{\frac{2N}{\beta}} \int_0^{t} \partial_i \bar S(\bar{h}_{s}) \bar{\sigma}^{i}_{\al}(\bar{h}_s) \dx W^{\al}_s
    \end{align}
    where $\bar{\mathcal{L}}$ denotes the generator of the process $\bar{h}_t$.
    Moreover, we define the stopping time
    \begin{align}
        \tau_R := \inf \set*{t \geq 0 : S(h_t) > R}.
    \end{align}
    %By \eqref{mobilitymeq} for $m \geq 2$ the condition $S(h_t) < R$ implies that $h_t$ is bounded away from the boundary $\partial \set*{h > 0}$.
    By definition, we have that $\bar{h}_t = h_t$ for $t \leq \tau_R$ and thus by \eqref{baritoform}, \eqref{ext} and \eqref{ext2} we get
    \begin{align}\label{itoform}
        S(h_{t \wedge \tau_R}) &= S(h_0) + \int_0^{t \wedge \tau_R}  \mathcal{L} S(h_{s}) \dx{s} \\
        &\quad+ \sqrt{\frac{2N}{\beta}} \int_0^{t \wedge \tau_R} \partial_i S(h_{s}) \bra*{A^T \sqrt{G}^{-1}(h_{s})}^{i}_{\al} \dx W^{\al}_s.
        % &= -\partial_i Ent(h_t) \bra*{A^T g(h_t) A A^T A h_t}^i \dx{t} \\
        % &+ \frac{1}{\beta} N\partial_i Ent(h_t) \bra*{ D \cdot \bra*{A^T g^{-1}(h_t) A}}^i \dx{t}  \\
        % &+\frac{1}{\beta} N \bra*{\sqrt{g}^{-1}(h_t) A D^2 Ent(h_t) A^T \sqrt{g}^{-1}(h_t)}^i_i  \dx{t} \\
        % &+ \sqrt{2\frac{1}{\beta} N} \partial_i Ent(h_t) \bra*{A^T \sqrt{g}^{-1}(h_t) \dx W_t}^i \\
        % &= \set*{I_1(h_t) + I_2(h_t) + I_3(h_t)} \dx t + I_4(h_t) \dx W_t.
        % &= - \frac{1}{N} \sum_{i = 1}^N G'(h^i_t) \bra*{A^T g^{-1}(h_t) A A^T A h_t}^i \dx{t} \\
        % &+ \frac{1}{\sqrt{N}} \frac{1}{\beta} \sum_{i=1}^N G'(h^i_t) \bra*{A^T \bar{D} \cdot g^{-1}(h_t)}^i \dx{t} \\
        % &+ \frac{1}{\beta} \sum_{i = 1}^N G''(h^i_t) \bra*{A^T g^{-1}(h_t) A}^i_i \dx{t} \\
        % &+ \sqrt{2\frac{1}{\beta} N^{-1}} \sum_{i = 1}^N G'(h^i_t) \bra*{A^T \sqrt{g}^{-1}(h_t)}^{i}_{\al'} \dx W^{\al'}_t \\
        % &= -\frac{1}{N} G'(h^l_t) \delta_{li} \at^i_{\al} g^{\al}_{\al'}(h_t) A^{\al'}_{j} \at^j_{\al''} A^{\al''}_{k} h^k_t \dx{t} \\ 
        % &+ \frac{1}{\beta} \frac{1}{\sqrt{N}} G'(h^l_t) \delta_{li} \at^{i}_{\al} \partial_{\al'} g^{\al \al'}(h_t) \dx{t} \\
        % &+ \frac{1}{\beta} \sum_{i = 1}^N G''(h^i_t) \at^{i}_{\al} g^{\al}_{\al'}(h_t) A^{\al'}_{i} \dx{t} \\
        % &+ \sqrt{2 \frac{1}{\beta} N^{-1}} G'(h^j_t) \delta_{ji}  \at^i_{\al} \sigma^{\al}_{\al'}(h_t) \dx W_t^{\al'} \\
    \end{align}
    % which by definition amounts to
    % \begin{align}
    %     S(h_{t \wedge \tau_R}) = S(h_0) + \int_0^{t } \mathcal{L}S(h_{s \wedge \tau_R}) \dx{s} + \sqrt{\frac{2N}{\beta}} \int_0^{t } D S(h_{s \wedge \tau_R}) \cdot A^T \sqrt{G}^{-1}(h_{s \wedge \tau_R}) \dx W_s.
    % \end{align}
    Here $\mathcal{L}$ denotes the generator of \eqref{finaldstfematrix}; according to \eqref{bke}, which we postprocess by \eqref{hattobar}, we have for any sufficiently nice function $f$ 
    \begin{align}
        \mathcal{L}f = N \frac{1}{\beta} \partial_i \bra*{\bra*{A^T G^{-1} A}^{ij} \partial_j f} - N \partial_i f \bra*{A^T G^{-1} A}^{ij} \partial_j E.
    \end{align}
    Then, we compute using \eqref{gradE}
    \begin{align}\label{genent}
        \mathcal{L} S(h) 
        &\stackrel{\eqref{gridentity}}{=} \frac{1}{\beta} \partial_i \bra*{ A^T A h }^i - \frac{1}{N} \sum_{i = 1}^N \bra*{\bra*{A^T A h}^i}^2 \\
        &\stackrel{\eqref{derA}}{=} \frac{2N^3}{\beta}  - \frac{1}{N} \sum_{i = 1}^N \bra*{\bra*{A^T A h}^i}^2.
    \end{align}

    We now consider the martingale in \eqref{itoform}
    \begin{align}
        X_t := \sqrt{\frac{2}{\beta N}}  \int_{0}^{t \wedge \tau_R} \sum_{i=1}^N s'(h^i_{s \wedge \tau_R}) \bra*{A^T \sqrt{G}^{-1}(h_{s \wedge \tau_R})}^{i}_{\al} \dx W_s^{\al}  \, ,
    \end{align}
    and note that, for $T > 0$ and $p < \infty$, the Burkholder--Davis--Gundy inequality (cf. \cite[p.161, Corollary 4.2]{yorrevuz}) yields
    \begin{align}
        \mathbb{E}\pra*{\sup_{0 \leq s \leq t} |X_s|^p}^\frac{1}{p}  \lesssim_p \mathbb{E}\pra*{\langle X_t \rangle^{\frac{p}{2}}}^{\frac{1}{p}} \, ,
    \end{align}
    where
    \begin{align}
        \langle X_t \rangle =  \frac{2}{\beta N}  \int_{0}^{t \wedge \tau_R} \sum_{i=1}^N s'\bra*{h^i_{s \wedge \tau_R}}  \bra*{A^T G^{-1}(h_{s \wedge \tau_R}) A}^i_j s'\bra*{h^{j}_{s \wedge \tau_R}} \dx s
    \end{align}
    is the quadratic variation of $X$. Here and from now on $\lesssim$ is equivalent to $\leq C$ for some universal constant $C$ that only depends on $p, m, L$. The integrand can be rewritten as follows
    \begin{align}
        \sum_{i=1}^N s'\bra*{h^i_{s \wedge \tau_R}}  \bra*{A^T G^{-1}(h_{s \wedge \tau_R}) A}^i_j s'\bra*{h^{j}_{s \wedge \tau_R}} 
        &\stackrel{\eqref{gridentity}}{=}    \sum_{\al=1}^N \bra*{Ah_{s \wedge \tau_R}}^{\al} \bra*{A s'(h_{s \wedge \tau_R})}^{\al} \\ 
        &= \sum_{i=1}^N \bra*{A^T Ah_{s \wedge \tau_R}}^{i} s'(h^i_{s \wedge \tau_R}). 
        % &=  2 \frac{1}{\beta} N^{-1} \frac{1}{1-m} \sum_{i=1}^N \bra*{A^T Ah_t}^{i} \bra*{\bra*{h^i_t}^{2-m} }^{\frac{m-1}{m-2}} \\
        % &\leq 2 \frac{1}{\beta} N^{4} C(m) Ent(h_t)^{2}
        % &= N^2 \frac{1}{m - 1} \sum_{\al = 1}^N (h^{\al+}_{t} - h^{\al-}_{t})\bra*{\bra*{h^{\al-}_{t}}^{1-m} - \bra*{h^{\al+}_{t}}^{1-m}} \\
        % &\leq N^2 \frac{1}{m - 1} \sum_{\al = 1}^N \bra*{ h^{\al+}_{t} \bra*{h^{\al-}_{t}}^{1-m} + h^{\al-}_{t} \bra*{h^{\al+}_{t}}^{1-m} } \\
        % &= N^2 \frac{1}{m - 1} \sum_{\al = 1}^N \bra*{ h^{\al+}_{t} \bra*{h^{\al-}_{t}}^{m-3} \bra*{h^{\al-}_{t}}^{2(2-m)} + h^{\al-}_{t} \bra*{h^{\al+}_{t}}^{m-3} \bra*{h^{\al+}_{t}}^{2(2-m)}} \\
        % &\leq (m-2)^2(m-1)^2 (Ent(h_t))^2 N^4 \frac{1}{m - 1} \sum_{\al = 1}^N \bra*{ h^{\al+}_{t} \bra*{h^{\al-}_{t}}^{m-3} + h^{\al-}_{t} \bra*{h^{\al+}_{t}}^{m-3}} \\
        % &\leq 2 (m-2)^2(m-1)^2 (Ent(h_t))^2 N^{m + 1} \frac{1}{m - 1}  \sum_{\al = 1}^N h^{\al-}_t \\
        % &= 2  N^{m + 2} (m-2)^2 (m-1) (Ent(h_t))^2
    \end{align}
    We estimate the second term in the above expression as 
    \begin{align}
        s'(h^i_{s \wedge \tau_R}) &\stackrel{\eqref{mobilitymeq}}{\lesssim} \bra*{h^i_{s \wedge \tau_R}}^{1-m} 
        \lesssim \bra*{ \sum_{j = 1}^N \bra*{h^j_{s \wedge \tau_R}}^{2-m} }^{\frac{m-1}{m-2}} 
        \stackrel{\eqref{mobilitymeq}}{\lesssim} N^{\frac{m-1}{m-2}} S^{\frac{m-1}{m-2}}(h_{s \wedge \tau_R})
    \end{align}
    and hence by conservation of mass we arrive at
    \begin{align}
         \sum_{i=1}^N s'\bra*{h^i_{s \wedge \tau_R}}  \bra*{A^T G^{-1}(h_{s \wedge \tau_R}) A}^i_j s'\bra*{h^{j}_{s \wedge \tau_R}} 
        &\lesssim N^{\frac{m-1}{m-2}}  S^{\frac{m-1}{m-2}}(h_{s \wedge \tau_R}) \sum_{i=1}^N \abs*{\bra*{A^T Ah_{s \wedge \tau_R}}^{i}} \\
        &\lesssim N^{\frac{m-1}{m-2} + 3} S^{\frac{m-1}{m-2}}(h_{s \wedge \tau_R}).
    \end{align}
    Looking at \eqref{itoform} and collecting all the estimates yields
    \begin{align}
        &\mathbb{E}\pra*{\bra*{\sup_{0 \leq s \leq t } S(h_{s \wedge \tau_R})}^p}^{\frac{1}{p}} \\ 
        &\lesssim \mathbb{E}\pra*{S^p(h_0)}^{\frac{1}{p}} 
        + \frac{N^3 t}{\beta} + \sqrt{\frac{N^{3 + \frac{1}{m-2}}}{\beta}} \mathbb{E}\pra*{\bra*{\int^{t \wedge \tau_R}_0 S^{\frac{m-1}{m-2}}(h_{s \wedge \tau _R}) \dx s}^{\frac{p}{2}}}^{\frac{1}{p}}
        \\  
        &\leq \mathbb{E}\pra*{S^p(h_0)}^{\frac{1}{p}} 
        + \frac{N^3 t}{\beta} + \sqrt{\frac{N^{{3 + \frac{1}{m-2}}}}{\beta}} \mathbb{E}\pra*{ \bra*{\sup_{0 \leq s \leq t} S(h_{s \wedge \tau_R}) \int^t_0 \sup_{0 \leq r \leq s} S^\frac{1}{m-2}(h_{r \wedge \tau_R}) \dx s}^{\frac{p}{2}} }^{\frac{1}{p}}.
    \end{align}
    Then, we use Young's inequality to the effect that
    \begin{align}
        &\mathbb{E}\pra*{\bra*{\sup_{0 \leq s \leq t } S(h_{s \wedge \tau_R})}^p}^{\frac{1}{p}} \\
        &\lesssim \mathbb{E}\pra*{S^p(h_0)}^{\frac{1}{p}} 
        + \frac{N^3 t}{\beta} + \frac{N^{{3 + \frac{1}{m-2}}}}{\beta} \mathbb{E}\pra*{\bra*{\int_0^t \sup_{0 \leq r \leq s} S^{\frac{1}{m-2}}(h_{r \wedge \tau_R}) \dx s}^p}^{\frac{1}{p}} \, .
    \end{align}
    Finally, by using Minkowski's and Jensen's inequalities, 
    % for some $\eps > 0$
    % \begin{align}
    %     &\mathbb{E}\pra*{\langle X_T \rangle^{\frac{1}{2}}} \lesssim \sqrt{2 \frac{1}{\beta} N^{4} } \mathbb{E}\pra*{\bra*{\sup_{0 \leq t \leq T} Ent(h_t) \int_{0}^T Ent(h_t)^{\frac{1}{m-2}} \dx{t}}^{\frac{1}{2}}}  \\
    %     &\lesssim \sqrt{2 \frac{1}{\beta} N^{4}  } \bra*{ \frac{\eps}{\sqrt{2}} \mathbb{E}\pra*{\sup_{0 \leq t \leq T} Ent(h_t)} + \frac{1}{\sqrt{2}\eps} \mathbb{E}\pra*{\int_{0}^T Ent(h_t)^{\frac{1}{m-2}} \dx{t}} } \\
    %     &\lesssim \sqrt{2 \frac{1}{\beta} N^{4} }  \bra*{ \frac{\eps}{\sqrt{2}} \mathbb{E}\pra*{\sup_{0 \leq t \leq T} Ent(h_t)} + \frac{1}{\sqrt{2}\eps} \mathbb{E}\pra*{\int_{0}^T \sup_{0 \leq r \leq t} Ent(h_r)^{\frac{1}{m-2}} \dx{t}} }.
    % \end{align}
    % Hence 
    we are left with
    \begin{align}\label{integralest}
        &\mathbb{E}\pra*{\bra*{\sup_{0 \leq s \leq t } S(h_{s \wedge \tau_R})}^p}^{\frac{1}{p}} \\
        &\lesssim \mathbb{E}\pra*{S^p(h_0)}^{\frac{1}{p}} 
        + \frac{N^3 t}{\beta} + \frac{N^{3 + \frac{1}{m-2}}}{\beta}   \int_{0}^t \mathbb{E}\pra*{ \bra*{\sup_{0 \leq r \leq s} S(h_{r \wedge \tau_R})}^p}^{\frac{1}{p(m-2)}} \dx{s}.
    \end{align}
    % and for $m=3$
    % \begin{align}
    %     \mathbb{E}\pra*{\bra*{\sup_{0 \leq t \leq T} S(h_{t \wedge \tau_R})}^p}^{\frac{1}{p}} &\lesssim  \mathbb{E}\pra*{S^p(h_0)}^\frac{1}{p} \\
    %     &+ \frac{ N^3 T}{\beta} + \frac{N^{4}}{\beta}   \int_{0}^T  \mathbb{E}\pra*{\bra*{\sup_{0 \leq r \leq t} S(h_{r \wedge \tau_R})}^p }^{\frac{1}{p}} \dx{t}.
    % \end{align}
    % by applying the Gronwall inequality and Fubini's theorem finally yields 
    By \eqref{lemintineq}, for $m=3$ the integral inequality \eqref{integralest} 
    % \begin{align}
    %     &\mathbb{E}\pra*{\bra*{\sup_{0 \leq t \leq T } S(h_{t \wedge \tau_R})}^p}^{\frac{1}{p}}  \\
    %     &\lesssim \mathbb{E}\pra*{S^p(h_0)}^{\frac{1}{p}} + \frac{N^3 T}{\beta}
    %      + \frac{N^4}{\beta} \int_{0}^T \mathbb{E}\pra*{ \bra*{\sup_{0 \leq r \leq t} S(h_{r \wedge \tau_R})}^p}^{\frac{1}{p}}  \dx{t}.
    % \end{align}
    yields
    \begin{align}\label{bound3}
         \mathbb{E}\pra*{\bra*{\sup_{0 \leq t \leq T } S(h_{t \wedge \tau_R})}^p}^{\frac{1}{p}} \lesssim  \bra*{\mathbb{E}\pra*{S^p(h_0)}^\frac{1}{p}+1} e^{ C \frac{ N^4  T}{\beta}}
    \end{align}
    for some constant $C$ depending on $m$ and $L$.
    On the other hand for $m > 3$, by \eqref{lemintineq} we get
        \begin{align}\label{boundbig3}
         \mathbb{E}\pra*{\bra*{\sup_{0 \leq t \leq T } S(h_{t \wedge \tau_R})}^p}^{\frac{1}{p}} \lesssim  \bra*{\bra*{\mathbb{E}\pra*{S^p(h_0)}^\frac{1}{p} + \frac{N^3 T}{\beta}}^{\frac{m-3}{m-2}} + \frac{N^{3 + \frac{1}{m-2}}T}{\beta}}^{\frac{m-2}{m-3}}.
        % \mathbb{E}\pra*{\sup_{0 \leq r \leq T} Ent(h_r)} \leq 2 \mathbb{E}\pra*{Ent(h_0)} e^{4 C (m-2)^2 (m-1) \frac{1}{\beta} N^{m+1}T}.
    \end{align}

    We now argue that in the proof the stopping time was not necessary.
    By Chebyshev's inequality, we have
    \begin{align}
        \mathbb{E}\pra*{\sup_{0 \leq t \leq T \wedge \tau_R} S(h_t)} \geq R ~ \mathbb{P}\bra*{\tau_R \leq T} 
    \end{align}
    and thus by envoking \eqref{bound3} respectively \eqref{boundbig3} we get
    \begin{align}
        R ~ \mathbb{P}\bra*{\tau_R \leq T} \lesssim
        \begin{cases}
            \bra*{\mathbb{E}\pra*{S(h_0)} + 1} e^{ C \frac{ N^4  T}{\beta}} \ &\mathrm{for} \ m = 3 \\
            \bra*{\bra*{\mathbb{E}\pra*{S^p(h_0)}^\frac{1}{p} + \frac{N^3 T}{\beta}}^{\frac{m-3}{m-2}} + \frac{N^{3 + \frac{1}{m-2}}T}{\beta}}^{\frac{m-2}{m-3}} \ &\mathrm{for} \ m > 3.
        \end{cases}
    \end{align}
    % and thus for $m>3$ this yields
    % \begin{align}
    %     R ~ \mathbb{P}\bra*{\tau_R \leq T}  \stackrel{\eqref{boundbig3}}{\lesssim}\bra*{\bra*{\mathbb{E}\pra*{S^p(h_0)}^\frac{1}{p} + \frac{N^3 T}{\beta}}^{\frac{m-3}{m-2}} + \frac{N^{3 + \frac{1}{m-2}}T}{\beta}}^{\frac{m-2}{m-3}}
    % \end{align}
    % and for $m = 3$ it yields
    % \begin{align}
    %     R ~ \mathbb{P}\bra*{\tau_R \leq T} \stackrel{\eqref{bound3}}{\lesssim}\bra*{\mathbb{E}\pra*{S(h_0)} + 1} e^{ C \frac{ N^4  T}{\beta}}.
    % \end{align}
    Hence we have in either case
    \begin{align}\label{stoptime}
        \lim_{R \to \infty} \mathbb{P}\bra*{\tau_R \leq T} = 0
    \end{align}
    and this proves the second assertion using Fatou's lemma. Finally, taking expectations in \eqref{itoform} and using \eqref{genent} together with \eqref{stoptime} gives the first assertion.
\end{proof}
As a direct consequence~\cref{stentest} yields
\begin{cor}
    Let $h_t$ be a solution to \eqref{finaldstfe} such that the mobility $M(h)$ satisfies \cref{mobilitym} for $m \geq 3$ and the initial datum satisfies $\mathbb{E}\pra*{S(h_0)} < \infty$. Then we have that
    \begin{align}
        \mathbb{P}\bra*{h > 0} = 1.
    \end{align}
\end{cor}   
In particular, we do not have to impose the reflecting boundary condition for the SDE \eqref{finaldstfe} if $m \geq 3$.
The main selling point of our discretization is thus that we do not need to impose additional physics and/or rely on numerical tricks in the simulation in order to preserve positivity.
\begin{rem}
    Although~\cref{stentest} yields positivity for $m \geq 3$ for every fixed $N \in \N$ the bound on the entropy grows with $N$.
    % Assuming we could do better and find a constant $C > 0$ independent of $N$ such that we have
    % \begin{align}\label{falseest}
    %     \mathbb{E}\pra*{S(h_t)} = \mathbb{E}\pra*{S(h_0)} - \int_{0}^t \mathbb{E}\pra*{\frac{1}{N} \sum_{i=1}^N \bra*{ \bra*{A^T A h_r}^i}^2} \dx{r} + \frac{C}{\beta} t
    % \end{align}
    % and this yields the estimate
    % \begin{align}\label{falseest2}
    %     \int_{0}^t \mathbb{E}\pra*{  \frac{1}{N} \sum_{i = 1}^N \bra*{\bra*{A^T A h_r}^i}^2 } \dx{r} \leq \mathbb{E}\pra*{S(h_0)} + \frac{C}{\beta}t
    % \end{align}
    % and hence the dissipation is bounded uniformly in $N$.
    First of all, it is clear that \eqref{enteq} does not survive naively in the limit $N \to \infty$ since the term $\frac{2N^3t}{\beta}$ will blow up. On the other hand, one can rearrange terms in the following way
    \begin{align}\label{smarteq}
        \mathbb{E}\pra*{S(h_t)} - \mathbb{E}\pra*{S(h_0)}  =  \frac{2N^3}{\beta} t - \int_{0}^t \mathbb{E}\pra*{\frac{1}{N} \sum_{i=1}^N \bra*{ \bra*{A^T A h_r}^i}^2} \dx{r}.
    \end{align}
    The spatial increments of $h_r$ behave like Brownian motion and hence the dissipation term $\mathbb{E}\pra*{\frac{1}{N} \sum_{i=1}^N \bra*{ \bra*{A^T A h_r}^i}^2}$ scales like $N^3$ which shows that the scaling in $N$ on the right hand side of \eqref{enteq} is natural and it is not unreasonable to expect that the right hand side of \eqref{smarteq} converges for $N \to \infty$. On the other hand, at equilibrium the right hand side of \eqref{smarteq} does not depend on the mobility but for $m \geq 5$ the left hand side is not finite in the continuum limit and thus we do not expect an equality like $\eqref{enteq}$ to hold for $N \to \infty$.
\end{rem}
\begin{rem}\label{Nequaltwo}
    We present an argument that the ranges $m < 3$ and $m \geq 3$ are qualitatively very different. To this end, for $N=2$ we consider the associated Dirichlet form of the process, i.e. the right hand side of \eqref{dfpevar}, namely
    \begin{align}
        \int_{0}^2 \frac{1}{g(h)} \partial_h f(h) \partial_h \zeta(h) \dx \nu(h)
    \end{align}
    where (cf. \eqref{metriccomp})
    \begin{align}
        g(h) \sim h^{1-m}(2-h)^{1-m}.
    \end{align}
    We perform a change of variables $h \mapsto \hat{h}$ that is defined according to
    \begin{align}
        \frac{\dx \hat{h}}{\dx h} = \sqrt{g}(h)
    \end{align}
    and we note that this yields the transformation
    \begin{align}
        g^{-1}(h) \partial_h f(h) \partial_h \zeta(h) \rightarrow \partial_{\hat{h}} f(\hat{h}) \partial_{\hat{h}} \zeta(\hat{h}).
    \end{align}
    Then for $h \ll 1$ we have
    \begin{align}
        \hat{h} \sim
        \begin{cases}
            \frac{2}{3-m} h^{\frac{3-m}{2}} \ &\mathrm{for} \ m \neq 3 \\
            \ln h \ &\mathrm{for} \ m = 3.
        \end{cases}
    \end{align}
    % \begin{align}
    %     \hat{h} \sim \frac{2}{3-m} h^{\frac{3-m}{2}}
    % \end{align}
    % if $m \neq 3$
    % as well as 
    % \begin{align}
    %     \hat{h} \sim \ln h
    % \end{align}
    % if $m = 3$.
    For $2-h \ll 1$ this holds similarly with $2-h$ instead of $h$.
    Hence for $m < 3$ the configuration space for $\hat{h}$ is bounded and for $m \geq 3$ it is unbounded and therefore we do not need any boundary conditions. In fact, this heuristic is in the spirit of the Feller test (cf. \cite[p.348, Theorem 5.29]{KS88}) which also yields that the process touches the boundary of the configuration space for $m < 3$ and does not for $m \geq 3$. For this reason, the threshold $m=3$ in \cref{stentest} is sharp.
    % We want to present an argument that the assumption $m \geq 3$ in \eqref{stentest} is sharp. To this end, we will perform a Feller test in the case $N=2$. The SDE \eqref{finaldstfe2} for $N=2$ after projecting takes the form
    % \begin{align}\label{1dsde}
    %     \dx h_t = \bra*{- 16^2 g^{-1}(h_t) (h_t - 1) + \frac{1}{\beta} 8 \bra*{g^{-1}}'(h_t)} \dx{t} + 4 \sqrt{\frac{1}{\beta}} \sqrt{g}^{-1}(h_t) \dx W_t
    % \end{align}
    % where 
    % \begin{align}
    %     g^{-1}(h) = \bra*{\dashint_{I} \frac{1}{M(h)} \dx{x}}^{-1}
    % \end{align}
    % for $I = [0, \frac{1}{2}]$.
    % Following p. 347- 348 in \cite{KS88} we choose some $a > 0$ and compute the quantity $v$ (which does not depend on $a$) as
    % \begin{align}
    %     v(x) &= \int_{a}^x e^{32\beta \int_{a}^y (z'-1) \dx z'} g^{-1}(y) \int_{a}^{y} \frac{\beta}{8} e^{-32\beta \int_{a}^z (z'-1) \dx z'} \dx{z} \dx{y} \\
    %     &\lesssim \frac{\beta}{8} \int_{a}^x g^{-1}(y) (y-a) \dx{y}
    % \end{align}
    % and by appealing to \eqref{metriccomp} we see that
    % \begin{align}
    %     v(x) \lesssim \int_{a}^x y^{2 -m} \dx{y}.
    % \end{align}
    % The right hand side is integrable for $x = 0$ if and only if $m < 3$. Hence by the Feller test (cf. \cite[Thm. 5.29, p. 348]{KS88}) we see that for $m < 3$ the process touches $0$.
    \label{rem:Feller}
\end{rem}
\section{The central difference discretization}\label{sec:cdd}
In this section we recall the finite-difference discretization used in \cite{Pavliotis} and compare it to the \emph{Gr\"un--Rumpf} discretization in the last section. We will argue that the finite-difference discretization has ''touch--down'' for any mobility $M(h)$, i.e. there is some $i = 1, \dots, N$ and some $t \geq 0$ such that $h^i_t = 0$.
% Consider the SPDE
% \begin{align}
%     \partial_t h + \partial_x \bra*{M(h) \partial_x^3 h} =  \sqrt{\frac{2}{\beta}} \partial_x \bra*{\sqrt{M(h)} \xi}
%     \label{stfe}
% \end{align}
% where $\xi$ is Gaussian space-time white noise.
% Again choose $N \in \N$ and let 
% \begin{align}
%     \R^N \ni h_t := \bra*{h(t, x_1), \dots, h(t, x_N)}^T
% \end{align}
% where $\set*{x_i}_{i = 1, \dots, N}$ is an equidistant partition of the torus $\T$.
% For convenience we will now recall the the setup of \cite{Pavliotis}.

\medskip

By $C = \bra*{C^{j}_i}^{j}_i$ we denote the central difference matrix, i.e. we have for all vectors $\bra*{b^i}^i$
    \begin{align}
        C^{j}_i b^i = N \bra*{b^{j + 1} - b^{j -1}}.
    \end{align}
and, moreover, we let 
\begin{align}
     G(h) := \bra*{g_{\al \al'}(h)}_{\al \al'}, \quad g_{\al \al'}(h) :=  \frac{1}{M(h^{\al})} \delta_{\al \al'}.
\end{align}
Then the finite-difference discretization of the SPDE \eqref{stfe} is the following SDE (cf. \cite[p.591, (38)]{Pavliotis})
% Thus we end up with the discrete stochastic thin-film equation with respect to the central difference discretization of the metric tensor
\begin{align}\label{dSTFEnaive}
        \dx h_t = - C^T G^{-1}(h_t) C A^T A h_t \dx t +  C^T \sqrt{G}^{-1}(h_t) \sqrt{\frac{2N}{\beta} } \dx W_t 
\end{align}
which is supplemented with reflecting boundary conditions on $\partial\set*{h > 0}$ and where the matrix $A$ is given by \eqref{forwarddif}.
In \cite[p.591-593]{Pavliotis} the authors check that the SDE \eqref{dSTFEnaive} obeys the detailed balance condition which is largely due to the fact that 
\begin{align}\label{itocorcd}
    \sum_{j = 1}^N \partial_j \bra*{C^T G^{-1}(h) C}^i_j = 0
\end{align}
for all $h \in \R^N$.
The term on the left hand side of \eqref{itocorcd} is reminiscent of the It\^o--correction term emerging in \eqref{finaldstfematrix}.
In particular, the equation \eqref{dSTFEnaive} has the same invariant measure as \eqref{finaldstfematrix}; see also \cref{sec:num:inv} for further numerical evidence on this.

\medskip

We will now give an argument that the process $h_t$ defined by \eqref{dSTFEnaive} touches down.
The boundary $\partial \mathcal{M}_N$ can be decomposed into several sets of lower codimension. We call the sets of codimension 1 the faces of the simplex, i.e. the sets of the form $F^i_N := \bar{\mathcal{M}_N} \cap \set*{h^i = 0, h^j > 0, j \neq i}$ for $i = 1, \dots, N$. Obviously, the hyperplane containing $F^i_N$ is orthogonal to the unit vector $e_i$.
% The diffusion matrix can be computed as
% \begin{align}\label{difmat}
%     & \hspace{2.5em} C^T G^{-1}(h) C =\\ &N^2
%         \begin{pmatrix}
%             M(h^2) + M(h^N) & 0 & -M(h^2) & \cdots & - M(h^N) & 0 \\
%             0 & M(h^1) + M(h^3) & 0 & \cdots & 0 & -M(h^1)  \\
%             \vdots  & \vdots  & \vdots & \vdots & \ddots & \vdots  \\
%             0 & -M(h^1) & 0 & \cdots & 0 & M(h^{N-1}) + M(h^1) 
%     \end{pmatrix}
% \end{align}
% \begin{align}\label{difmat}
%     \bra*{C^T G^{-1}(h) C}_{ij} = N^2
%     \begin{cases}
%         M(h^{i-1}) + M(h^{i+1}), & i=j \\
%         -M(h^{j+1}), & j = i + 2 \\
%         -M(h^{j-1}), & j = i - 2 \\
%         -M(h^{i + 1}), & i = j + 2 \\
%         -M(h^{i-1}), & i = j - 2 \\
%         0, & \mathrm{else}.
%     \end{cases}
% \end{align}
Note that the quadratic variation of $h^i_t$ is given by $\int_0^t \bra*{C^T G^{-1}(h_t) C}_{ii} \dx{t}$.
Then we see that the matrix $C^T G^{-1} C$ does not degenerate in the direction orthogonal to the faces since
\begin{align}
    \bra*{C^TG^{-1}(h)C}_{ii} = N^2 \bra*{M(h^{i-1}) + M(h^{i+1})} > 0
\end{align}
for $h \in F^i_N$ and hence the quadratic variation stays positive even on $F^i_N$.
This suggests that this discretization of the stochastic thin-film equation indeed features touch-down and we also observe this phenomenon numerically, see~\cref{sec:numpositivity}. Notice that on the other hand in case of the \emph{Gr\"un--Rumpf} discretization, the corresponding diffusion matrix $C^T G^{-1}C$ does degenerate in the direction orthogonal to the faces. We provide a small schematic for $N=3$ in~\cref{fig:degeneracy} to demonstrate these features of the two discretizations. 

\begin{figure}
    \centering
    \begin{tikzpicture}

    \coordinate[] (A) at (0,0);

    \node at (A)[circle,fill,inner sep=2pt, color=red]{};
    
    \coordinate[] (B) at (3,0);

    \node at (B)[circle,fill,inner sep=2pt,color=red]{};
    
    \coordinate[] (C) at (1.5,2.59);

    \node at (C)[circle,fill,inner sep=2pt, color=red]{};

    % \coordinate (F) at (2.5,0);
    
    % \coordinate (E) at (-0.5,0);
    
     \draw[line width=0.25mm] (A) -- (B) -- (C) -- (A);
      \coordinate[label=below:(A)] (X) at (1.5,0);
     
    \coordinate[] (D) at (5,0);

    \node at (D)[circle,fill,inner sep=2pt, color=red]{};
    
    \coordinate[] (E) at (8,0);

    \node at (E)[circle,fill,inner sep=2pt,color=red]{};
    
    \coordinate[] (F) at (6.5,2.59);

    \node at (F)[circle,fill,inner sep=2pt, color=red]{};
    
    \draw[color=red,line width=0.25mm] (D) -- (E) -- (F) -- (D);
    
          \coordinate[label=below:(B)] (Y) at (6.5,0);
  
    \end{tikzpicture}
    \caption{The configuration space $\mathcal{M}_3$ for the two discretizations: central difference on the left (A) and \emph{Gr\"un--Rumpf} on the right (B). The edges and corners where the diffusion matrix degenerates are colored in red. As can be seen from the figure, the central difference discretization does not degenerate orthogonal to the  $d=1$ codimension subsets of $\mathcal{M}_3$,  while the \emph{Gr\"un--Rumpf} discretization degenerates on the whole boundary.}
    \label{fig:degeneracy}
\end{figure}
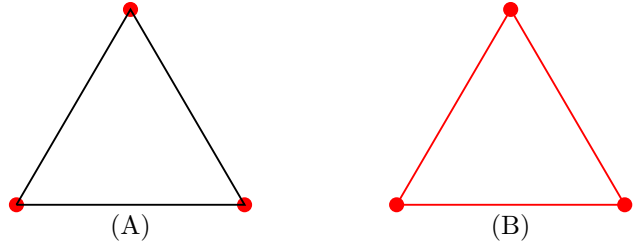
\begin{rem}[The It\^o-correction term]\label{itocorrection}
Consider the continuum stochastic thin-film equation in Stratonovich form with cut-off noise $\xi^N$ (i.e. cutting off at the $N$\textsuperscript{th} Fourier mode):
\begin{equation}
    \partial_t h = - \partial_x (M(h) \partial_x^3 h) + \sqrt{\frac{2}{\beta}}\partial_x(\sqrt{M(h)}\circ \xi^N) \, .
\end{equation}
It is fairly straightforward to check (cf.~\cite[Equation 2.5]{TN04}) that the same SPDE can be written down in It\^o form as follows
\begin{equation}
    \partial_t h = - \partial_x (M(h) \partial_x^3 h) + \frac{N}{8 \beta} \partial_x\bra*{\frac{(M'(h))^2}{M(h)}\partial_x h}+\sqrt{\frac{2}{\beta}}\partial_x(\sqrt{M(h)} \xi^N)  \, .
\end{equation}
The above situation closely mimics the one in our scenario: We have presented two spatial discretizations of the thin-film equation with thermal noise and they differ from each other by the correction term
\begin{equation}
    \frac{N}{\beta}A^T \bar{D} \cdot G^{-1}(h_t) \, .
\end{equation}
The reader can convince themselves, that as $N$ goes to $\infty$,  the above expression formally converges to 
\[
-\frac{N}{\beta} \partial_x\bra*{(M(h))^2\partial_x\bra*{\frac{M'(h)}{(M(h))^2} }}= \frac{N}{\beta} \partial_x\bra*{\bra*{2 \frac{(M'(h))^2}{M(h)}- M''(h) }\partial_x h}
\footnote{for the specific case of $m=3$ one can see this from the explicit form of the It\^o-correction term provided in~\eqref{itocnum}}.
\]
For the case of power law mobilities $M(h)=h^m$, one can check that the two correction terms are the same, up to a multiplicative constant. This observation is consistent with the finding of~\cite{HMW14} in which the authors discuss how different spatial discretizations of the stochastic Burgers equation can differ by terms which are analogous to the It\^o-to-Stratonovich correction for SDEs. It would not be unreasonable to expect that such a term plays a role in renormalization as a possible counter term.
    \label{rem:itocorrection}
\end{rem}
\section{Touch-down for the continuum system} \label{sec:ldp}
The open question of whether the deterministic thin-film equation with cubic mobility preserves positivity, is related to the degeneracy of the mobility when the film height approaches zero. In fact, in the case of high mobility exponent $m\ge  \frac{7}{2}$, it has been shown that indeed strict positivity is preserved (cf. \cite[p.194 , Theorem 4.1, (iii)]{Beretta}), while the opposite has been shown for $m < \frac{1}{2}$ in  \cite[p.198, Theorem 6.1]{Beretta}.

\medskip

In this section, we would like to discuss the same question (touch-down vs. positivity) for the continuum thin-film equation with thermal noise. We address this question through the associated large deviations rate functional of the continuum system. Before proceeding, we note that the entropic repulsion exhibited by the conservative Brownian excursion defined in~\cref{subsec:eqmeasure} is a purely energetic phenomenon. As such, it is independent of the degeneracy of the mobility and is thus orthogonal to the discussion of touch-down which will be presented in this section.  

\medskip

There is a well-known connection between the large deviation principle for a microscopic reversible Markov process and the (appropriate) gradient flow structure of its mean-field limit (cf. ~\cite{DG87,MPR14}). It is classical that for a reversible stochastic perturbation of a (finite-dimensional, but Riemannian) gradient flow, the rate functional $I$ is given in terms of the metric tensor $g$ and the energy function $E$ (see, for example,~\cite[Chapter 4, Section 3, Theorem 3.1]{FW84}): For a given time horizon $[0,T], T>0$, $I_T$ is the following functional on the space of all paths $[0,T]\ni t\mapsto h_t\in{\mathcal M}$ 
\begin{align}
I_T(h):=&\frac{1}{2}\int_0^Tg_{h_t}\left(\frac{\dx h_t}{\dx{t}}+\nabla E(h_t),
\frac{\dx h_t}{\dx{t}}+\nabla E(h_t)\right) \dx t  \\
=&\frac{1}{2}\int_0^Tg_{h_t}\left(\frac{\dx h_t}{\dx{t}},\frac{\dx h_t}{\dx{t}}\right) \dx t
+\frac{1}{2}\int_0^Tg_{h_t}\left(\nabla E(h_t),\nabla E(h_t)\right)\dx t \label{mn01} \\
&+E(h_T)-E(h_0).
\end{align}

\medskip

Formally,~\eqref{mn01} extends to infinite-dimensional situations like ours: While the SPDE might require a renormalization, the rate functional often does not (cf. ~\cite{HW15}) -- and can be analyzed rigorously (cf. ~\cite{KORVE07}). We take this route in order to give a heuristic argument that touch-down is generic for power-law mobilities\footnote{we consider power-law mobilities for convenience. One would expect the same result to hold with more general mobilities under the appropriate upper and lower bounds on the mobility.} $M(h)=h^m$ with mobility exponents $m<8$ and constitutes an extremely unlikely event for $m\ge 8$. To this end, we assume that the small-noise/high temperature large deviations rate functional $I_T$ for~\eqref{stfe} is given by~\eqref{mn01} with $g_{h_t}$ defined as in~\eqref{mettens}\footnote{in the sequel, for the sake of simplicity, we will consider the metric $g_h$ (and the equation) on $\R$. It can be defined in the natural way as in~\eqref{mettens}.}, $E$ given by the Dirichlet energy \eqref{energy}, and the gradient $\nabla E$ defined by duality as in~\eqref{gradient}. We first present our result for $m<8$, where we argue that touch-down is a generic phenomenon using an upper bound for the rate functional obtained via a self-similar ansatz.

\begin{prop}
Assume $M(h)=h^m$ for some $m<8$ and fix $T>0$. Then, there exists a curve $[-T,0] \ni t \mapsto  h_t \in \cM $ such that 
\begin{equation}
I_T( h)<\infty, \quad \min_{x \in \R}  h_{-T} >0, \quad \mathrm{and} \quad \min_{x \in \R}  h_0 =0 \, .
\end{equation}
\label{m<8}
\end{prop}
\begin{proof}
For the sake of convenience, we present the proof only for the range $1<m<8$.    For any curve $[-T,0]\ni t \mapsto h_t \in \cM$, we can write the rate functional as follows
\begin{align}
I_T(h)  =& \frac{1}{2}\int_{-T}^0 g_{h_t}(\partial_t h_t, \partial_t h_t) \dx t
+\frac{1}{2}\int_{-T}^0 g_{h_t}\bra*{\nabla E(h_t), \nabla E(h_t)} \dx t \\
&+E(h_0)-E(h_{-T}) \, .
\label{ratefunction2}
\end{align}
Note that we can apply Cauchy--Schwarz and Young's inequality to obtain the bound
\begin{align}
    &\abs*{E(h_0)-E(h_{-T})} = \abs*{\int_{-T}^0 g_{h_t}\bra*{\partial_t h_t,\nabla E(h_t)} \dx{t}}  \\
    \leq & \frac{1}{2}\int_{-T}^0 g_{h_t}(\partial_t h_t, \partial_t h_t) \dx t +\frac{1}{2}\int_{-T}^0 g_{h_t}\bra*{\nabla E(h_t), \nabla E(h_t)} \dx t \, .
    \label{eq:energyI}
\end{align}
This leaves us with
\begin{equation}
I_T(h)  \leq \int_{-T}^0 g_{h_t}(\partial_t h_t, \partial_t h_t) \dx t +\int_{-T}^0 g_{h_t}\bra*{\nabla E(h_t), \nabla E(h_t)} \dx t.
\end{equation}
We now consider the following self-similar ansatz
\[
 h_t(x)=(-t)^{\eta\gamma}\hat{h}(x(-t)^{-\eta}),\quad\hat{h}(\hat x)=(\hat x^{2}+1)^{\frac{\gamma}{2}},
\]
with $\eta>0$ and $0<\gamma<1$. Then,
\[
\lim_{t\uparrow 0}h_t(x)=|x|^{\gamma}.
\]
We thus have that
\begin{align}
h_t(x) & =(x^{2}+(-t)^{2\eta})^{\frac{\gamma}{2}}\, .
\end{align}
Note now that, from the definition of the metric tensor \eqref{mettens},
\begin{align}
 \int_{-T}^{0} g_{h_t}(\partial_t h_t,\partial_t h_t) \dx{t} = \int_{-T}^0 \int_{\R}\frac{j_t^2}{ h_t^m}  \dx{x} \dx t \, ,
\end{align}
where $ j=j_t$ is a time-dependent flux field satisfying 
 \begin{align}
 \partial_t  h_t + \partial_x  j_t=0 \, .
 \end{align}
 It turns out that $ j_t$ also has a simple structure in self-similar variables. Indeed, it can be written as
 \begin{equation}
  j_t(x) = (-t)^{\eta \gamma +\eta-1}\hat{j}(x(-t)^{-\eta}) \, ,
 \end{equation}
 where
 \begin{equation}
\hat{j}(\hat x)= -\eta \gamma \int_0^{\hat x} (y^2 +1)^{\frac{\gamma}{2}-1} \dx{y} \, .
 \end{equation}

 We then have that
\begin{align}
\int_{-T}^{0}g_{h_t}(\partial_t h_t,\partial_t h_t) \dx t= & \int_{-T}^{0} (-t)^{\eta \gamma (2-m ) +2\eta -2}\int_{\R} \frac{\hat j^2(x(-t)^{-\eta})}{\hat h^m(x(-t)^{-\eta})} \dx{x}\dx{t}\\
=& \int_{-T}^{0} (-t)^{\eta \gamma (2-m )  + 3\eta-2}\int_{\R} \frac{\hat j^2(\hat x)}{\hat h^m(\hat x)} \dx{\hat x}\dx{t} \, .
% = & \int_{0}^{T}\int_{\R}\frac{[\eta\gamma(|x|^{2}+t^{2\eta})^{\frac{\gamma}{2}-1}t^{2\eta-1}]^{2}}{(|x|^{2}+t^{2\eta})^{\frac{m\gamma}{2}}}\\
% = & (\eta\gamma)^{2}\int_{0}^{T}t^{2(2\eta-1)}\int_{\R}\frac{[(|x|^{2}+t^{2\eta})^{\frac{\gamma}{2}-1}]^{2}}{(|x|^{2}+t^{2\eta})^{\frac{m\gamma}{2}}}\\
% = & (\eta\gamma)^{2}\int_{0}^{T}t^{2(2\eta-1)-m\gamma\eta+2\eta(\gamma-2)+2\eta}\int_{\R}\frac{[(|x|^{2}+1)^{\frac{\gamma}{2}-1}]^{2}(xt^{-\eta})}{(|t^{-\eta}x|^{2}+1)^{\frac{m\gamma}{2}}}\\
% = & (\eta\gamma)^{2}\int_{0}^{T}t^{2(2\eta-1)-m\gamma\eta+2\eta(\gamma-2)+3\eta}\int_{\R}\frac{[(|x|^{2}+1)^{\frac{\gamma}{2}-1}]^{2}(x)}{(|x|^{2}+1)^{\frac{m\gamma}{2}}}
\end{align}
For the integrability of the time-dependent term in the integrand we require that
\begin{align}
\eta \gamma (2-m )  + 3\eta >1 \, .
\label{tt1}
\end{align}
On the other hand, for the space-dependent term in the integrand we note that  $|\hat j|(\hat x) \lesssim 1+ (\hat x^2 +1)^{\frac{\gamma-1}{2}} $ and $\hat h(\hat x) = (\hat x^2 +1)^{\frac{\gamma}{2}} $. It follows that for the integrability of this term it is sufficient to have
\begin{align}
-m\gamma < -1 \, .
\label{st1}
\end{align}
We now turn our attention to the second term in~\eqref{ratefunction2}. We compute
\begin{equation}
\partial_x^3  h_t= (-t)^{\eta(\gamma-3)}\hat h'''(x(-t)^{-\eta}) \, .
\end{equation}
Using the definition of the metric tensor~\eqref{mettens} and of the gradient $\nabla E$~\eqref{gradient}, we obtain
\begin{align}
\int_{-T}^{0} g_{h_t}(\nabla E(h_t),\nabla E(h_t)) \dx{t} = &\int_{-T}^{0}\int_{\R} h_t^{m}(\partial_{x}^3  h_t)^{2} \dx{x}\dx{t}\\
=& \int_{-T}^0 (-t)^{\eta \gamma (m+2)- 6 \eta } \int_{\R} \hat h^{m}(x(-t)^{-\eta}) (\hat h'''(x(-t)^{-\eta}))^2\dx{x} \dx{t} \\
=& \int_{-T}^0 (-t)^{\eta \gamma (m+2)- 5 \eta } \int_{\R} \hat h^{m}(\hat x) (\hat h'''(\hat x))^2\dx{\hat x} \dx{t} \, .
\end{align}
For the integrability of the time-dependent term in the above expression, it is sufficient to have
\begin{equation}
\eta(\gamma(m+2)-5)>-1 \, .
\label{tt2}
\end{equation}
On the other hand, note that $\bra*{\hat{h}^{m}(\hat h''')^2}(\hat x) \lesssim (\hat x^2 +1)^{\frac{ m\gamma}{2} +\gamma -3}$. Thus, for the integrability of the space-dependent term we require
\begin{equation}
m \gamma + 2 \gamma -6 <-1 \, .
\label{st2}
\end{equation}
We first note that~\eqref{st1} can be reduced to
\begin{equation}
\frac{1}{m}<\gamma <1  \, ,
\end{equation}
if $1<m<8$. On the other hand,~\eqref{st2} is equivalent to the following condition 
\begin{equation}
\gamma < \frac{5}{2 +m} \, .
\label{st3}
\end{equation}
The remaining conditions~\eqref{tt1} and~\eqref{tt2} can be reformulated as
\begin{equation}\label{eta}
3-\gamma(m-2)  >\frac{1}{\eta}>5-\gamma(m+2) \, .
\end{equation}
Note that if~\eqref{st3} is satisfied then $5-\gamma(m+2)$ is always larger than $0$. On the other hand, $3-\gamma(m-2) > 5-\gamma(m+2) $ if and only if $\gamma>1/2$ . Thus, we can choose $\gamma$ such that
\begin{align}
\max\left(\frac{1}{2},\frac{1}{m}\right)< \gamma < \min\left( 1, \frac{5}{2+m}\right) \, ,
\end{align}
for all $1<m<8$.  We can then choose $\eta>0$ so that \eqref{eta} is satisfied. Thus, for these choices of $\eta$ and $h$ we have $I_T(h) < \infty$, and the result follows.
% Now we let $\bar h$ be the time reversal of ${h}$, i.e. $\bar h_t={h}_{T-t}$. From our previous calculations, it is then clear that
% \begin{align}
% I_T(\bar h)  < &\infty \, ,
% \end{align}
% and that
% \[
% \bar{h}_0= h_T=T^{2\eta \gamma}\hat{h}( x T^{-\eta}),\quad\lim_{t\uparrow T}\bar h_t=|x|^{\gamma}.
% \]
% Consequently, starting from the positive initial state $\bar h_0$, the probability for a solution of~\eqref{stfe} to reach a neighborhood of the state $\bar h_T=|x|^\gamma$ (which is $0$ at $x=0$) can be estimated by the large deviations lower bound: for all $\delta>0$, there is an $\beta_{0}=\beta_{0}(\delta)$ so that, for all $\beta<\beta_{0}$,
% \[
% \mathbb{P}(\| h_T-|x|^{\gamma}\|_{\Leb^\infty(\R)}<\delta)\geq C \exp(-\beta I(\bar{h}))>0 \, ,
% \]
% for some constant $C>0$ independent of $\beta>0$.
\end{proof}
% \begin{rem}
% The above result tells us that with positive probability the solution $h$ to the stochastic thin-film equation~\eqref{stfe} with initial profile $h_0$ reaches values arbitrarily close to zero. Hence, microscopically small scales are expected to be reached with positive probability. 
% %We also note that the choice of initial datum $h_0$ in the statement of~\cref{m<8} is more out of convenience than necessity. One would expect the result to hold true for all positive initial data by constructing appropriate paths (along which the rate functional is finite) between the initial datum and $h_0$.
% \end{rem}
We now turn to the case $m \geq 8$ where we argue that touch-down is an extremely rare event by obtaining an ansatz-free diverging (as $h \to 0$) lower bound for the rate functional. For simplicity, we restrict ourselves to paths $[0,T]\ni t \mapsto h_t $ that start at $h_0 \equiv 1$.
\begin{prop}
Assume $M(h)=h^m$ for some $m \geq 8$.  Then, for any path $[0,T]\ni t \mapsto h_t \in \cM $ starting from $h_0\equiv 1$, the rate function $I_T$ diverges in the following quantitative sense
\begin{align}\label{mn10}
T^\frac{1}{4}I_T(h)\gtrsim
\begin{cases}
\sup_{x \in \R}\left(\ln\frac{1}{h_T}-1+h_T\right)_+  & m=8\\
\sup_{x \in \R}\left(\frac{1}{h_T^{\frac{m}{8}-1}}-1\right)_+^2 & m>8 
\end{cases}
\, ,
\end{align}
as $\inf_{x \in \R}h_T \to 0$\footnote{although we present the result for $\R$ an essentially identical argument should also work for the torus} where the implicit constant in $\gtrsim$ depends only on $m$.
\label{m>=8}
\end{prop}
\begin{proof}
We note first that the second identity in~\eqref{mn01} yields the following inequality
\begin{align}\label{mn04}
E(h_t)\le I_T(h)+E(h_0) \, ,
\end{align}
for all $t \in [0,T]$. Note that in view of~\eqref{mettens} we learn from~\eqref{mn01} that there exists
a time-dependent flux field $j=j_t(x)$ satisfying the continuity equation
\begin{align}\label{mn02}
\partial_th_t+\partial_xj_t=0 \, ,
\end{align}
such that the dissipation is controlled as 
\begin{align}\label{mn05}
\frac{1}{2}\int_0^T\int_{\R}\frac{j_t^2}{h_t^m}  \dx{x} \dx{t} \le I_T(h)+E(h_0) \, .
\end{align}
We again monitor some ``\emph{entropy}'' $\int_{\R} s(h_t) \dx{x}$ along the path, 
where $s=s(h)$ is now defined via
\begin{align}\label{mn03}
\begin{array}{c}
s(1)=s'(1)=0,\quad s(h)=0\;\mbox{for}\;h\ge 1,\\[1ex]
s''(h)=\frac{1}{h^{\frac{m}{2}}}\;\mbox{for}\;h<1.
\end{array}
\end{align}
Since by~\eqref{mn02}
\begin{align}
\frac{\dx}{\dx{t}}\int_{\R} s(h_t)\dx{x}=\int_{\R} s''(h_t)j_t\partial_x h_t \dx{x} \, ,
\end{align}
we obtain from~\eqref{mn03} and by Cauchy-Schwarz in the $x$-variable
\begin{align}
\abs*{\frac{\dx}{\dx{t}}\int_{\R} s(h_t) \dx{x} }^2\le\int_{\R}\frac{j_t^2}{h_t^m} \dx{x}
\int_{\R}\bra*{\partial_xh_t}^2  \dx{x} \, ,
\end{align}
Thus, by~\eqref{mn04} and~\eqref{mn05}, 
\begin{align}
\int_0^T\abs*{\frac{\dx}{\dx{t}}\int_{\R} s(h_t) \dx{x}}^2 \dx{t}\le 2(I_T(h)+E(h_0))^2 \, .
\end{align}
By integration and Cauchy-Schwarz in the $t$-variable, this yields
%
%\begin{align}
%|S(h_T)-S(h_0)|^2\le2T(R_T+E(h_0))^2.
%\end{align}
%
%
\begin{align}
\frac{1}{\sqrt{2T}}\abs*{\int_{\R} s(h_T)\dx{x}-\int_{\R} s(h_0) \dx{x}}\le I_T(h)+E(h_0) \, .
\end{align}
Appealing once more to~\eqref{mn04} this entails
\begin{align}
\frac{1}{\sqrt{2T}}\int_{\R} s(h_T) \dx{x} +E(h_T)\le\frac{1}{\sqrt{2T}}\int_{\R} s(h_0)\dx{x}+2E(h_0)+2I_T(h) \, .
\end{align}
For our special initial data $h_0\equiv 1$ and in view of~\eqref{mn03}, this simplifies to
\begin{align}\label{mn07}
\frac{1}{\sqrt{2T}}\int_{\R} s(h_T) \dx{x}+\frac{1}{2}\int_{\R}(\partial_x h_T)^2 \dx{x} \le 2I_T(h) \, .
\end{align}
%
%Using Assumption \ref{7.2} on definition (\ref{mn03}) we get 
We note now that 
\begin{align}
s(h) \gtrsim \bra*{\frac{1}{h^{\frac{m}{4}-1}}-1}_+^2 \, .
\end{align}
Thus,~\eqref{mn07} implies by Cauchy-Schwarz in the $x$-variable
\begin{align}
\int_{\R} \left(\frac{1}{h_T^{\frac{m}{4}-1}}-1\right)_+\abs*{\partial_x h_T} \dx{x} \lesssim T^\frac{1}{4}I_T(h) \, .
\label{mn00}
\end{align}
For $m=8$, the left hand side of the above expression is equal to $\int_{0}^1\abs*{\partial_x(\ln\frac{1}{h_T}+h_T-1)_+}\dx{x}$.
Since the spatial average of $h_T$ is equal to one, $(\ln\frac{1}{h_T}+h_T-1)_+$ must vanish in at least one point. Thus, the left hand side of~\eqref{mn00} controls $\sup_{x \in \R}(\ln\frac{1}{h_T}+h_T-1)_+$. This establishes the first item in~\eqref{mn10}; the second item follows similarly.
\end{proof}
We conclude this section by showing that a curve with finite rate functional also has the expected regularity in time. While this is a priori unrelated to non-negativity of the film, we will see that we can use this scale-invariant regularity estimate to obtain a strengthening of~\cref{m>=8}  in ~\cref{cor:m>=8}, where the mobility exponent $m=8$ again plays a special role.
\begin{prop}\label{timereg}
Assume $M(h)=h^m$ for some $m \geq 0$ and consider a curve $[0,\infty) \ni t \mapsto h_t \in \mathcal{M}$ such that
\begin{align}
\bar I(h):=\frac{1}{2}\int_{0}^\infty g_{h_t}(\partial_t h_t, \partial_t h_t) \dx t +  \frac{1}{2}\int_{0}^\infty g_{h_t}(\nabla E (h_t), \nabla E (h_t)) \dx t + E(h_0)  <\infty \, .
\end{align}

Then, $h_t$ is locally H\"older continuous in time with exponent $\frac{1}{8}$. Furthermore, it satisfies the following scale-invariant estimate
\begin{equation}
    \abs{h_t(x) -h_s(y)} \lesssim \bar I^{\frac{1}{2}}(h)\bra*{\min\{h_t(x),h_s(y)\}^{\frac{m}{8}}\abs{t-s}^{\frac18} + \abs{x-y}^{\frac12}} \, ,
\end{equation}
for all $x,y \in \R$ and $|t-s| \ll \bar{I}^{-4}(h) \min\{h_t(x), h_s(y)\}^{8-m}$, where the implicit constants in $\lesssim$, $\ll$ depend only on $m$.
\end{prop}
\begin{proof}
To start with, we consider the case where $[0,\infty)\ni  t \mapsto {h}_{ t}$ is such that ${h}_0(0) \leq 1$ and $\bar I( h) \leq 3$. Note that this along with~\eqref{eq:energyI} implies that
\begin{equation}
\sup_{t \in [0,\infty)}E( h_{ t}) \leq \bar I( h) \leq 3 \, ,  
\label{eq:energybounded}
\end{equation}
which in turn implies that $h_t$ is $\frac12$-H\"older continuous in space for all $t\geq 0$ with the bound
\begin{equation}
    \abs*{{h}_{ t}( x) - {h}_{ t}( y)} \lesssim \bar{I}^{\frac12}( h)\abs{ x- y}^{\frac12} \, .
    \label{eq:holder12}
\end{equation}
We now fix a smooth compactly supported nonnegative function $\varphi $ which is strictly positive in $(-1,1)$ and satisfies $\int_\R \varphi \dx{x}=1$ and $\varphi(x) \leq 1$. We then define
\begin{equation}
    F_{ t} := \int_{\R}\varphi  h_{ t} \dx{ x} \, . 
\end{equation}
We then have
\begin{align}
    \frac{\dx }{\dx  t} F_{ t} =& \int_{\R}  \varphi'  j_{ t} \dx{ x} \, , 
\end{align}
where $ j_{ t}= j_{ t}( x)$ is a time-dependent flux field which solves
\begin{equation}
    \partial_{ t}  h_{ t} + \partial_{ x}  j_{ t} =0 \, .
\end{equation}
Dividing and multiplying by $h_t^{\frac m2}$ and then applying the Cauchy--Schwarz inequality in space, we obtain
\begin{align}
     \frac{\dx }{\dx  t} F_{ t} 
     \leq &\bra*{\int_{\R} \bra*{ \varphi'}^2  h_{ t}^{m} \dx{ x}}^{\frac12} a( t) \, 
\end{align}
where
\begin{align}
    a({t}) := \bra*{\int_{\R} \frac{ j_{ t}^2}{ h_{ t}^m} \dx{ x}}^{\frac12}.
\end{align}
For the first term  on the right hand side of the above expression, we have the following bound
\begin{align}
    \int_{\R} \bra*{\varphi'}^2  h_{ t}^{m} \dx{ x} \leq & \sup_{x \in \R}(\varphi')^2 \abs*{\int_{-1}^1  h_{ t}^m \dx{ x}} \\
    \lesssim & \abs*{\int_{-1}^1 \bra*{\min_{ x \in [-1,1]} h_{ t}( x) + \int_{ x_*}^{ x}\partial_{ y}  h_{ t}( y) \dx{ y}}^m \dx{ x}}  \, ,
\end{align}
where $ x_*= \mathrm{argmin}_{ x \in [-1, 1]} h_{ t}( x)$. Using~\eqref{eq:energybounded} and Jensen's inequality and the fact that $\varphi$ is strictly positive in $(-1,1)$, we obtain
\begin{align}
    \int_{\R} \bra*{\varphi'}^2  h_{ t}^m\dx{ x} 
    \lesssim & \bra*{ \min_{ x \in [-1,1]}  h_{ t}( x)    + \abs*{ \int_{-1}^1 \bra*{\int_{ x_*}^{ x} \partial_{ y}  h_{ t}( y) \dx{ y}} \dx{ x} }}^m  \\
    \lesssim &  (F_{ t} + 1)^m \, .
\end{align}
 This leaves us with
\begin{align}
    \frac{\dx}{\dx  t} F_{ t} \lesssim &\bra*{1 +F_{ t} }^{\frac{m}{2}}a( t)\, . 
\end{align}
We can now use the fact $\int_0^\infty a^2( t) \dx{ t}\leq 2 \bar{I} ( h)\leq 6$ along with the Cauchy--Schwarz and Young inequalities, to rewrite the above inequality as
\begin{align}
     F_{ t} \lesssim & 1 + F_{0} + {t} +\int_0^{{t}}  F_{ s}^m \dx{ s} \, .
\end{align}
We thus obtain for $t\leq 1$ (cf.~\cref{lemintineq})
\begin{align}
    F_{ t} \lesssim& \bra*{(1 + {t} + F_0)^{1-m} + (1-m){t}}^{\frac{1}{1-m}} \, ,
\end{align}
if $m \neq 1$ and 
\begin{align}
    F_{ t} \lesssim & (1 + F_0) e^{ C {t}}  \, ,
\end{align}
if $m=1$ for some constant $C>0$. In either of the two cases, we have that $F_{ t} \leq 3$ for all $0<  t \leq t_*$ for some $t_* >0$ depending on $m$, as long as $F_0$ is finite, which itself holds true since~\eqref{eq:energybounded} and $h_0(0)\leq 1$ imply 
\[
F_0 \leq  \int_{-1}^1  h_0 \dx{ x} \lesssim 1 \, .
\]
We can then use Jensen's inequality and~\eqref{eq:energybounded} to obtain
\begin{align}
     h_{ t} ( x) = \min_{ x \in [-1,1]}  h_{ t}( x)  + \int_{x_*}^{ x} \partial_{ y}  h_{ t} \dx{ y }
    \lesssim 1 \, ,
    \label{eq:localbound}
\end{align}
for all $0<  t \leq t_*$ and $ x \in [-1/2,1/2]$.

By the shift-invariance\footnote{$\bar I$ is not truly shift invariant, but we simply use the fact that $\bar I(\tau_{y,s} h) \leq 2 \bar I(h)$ with $\tau_{y,s}h_t= h_{t+s }(\cdot +x)$ } of $\bar I$, we may check the time regularity of $h$ at some fixed point, say $ x, t=0$. Define $\varphi_\eps(\cdot):= \eps^{-1}\varphi(\eps^{-1}\cdot)$. Then, for any $0\leq  t \leq t_*$, we can use~\eqref{eq:holder12} to obtain
\begin{align}
\abs*{ h_{ t}(0)- h_0(0)} \lesssim \, \eps^{\frac{1}{2}} + \abs*{\int_0^{ t} \int_\R \varphi^\eps \partial_{ s} {h}_{ s} \dx{ x} \dx{ s}} \, .
\end{align}
As before, we use the fact that $ h_{ t}$ satisfies the continuity equation~\eqref{mn02} with time-dependent flux field $ j_{  t}= j_{ t}( x)$ to obtain
\begin{align}
    \abs*{ h_{ t}(0)- h_0(0)} \lesssim & \, \eps^{\frac{1}{2}}+ \abs*{\int_0^{ t} \int_\R \eps^{-2}\varphi'( x/\eps)  j_{ s}  \dx{ x} \dx{ s}} \, . 
\end{align}
Dividing and multiplying by $h_t^{\frac m2} $ as before and applying the Cauchy--Schwarz and Young inequalities, we obtain
\begin{align}
    \abs*{ h_{ t}(0)- h_0(0)} \lesssim & \, \eps^{\frac{1}{2}}+ \eps^{-\frac{9}{2}}\int_0^{ t} \int_{-\eps}^\eps \bra{\varphi'( x/\eps)}^2  h_{ s}^m   \dx{ x} \dx{ s}+ \eps^{\frac{1}{2}} \int_0^{ t} \int_\R \frac{ j_{ s}^2}{ h_{ s}^m} \dx{ x} \dx{ s} \, . 
    \label{eq:epsdel}
\end{align}
For the second term on the right hand side of the above expression, we rescale in $ x$ and use~\eqref{eq:localbound}, to obtain
\begin{align}
    \eps^{-\frac{9}{2}}\int_0^{ t} \int_{-\eps}^\eps \bra{\varphi'( x/\eps)}^2  h_{ s}^m \dx{ x} \dx{ s} \lesssim & \, \eps^{-\frac{7}{2}} t \, .
\end{align}
For the third term on the right hand side of~\eqref{eq:epsdel} we simply apply the bound~\eqref{mn05} and use the fact that the $\bar I ( h)$  is bounded to arrive at
\begin{align}
    \eps^{\frac{1}{2}}\int_0^{ t} \int_\R \frac{ j_{ s}^2}{ h_{ s}^m} \dx{ x} \dx{ s}  \lesssim & \, \eps^{\frac{1}{2}} \, .
\end{align}
This leaves us with 
\begin{align}
    \abs*{h_{ t}(0)- h_0(0)} \lesssim & \, \eps^{\frac{1}{2}} + \eps^{-\frac{7}{2}} t \, .
\end{align}
Choosing $\eps = t^{\frac{1}{4}}$ and applying~\eqref{eq:holder12}, we obtain
\begin{align}
\abs{ h_{ t}( x) - {h}_0(0)} \lesssim \abs{ t}^{\frac{1}{8}} + \abs{ x}^{\frac12}  \, ,   
\end{align}
for $( t, x) \in [0,t_*) \times \R $.  

We can now rescale  to recover the corresponding estimate for an arbitrary $[0,\infty) \ni t \mapsto h_t \in \cM $ with $\bar I(h)< \infty$. To this end, we introduce
\begin{align}
    \hat h_{\hat t} (\hat x) = \lambda h_t ( x) \, , \hat x =\mu  x \, , \hat t = \nu  t
\end{align}
for some $\lambda,\nu, \mu >0 $ to be chosen later.  Under this choice of scaling, we have
\begin{align}
    E(h_t) =& \int_{\R} \bra*{\partial_x h_t}^2 \dx{x} 
    =\mu \lambda^{-2}E(\hat h_{\hat t}) \, ,
\end{align}
and 
\begin{align}
    \frac{1}{2}\int_0^\infty \int_{\R}\frac{j_t^2}{h_t^m} \dx{x} \dx{t}  
   = \nu \mu^{-3}\lambda^{m-2}  \frac{1}{2}\int_0^\infty \int_{\R}\frac{\hat j_{\hat t }^2}{\hat h_{\hat t}^m} \dx{\hat x} \dx{\hat t}  \, ,
\end{align}
where $j_t=j_t(x)$ is as before and $\hat j_{\hat t}= \hat j_{\hat t}(\hat x)$ satisfies
\begin{align}
    \partial_{\hat t } \hat h_{\hat t} +\partial_{\hat x} \hat j_{\hat t} =0 \, .
\end{align}
Furthermore, the remaining term in $\bar I$ scales as
\begin{align}
    \frac{1}{2} \int_0^{\infty} \int_{\R} \bra*{\partial_x^3 h_t}^2 h^m_t \dx{x} \dx{t} = \lambda^{-m-2} \mu^5 \nu^{-1} \frac{1}{2} \int_0^{\infty} \int_{\R} \bra*{\partial_{\hat{x}}^3 \hat{h}_{\hat{t}}}^2 \hat{h}^m_{\hat{t}} \dx{\hat{x}} \dx{\hat{t}} \, .
\end{align}
Since we may assume, without loss of generality, that $h_0(0)>0$, we make the following choices
\begin{align}
    \lambda = \frac{1}{h_0(0)} \, , \mu =\lambda^2 \bar I (h)\, , \nu = \mu^3 \lambda^{2-m} \bar I(h) \, .
\end{align}
It follows that $\bar I(\hat h) \leq 3$, and $\hat{h}_0(0)=1$. We thus have
\begin{align}
    \abs{h_t(x)-h_0(0)} =& \lambda^{-1}\abs{\hat h_{\hat t}(\hat x)-1} 
    \lesssim  \lambda^{-1} \bra*{\nu^{\frac18}\abs{t}^{\frac18} + \mu^{\frac12}\abs{x}^{\frac12} } \\
    \lesssim & h_0(0)\bra*{ \bar I^{\frac12}(h) h_0^{\frac{m-8}{8}}(0)\abs{t}^{\frac18} + h_0^{-1}(0) \bar I^{\frac12}(h)\abs{x}^{\frac12}} \, ,
\end{align}
 for all $0 \leq t \leq \bar I^{-4}(h) h_0^{8-m}(0)t_*$  and $ x \in \R$. 
%Applying~\eqref{mn05}, we obtain
% \begin{align}
%     \abs{h_t(x)-h_0(0)} \lesssim \bar I^{\frac{1}{2}}(h)\bra*{h_0^{\frac{m}{8}}(0) \abs{t}^{\frac18} + \abs{x}^{\frac12}}
% \end{align}
% for all $0 \leq t \leq \bar I^{-4}(h) h_0^{8-m}(0)t^* $ and $x \in \R$.
\end{proof}
\begin{cor}
    Let $m \geq 8$ and let $t \mapsto h_t \in \mathcal{M}$ satisfy $\bar{I}(h) < \infty$.
    Assume that, for some $x \in \R$, $h_0$ is almost touching down, i.e. $h_0(x) \ll 1$. Then, for all $t \geq 0$ such that $h_{t}(x) = 1$ it holds that
    \begin{align}
        t \gtrsim
        \begin{cases}
            \bar{I}^{-4}(h)h_0^{8-m}(x) \, &\mathrm{for} \, m > 8 \\
            \bar{I}^{-4}(h)\ln(h^{-1}_0(x)) \, &\mathrm{for} \, m = 8 \, ,
        \end{cases}
    \end{align}
    where the implicit constant in $\gtrsim$ depends only on $m$.
    \label{cor:m>=8}
\end{cor}
\begin{proof}
   The dependence on $x$ does not play any role in the proof since the argument we will present is pointwise in space. 
  We will thus omit it for the rest of the proof. Moreover, we will set the implicit constants in $\lesssim$ in \cref{timereg} to $1$.
    By \cref{timereg}, we have for $0 \leq t \leq \bar{I}^{-4}(h) h_0^{8-m}$
    \begin{align}
        \abs*{h_t - h_0} \leq \bar{I}^{\frac{1}{2}}(h) h_0^{\frac{m}{8}} t^{\frac{1}{8}} \leq h_0.
    \end{align}
    Then, we set $\tau_0 := 0$ and $\tau_1 := \bar{I}^{-4}(h) h_0^{8-m}$ and we observe that we have
    \begin{align}
        h_{\tau_1} \leq 2 h_0.
    \end{align}
    Inductively, we define $\tau_k := \tau_{k-1} + \bar{I}^{-4}(h) h_{\tau_{k-1}}^{8-m}$ for $k \in \N$. Then, it holds that
    \begin{align}
        \tau_k = \bar{I}^{-4}(h) \sum_{i=0}^{k-1} h_{\tau_i}^{8-m}
    \end{align}
    as well as (using \cref{timereg})
    \begin{equation}
        h_{\tau_i} \leq 2^i h_0.
        \label{eq:discretebound}
    \end{equation}
    Choosing $n := \lceil \log_2(h^{-1}_0)\rceil$ we have $h_t \geq  1$ only if $t \geq \tau_n$.   Note that if $m\geq 8$, we can apply~\eqref{eq:discretebound} to obtain $h_{\tau_i}^{8-m} \geq 2^{i(8-m)}h_{0}^{8-m}$. This tells us that
    \begin{align}
        \tau_n =& \bar{I}^{-4}(h) \sum_{i=0}^{n-1} h^{8-m}_{\tau_{i}}\\
        \geq& \bar{I}^{-4}(h) \log_2(h_0^{-1}) \, ,
    \end{align}
    for $m = 8$. The case $m > 8$ can be derived in an essentially identical manner.
\end{proof}
\section{Numerical experiments}\label{sec:num}
\subsection{Description of the time-stepping scheme}
We describe here the time-stepping scheme for the SDE \eqref{finaldstfematrix} with the \emph{Gr\"un--Rumpf} metric as described in \cref{grsection}. The central difference discretization (cf. \cref{sec:cdd}) is treated in an identical manner. 
For our simulations, we rely on a semi-implicit Euler--Maruyama method which treats the noise, It\^o-correction term, and metric tensor in~\eqref{finaldstfematrix} explicitly but treats the rest of the drift in an implicit manner. With $\Delta t>0$ denoting the time step, the scheme can be described as follows  
\begin{align}\label{timestepping}
\begin{cases}
    h_0 &= h \in \mathcal{M}_N  \\
    % h^i_{k+1} =& (\delta_{ll'} + \Delta t \at^l_{\al} \grm^{\al}_{\al'}(h_k)A^{\al'}_{n} \at^n_{\al''} A^{\al''}_{l'} )^{i j}h^j_k + \Delta t  \frac{1}{\beta} N \partial_j \at^i_{\al} \grm^{\al \al'}(h_k) \at^{j}_{\al'} \\ 
    % &+ \sqrt{2\Delta t  N \frac{1}{\beta}} \at^i_{\al} \sqrt{\grm}^{\al}_{\al'}(h_k) G^{\al'}_k
    h_{k+1} &= \bra*{\mathrm{Id} + \Delta t A^T G^{-1}(h_k) A A^T A }^{-1} \Big[ h_k +  \frac{\Delta t N}{\beta} A^T \bar{D} \cdot G^{-1}(h_k) \\
    &\quad+ \sqrt{\frac{2 N \Delta t}{\beta}} A^T \sqrt{G}^{-1}(h_k) W_k \Big]
\end{cases}
\end{align}
for all $k \in \N$, where $h_k$ denotes the vector of film heights at the nodal points $\bra*{x_i}_i$ and at time $k \Delta t$ and $ \bra*{W_k}_k$ is a sequence of independent $\mathcal{N}(0,I)$-distributed random vectors. We refer the reader to~\cref{comp} where we provide numerically stable expressions for the inverse metric and the It\^o-correction term. For the specific choice of $M(h) = h^3$ the inverse metric $G^{-1}(h_k)$ is computed at each time step using \eqref{metriccomp} and the It\^o-correction term $A^T\bar{D} \cdot G^{-1}(h_k)$ using~\eqref{itocnum}. Since $G^{-1}$ is a diagonal matrix, its square root can be computed explicitly. Due to the semi-implicit nature of the time-stepping scheme, in each step we have to compute the inverse of $\mathrm{Id} + \Delta t A^T G^{-1}(h_k) A A^T A $ which we do using the \texttt{MATLAB} function \texttt{mldivide}, which itself uses a Cholesky decomposition to perform the required matrix inversion.

\begin{figure}
    \centering
    \begin{minipage}{0.4\textwidth}
    \centering
    \includegraphics[width=1.1\linewidth]{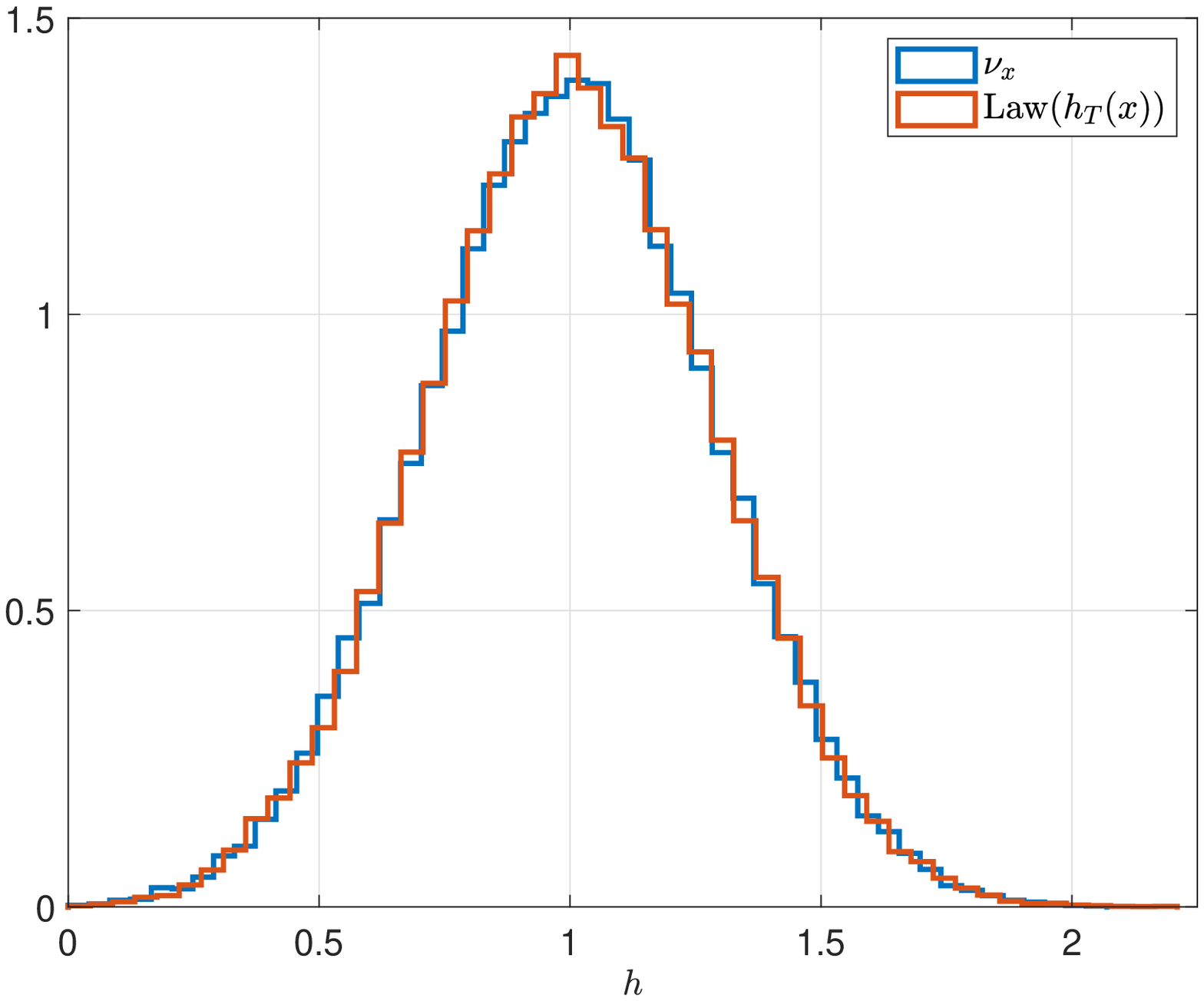}
    \subcaption{}
    \end{minipage}%
    \begin{minipage}{0.4\textwidth}
    \centering
    \includegraphics[width=1.1\linewidth]{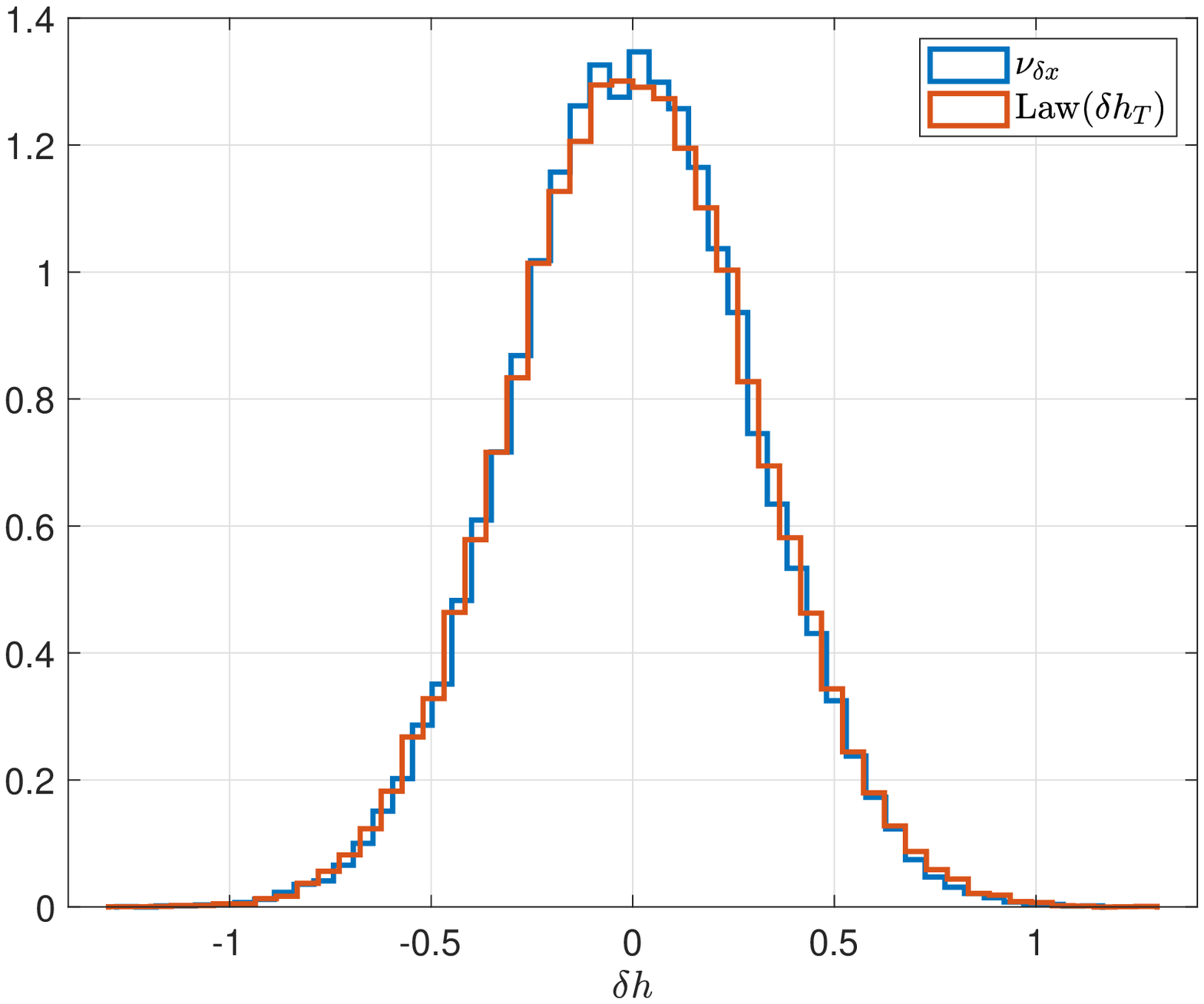}
    \subcaption{}
    \end{minipage}
    \begin{minipage}{0.4\textwidth}
    \centering
    \includegraphics[width=1.1\linewidth]{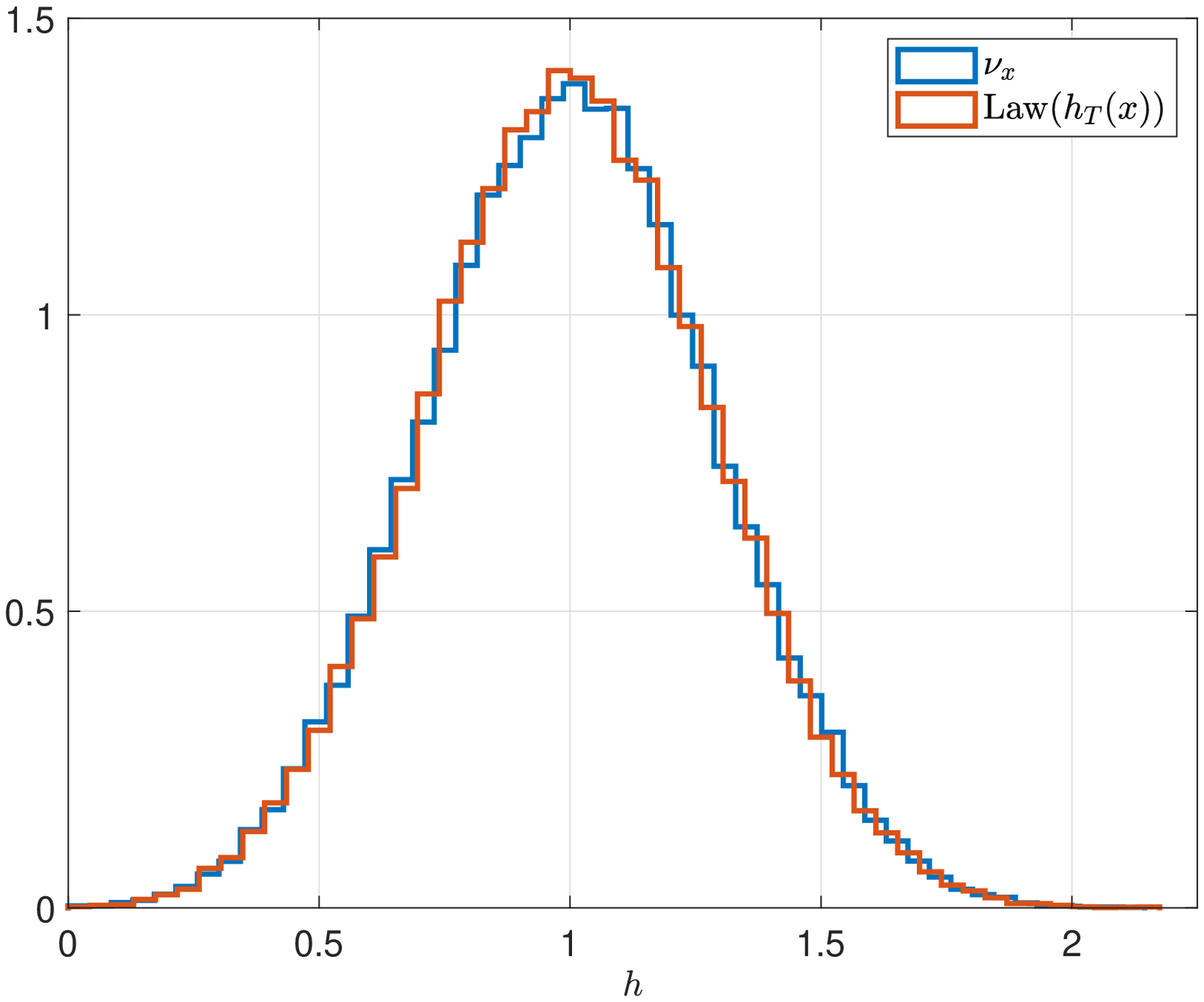}
    \subcaption{}
    \end{minipage}
    \begin{minipage}{0.4\textwidth}
    \centering
    \includegraphics[width=1.1\linewidth]{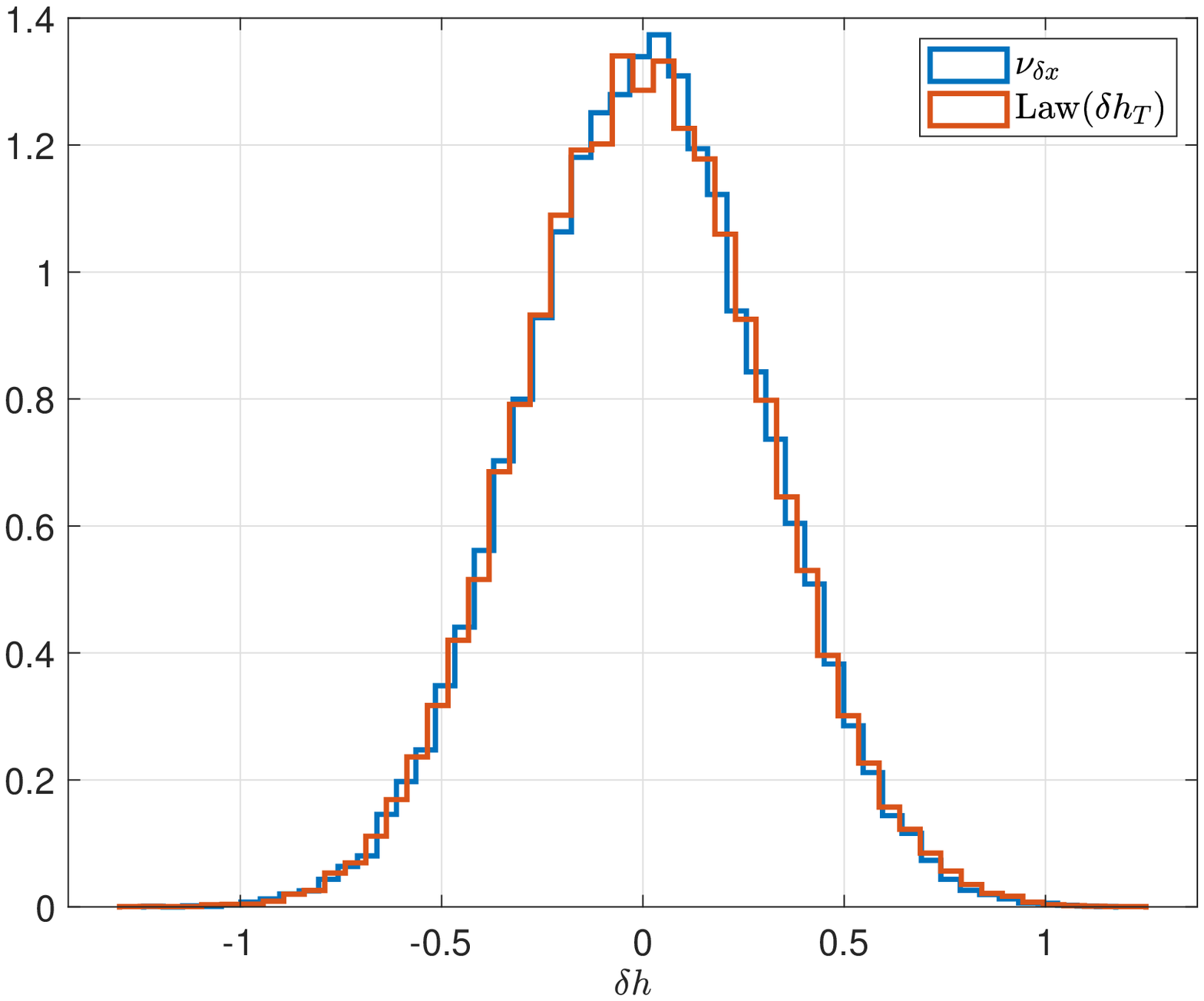}
    \subcaption{}
    \end{minipage}
    \caption{Plots of the histograms for $M=1000$ samples of the single-point statistics  and two-point correlations of the film height, i.e. $h_T$ and $\delta h_T=h_T(x+\delta x)-h_T(x)$, for the \emph{Gr\"un--Rumpf} ((A),(B)) and the central difference ((C),(D)) discretizations compared to the reference measure, the conservative Brownian excursion $\nu_N$. The simulations were carried out with the following parameters: $N=50$, $\Delta t =10^{-10}$, $\beta=1$, $T=10^{-3}$, $\delta x=0.1$, and $h_0\equiv 1$.}
    \label{fig:invariance}
\end{figure}
\subsection{Invariance of the measure $\nu_N$}\label{sec:num:inv}
In this subsection, we perform some numerical experiments to check the invariance of the measure $\nu_N$. We start by describing below a simple numerical procedure to sample from $\nu_N$.

\begin{algorithm}[H]
\SetAlgoLined
\KwResult{Realization of $\nu_N$}
    Sample discrete spatial white noise at temperature $\beta^{-1}$, i.e. a random $N$-dimensional vector of i.i.d. $\mathcal{N}(0,\beta^{-1}N \times  \mathrm{Id})$-distributed random variables $dW_N$\;
    Project onto average zero vectors: $\dx W_N^0=\dx W_N- N^{-1}\sum_i \dx W_{N,i}$\; 
    Integrate to get a discrete Brownian bridge: $W_{N,1}^0=0,\, W_{N,i}^0=W_{N,i-1}^0+N^{-1}\dx W_{N, i-1}^0$\;
    Project onto average $1$ vectors: $W_N=W_N^0 - N^{-1}\bra*{\sum_i W_{N,i}^0}  +1 $\;
  \eIf{$\exists i$ s.t. $W_{N,i}<0$}{
   reject\;
   }{
   accept\;
 }
 \caption{Sampling from $\nu_N$}
 \label{nualgo}
\end{algorithm}
%\begin{enumerate}[label=Step \arabic*:]
%\item Sample discrete spatial white noise, i.e. a random vector of i.i.d. \\
%        $N(0,1)-$distributed random variables
%\item Project onto mean free vectors
%\item Integrate to get a shifted Brownian bridge
%\item Project onto mass $1$
%\item Accept or reject if positive or not.
%\end{enumerate}
We now integrate in time  starting from $h_0 \equiv 1$ according to the semi-implicit Euler--Maruyama algorithm described in \eqref{timestepping} up to some large time $T \gg \Delta t$.   Repeating this procedure, we obtain a large number of samples, $M\gg1$, of the process at time $t=T$ which we compare to the samples of $\nu_N$ generated by~\cref{nualgo}. Note that $T$ needs to be chosen to be larger than the typical relaxation time (to the invariant measure) of both discretizations. We found that $T=10^{-3}$ works well for this purpose. We compare both the single-point distributions and the two-point correlations, i.e. the law of $\delta h_T = h_T(x+ \delta x)-h_T(x)$ for some $\frac{1}{N}=: \Delta x \ll \delta x \ll 1$. Due to the stationarity (in space) of the invariant measure the choice of $x \in \pra*{0,1}$ is irrelevant.  We present the results of this experiment in~\cref{fig:invariance}. 

\subsection{Positivity, exit times, and entropic repulsion} \label{sec:numpositivity}
\begin{figure}[h!]
    \centering
    \begin{minipage}{0.5\textwidth}
    \includegraphics[width=1.1\linewidth]{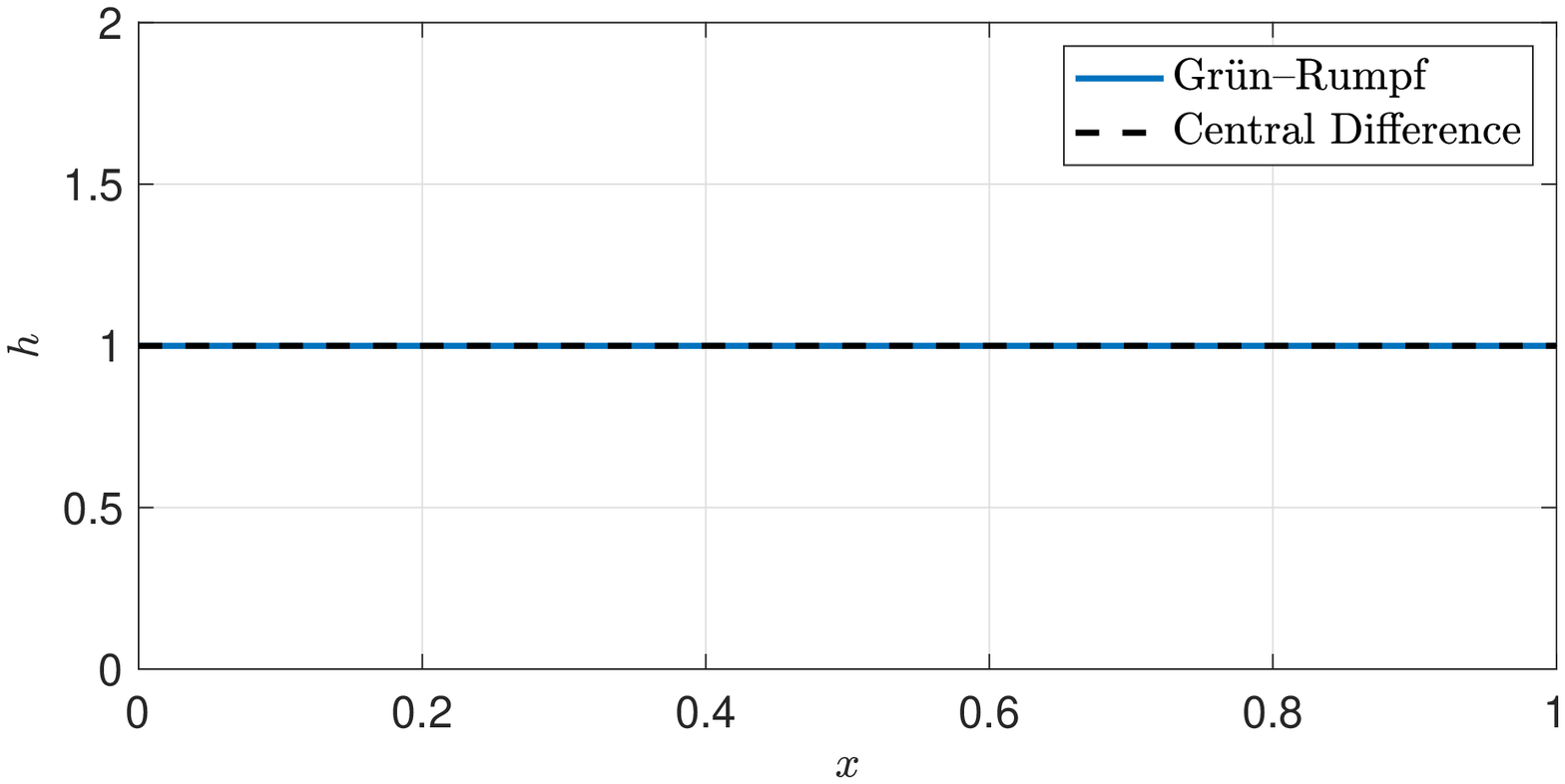}
    \subcaption{$t=0$}
    \end{minipage}%
    \begin{minipage}{0.5\textwidth}
    \includegraphics[width=1.1\linewidth]{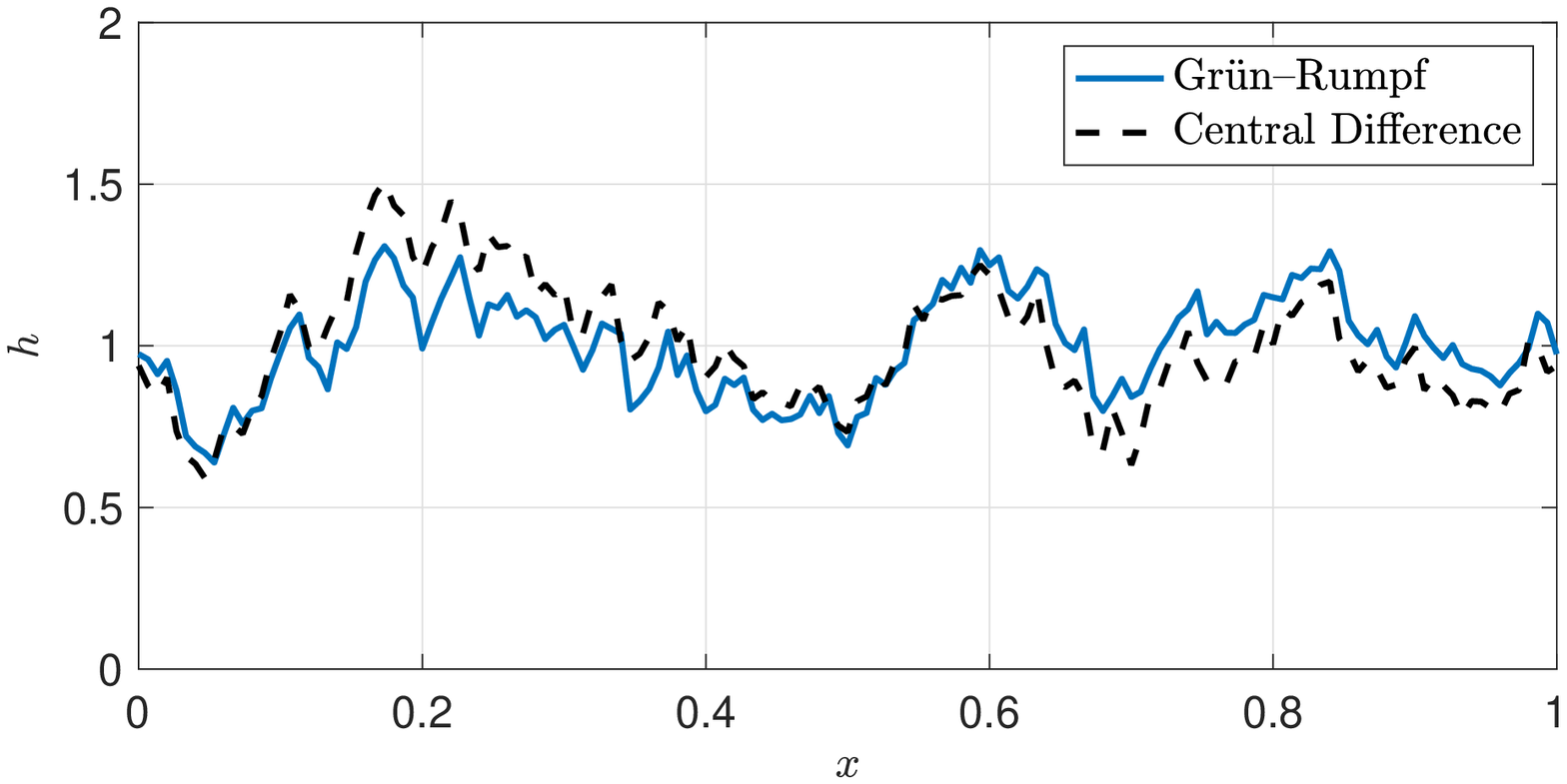}
    \subcaption{ $t\approx 1.64 \times 10^{-4}$}
    \end{minipage}%
    
    \begin{minipage}{0.5\textwidth}
    \includegraphics[width=1.1\linewidth]{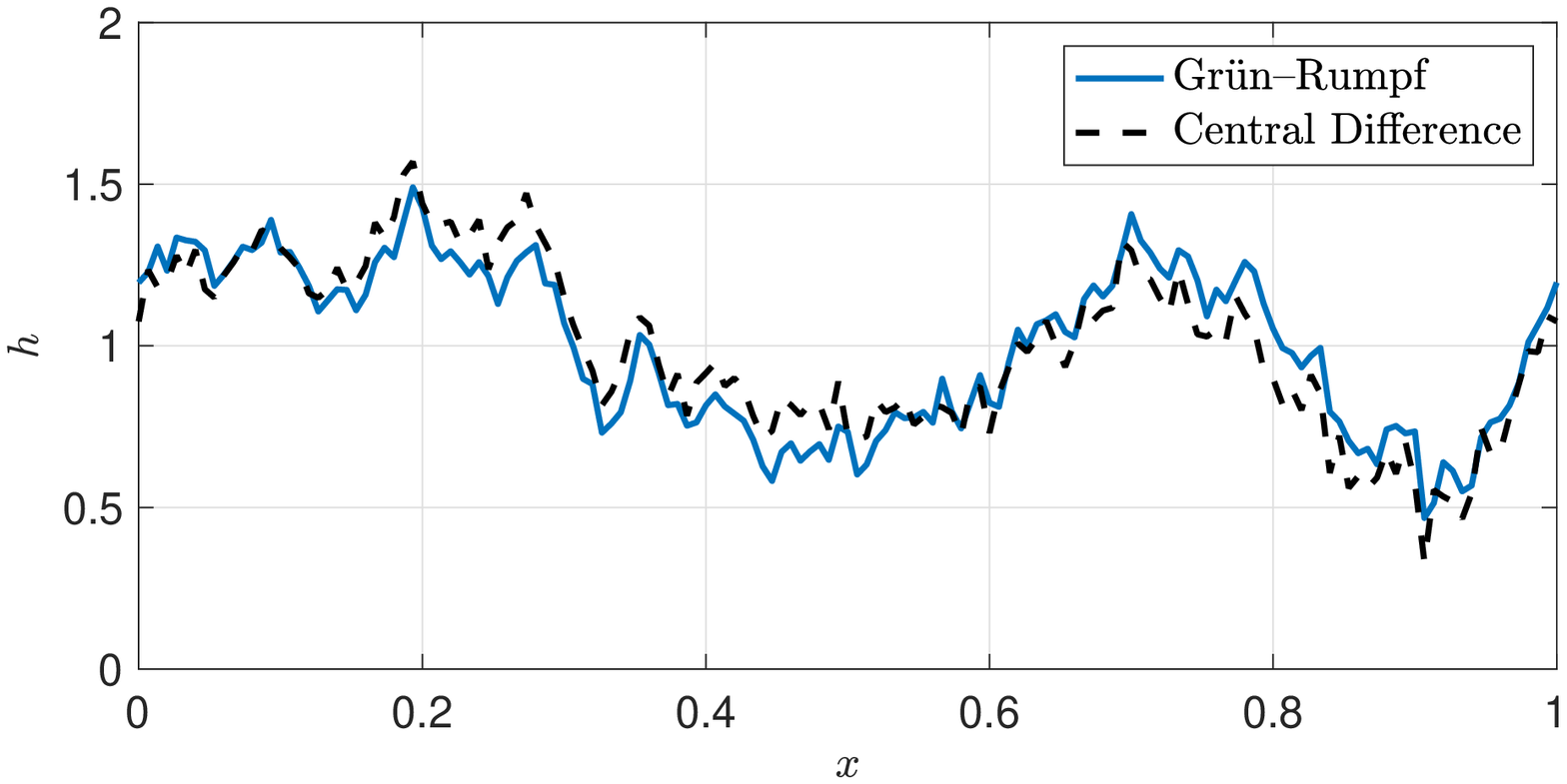}
    \subcaption{$t\approx 3.28 \times 10^{-4}$}
    \end{minipage}%
    \begin{minipage}{0.5\textwidth}
    \includegraphics[width=1.1\linewidth]{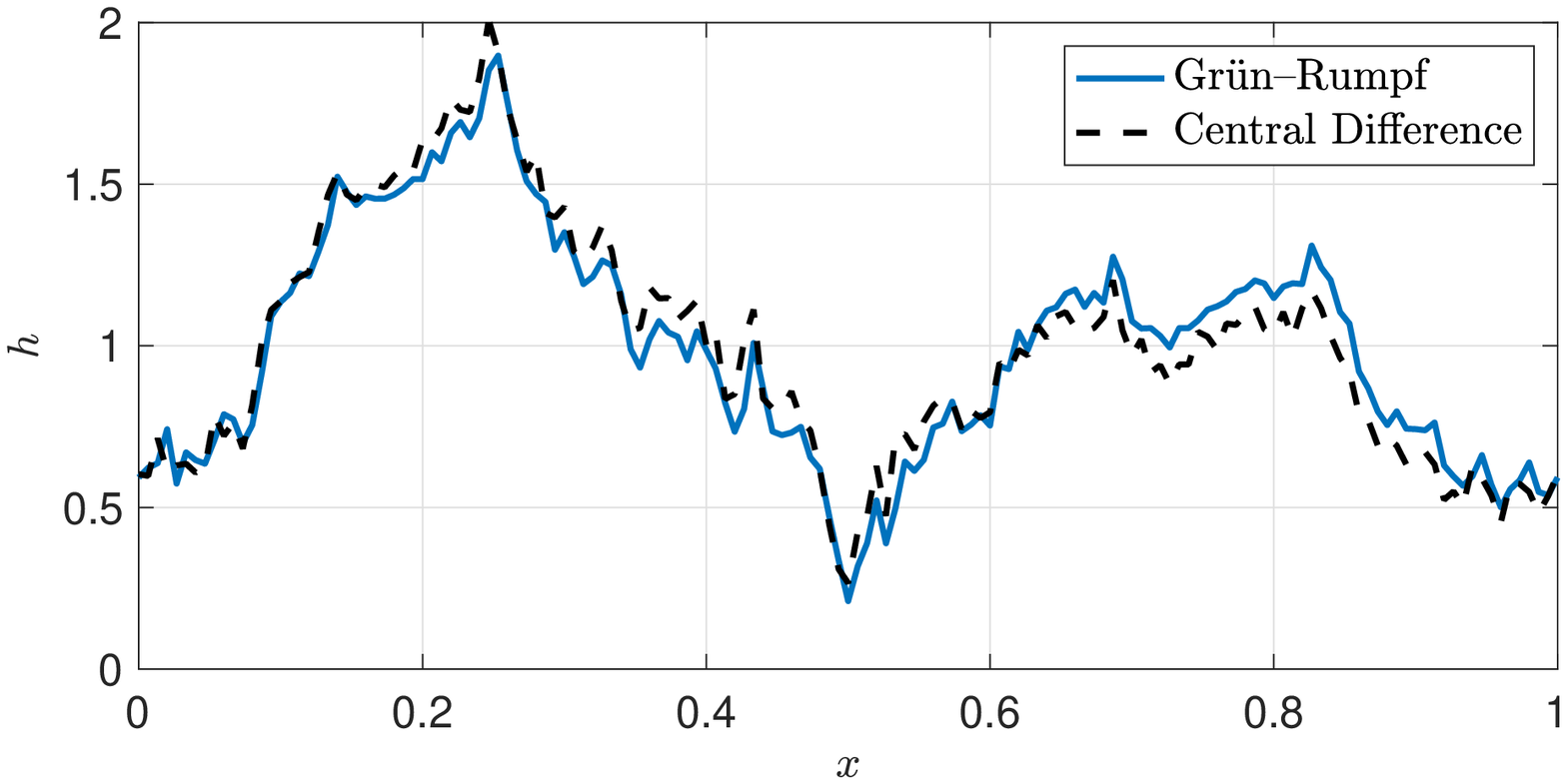}
    \subcaption{$t\approx 4.92 \times 10^{-4}$}
    \end{minipage}%
    
    \begin{minipage}{0.5\textwidth}
    \includegraphics[width=1.1\linewidth]{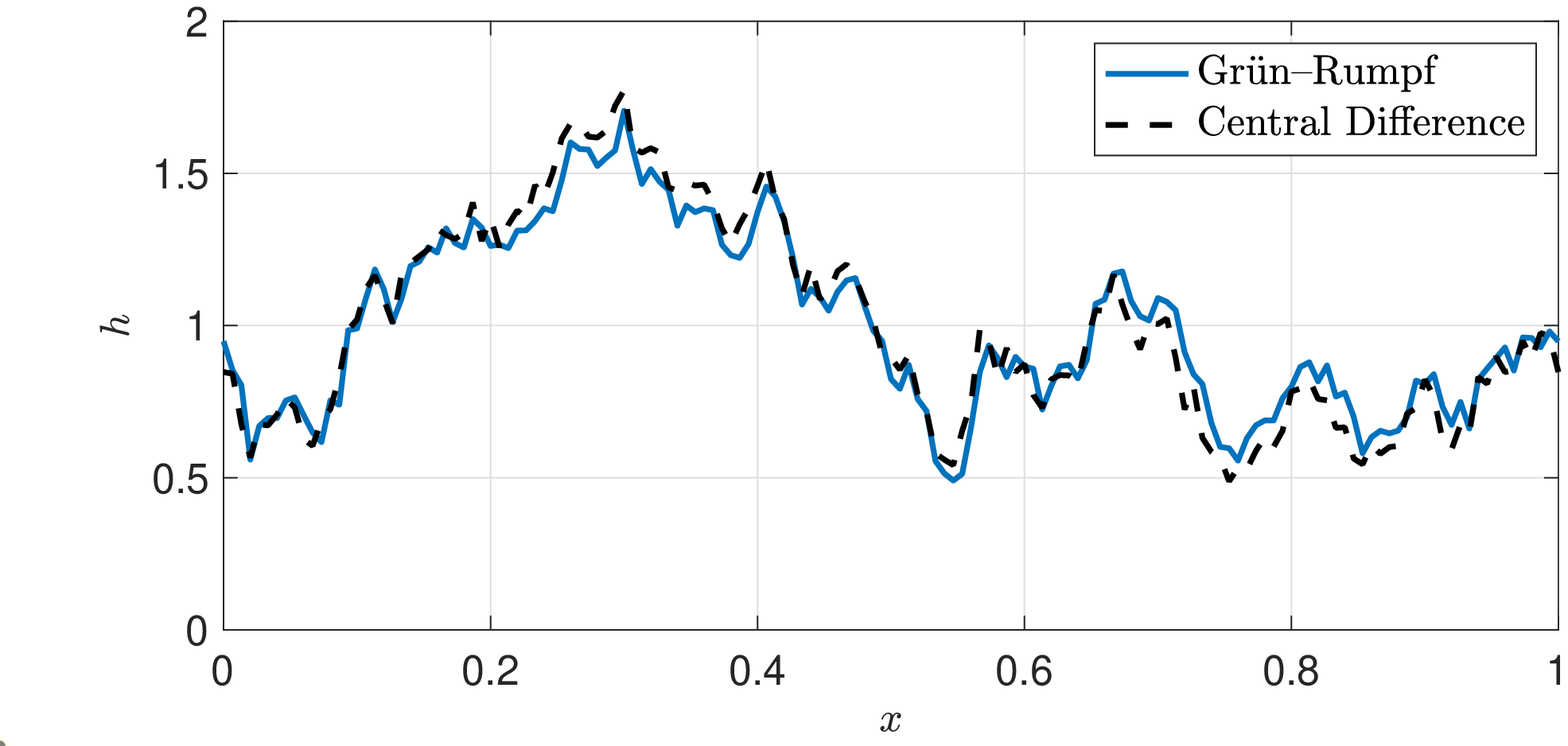}
    \subcaption{$t\approx 6.56 \times 10^{-4}$}
    \end{minipage}%
    \begin{minipage}{0.5\textwidth}
    \includegraphics[width=1.1\linewidth]{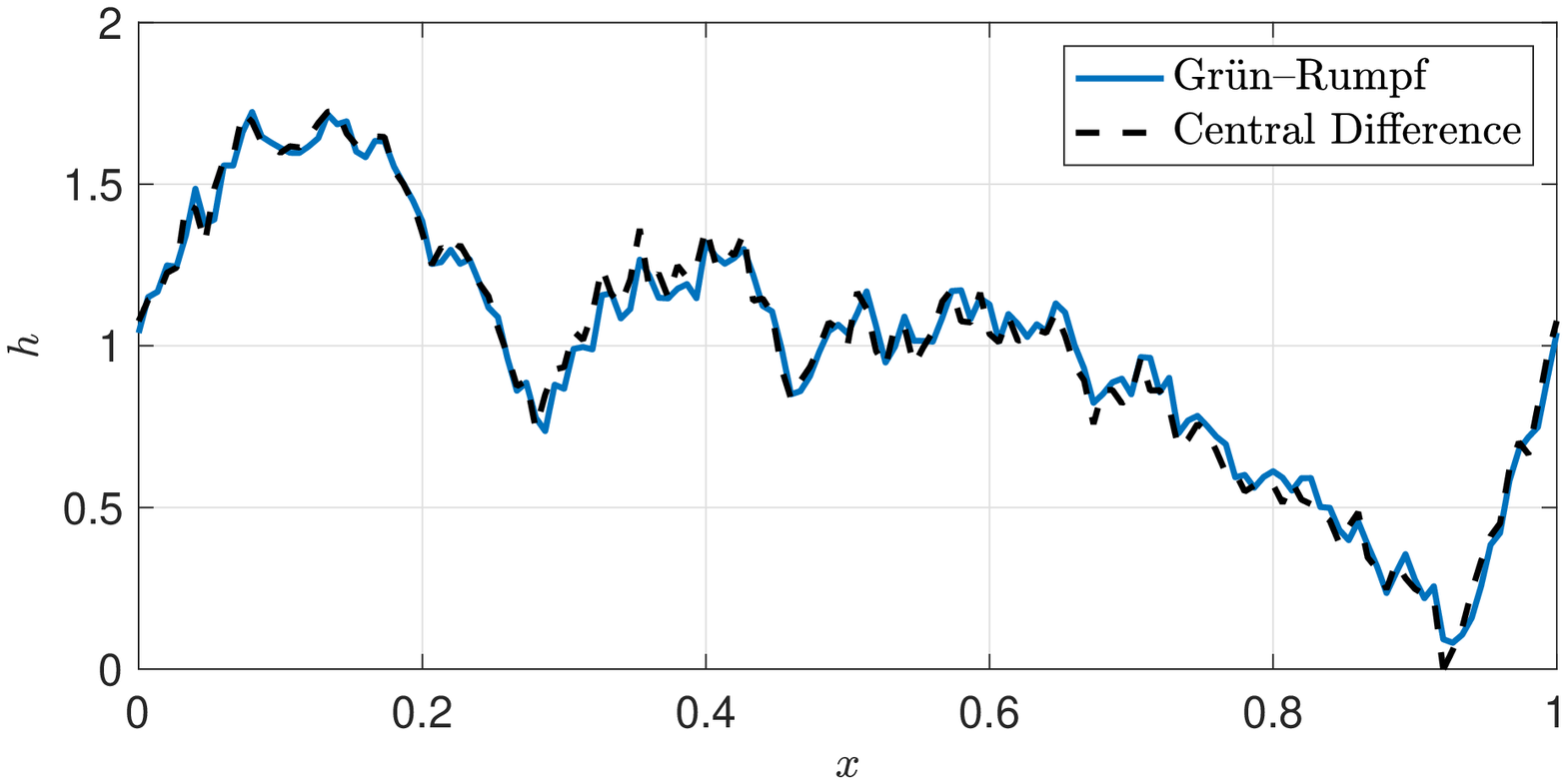}
    \subcaption{$t\approx 8.2 \times 10^{-4}$}
    \end{minipage}%
    \caption{Snapshots of the film height for the \emph{Gr\"un--Rumpf} and central difference discretizations at equally spaced time increments (time goes from (A) $\to$ (F)) for the same realization of the noise. As can be seen from the figures, the central difference discretization touches down (at $t_*\approx 8.2 \times 10^{-4}$, see (F)) while the \emph{Gr\"un--Rumpf} discretization stays away from the boundary. The simulations were performed with the following parameters: $N=150$, $\Delta t =10^{-10}$, $\beta=1$, and  $h_0 \equiv 1$.}
    \label{fig:touchdown}
\end{figure}

As shown in~\cref{stentest}, under appropriate conditions on the initial datum, the \emph{Gr\"un--Rumpf} discretization stays away from the boundary $\partial \mathcal{M}_N$. On the other hand, one expects (see the discussion in~\cref{sec:cdd}) the central difference discretization to touch the boundary with probability 1. We provide some numerical evidence for these features of the two discretizations in~\cref{fig:touchdown}. Indeed, for the same realization of the noise, the\emph{ Gr\"un--Rumpf} discretization stays away from $0$, while the central difference discretization touches down.

\medskip

We can provide stronger numerical evidence for the fact that the central difference discretization touches down by computing the mean exit time from $\mathcal{M}_N$ of the associated process. If this quantity is finite, this implies that the central difference discretization leaves $\mathcal{M}_N$, i.e. touches down, almost surely. Let $h^{h_0}_t$ be a solution of the central difference discretization of the stochastic thin-film equation \eqref{dSTFEnaive} with initial condition $h_0 \in \mathcal{M}_N$. Then, we define the exit time  of $h_t^{h_0}$ from the interior to be
\begin{align}
    \tau(h_0) := \inf\set*{t \geq 0 : h^{h_0}_t \notin \mathcal{M}_N}.
\end{align}
We take $h_0 \equiv 1$ and set $\tau := \tau(1)$. Then, we sample $\tau$ by running a Monte-Carlo simulation of \eqref{dSTFEnaive}  according to the time-stepping scheme described in~\eqref{timestepping}. This time, instead of imposing reflecting boundary conditions, we stop the simulation as soon as we reach the boundary $\partial \mathcal{M}_N$, i.e. when the film touches down.
~\cref{fig:my_label2} shows the behavior of the mean exit time as $N$ grows. In particular, it seems that the mean exit time is finite and remains bounded as $N$ tends to infinity.
\begin{figure}[h!]
    \centering
    \includegraphics[scale=0.45]{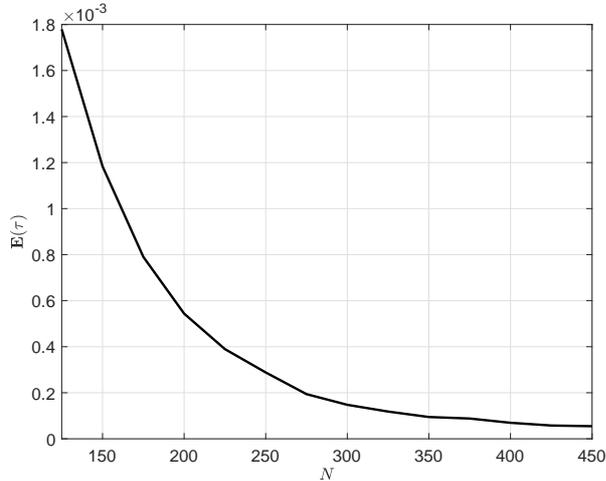}
    \caption{The dependence of the mean exit time  of the central difference discretization on $N$. The simulations were performed with the following parameters: $\Delta t=10^{-10}$, $\beta=1$, $M=100$, and $h_0 \equiv 1$.}
    \label{fig:my_label2}
\end{figure}
In the final part of this subsection, we study numerically the positivity properties of the continuum conservative Brownian excursion $\nu$, i.e. its entropic repulsion. As has been mentioned before, our conservative Brownian excursion is qualitatively similar to the classical Brownian excursion from stochastic analysis. Moreover, it is known that the classical Brownian excursion features an entropic repulsion, in the sense that the single point distribution decays to $0$ at $0$. In fact, one can compute the single point statistics for the classical Brownian excursion $\bra*{Y_t}_{t \geq 0}$ explicitly (cf. \cite[p.463]{yorrevuz}): For fixed $t \geq 0$ and $x, y > 0$ such that $Y_0 = x$ and $Y_T = y$ a.s., it takes the form
\begin{align}
    p^{x, y}_t(z) = \frac{T}{t(T-t)} z \frac{I_{\frac{1}{2}}(\frac{xz}{t}) I_{\frac{1}{2}}(\frac{zy}{T-t})}{I_{\frac{1}{2}}(\frac{xy}{T})} e^{-\frac{x^2 + z^2}{2t}} e^{-\frac{z^2 + y^2}{2(T-t)}} e^{\frac{x^2 + y^2}{2T}}
\end{align}
where $I_{\frac{1}{2}}$ is the modified Bessel function of the first kind of order $\frac{1}{2}$. Notice that for $z \ll 1$, it holds that $I_{\frac{1}{2}}(z) \sim z^{\frac{1}{2}}$.  
From the above expression, it is clear that the distribution decays to $0$ quadratically as $z \to 0$.  In ~\cref{fig:entropicrepulsion} we see that the single point distribution of our conservative Brownian excursion for $N \gg 1$ also exhibits quadratic decay at $0$.
\begin{figure}[h!]
    \centering
    \includegraphics[scale=0.5]{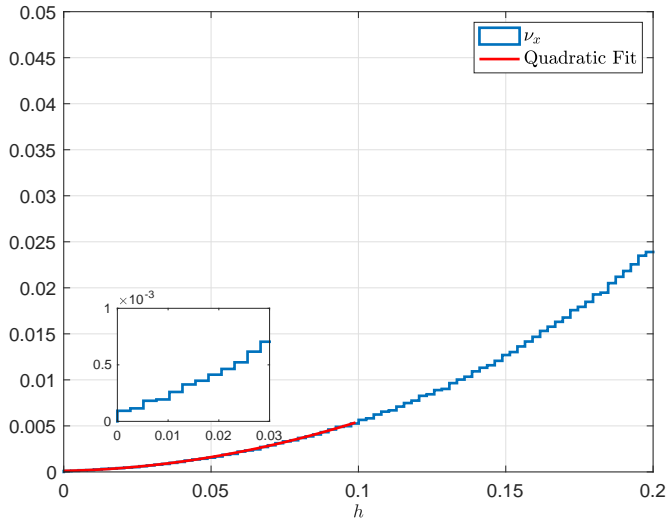}
    \caption{The entropic repulsion of the continuum conservative Brownian excursion $\nu$ as observed through the single point statistics of $\nu_N$ for $N$ large ($=2000$) obtained from $M=2\times 10^5$ samples. The single point distribution (in blue) decays quadratically as $h \to 0$ as can be seen by comparing it to the fitted curve (in red) $p(h) \approx 0.4704 \times h^2$. The zoomed-in version of the histogram exhibits the fact that entropic repulsion is a feature of the continuum invariant measure; for finite but large $N$ the single point density is positive but small at $0$.}
    \label{fig:entropicrepulsion}
\end{figure}
\subsection{Convergence of the two discretizations}\label{conv2}

\begin{figure}[h!]
    \centering
    \begin{minipage}{0.5\textwidth}
    \centering
    \includegraphics[width=1.1\linewidth]{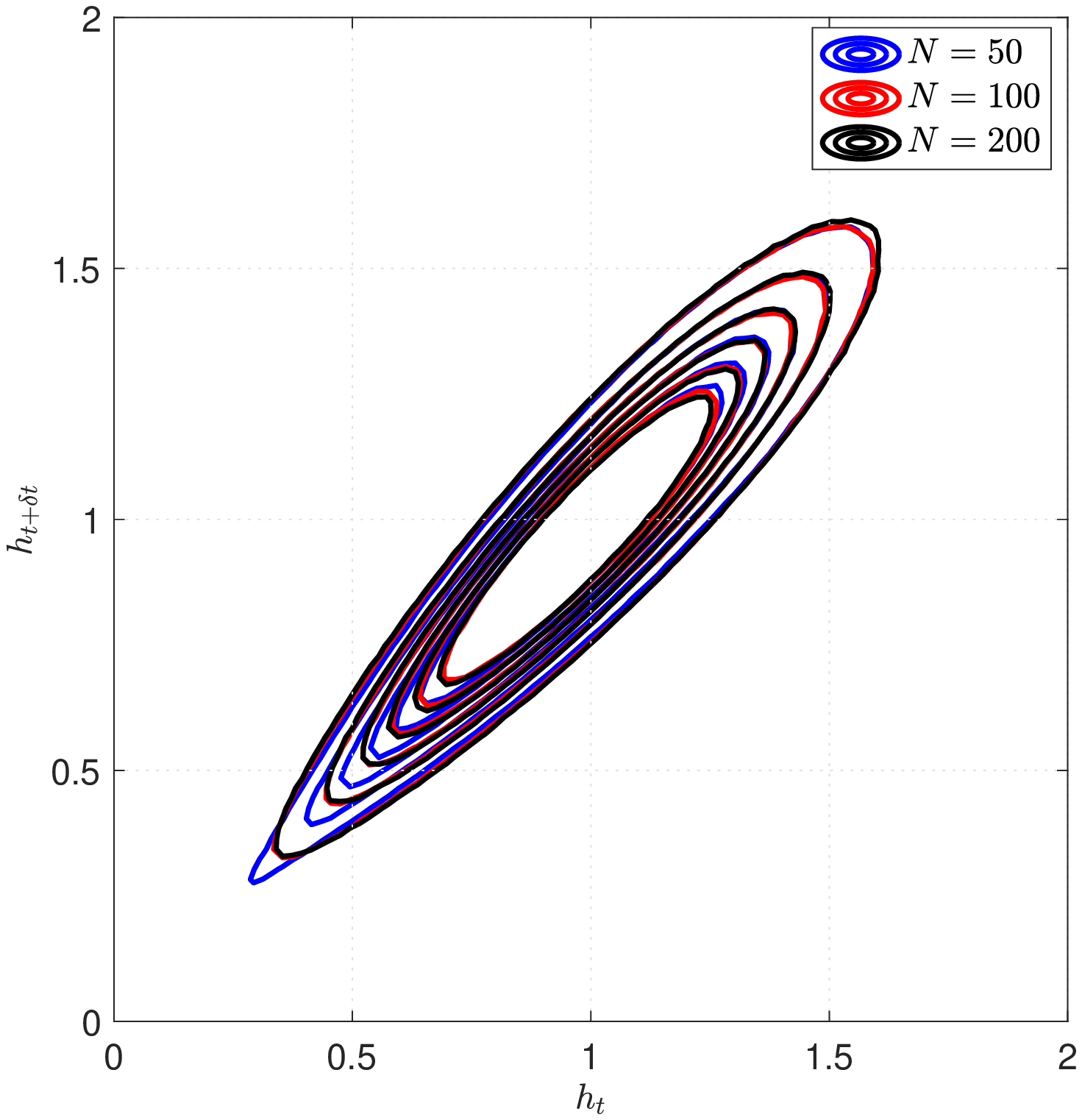}
    \subcaption{}
    \end{minipage}%
    \begin{minipage}{0.5\textwidth}
    \centering
    \includegraphics[width=1.1\linewidth]{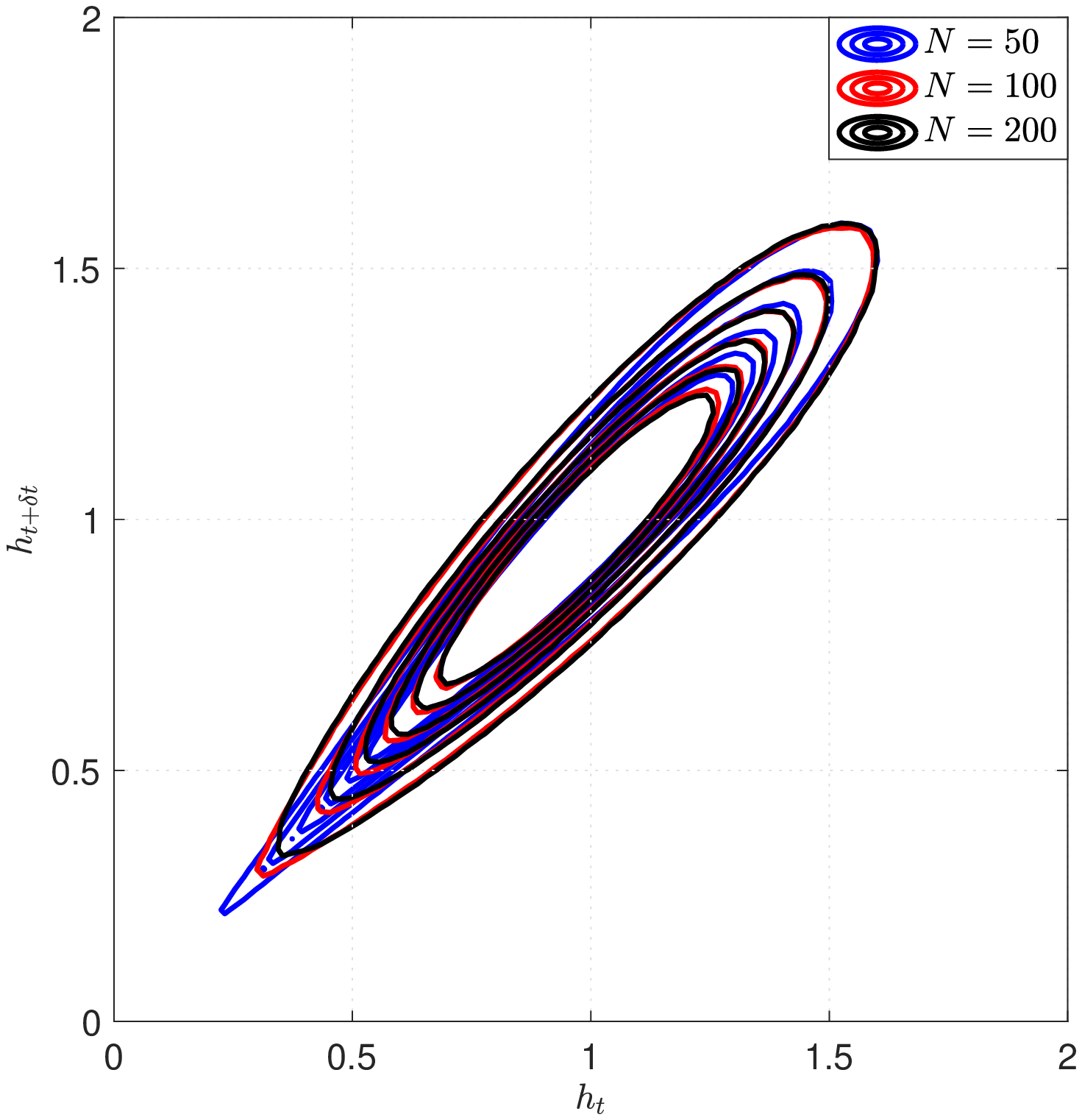}
    \subcaption{}
    \end{minipage}
    \caption{Level sets of the two-point (in time) distributions, i.e. the joint  distributions of $h_t$ and $h_{t + \delta t}$, for (A) the \emph{Gr\"un--Rumpf} and (B) the central difference  discretizations for $N=50,100,200$.}
    \label{fig:pathspace1}
\end{figure}

\begin{figure}[h!]
    \centering
    \begin{minipage}{0.5\textwidth}
    \centering
    \includegraphics[width=1.1\linewidth]{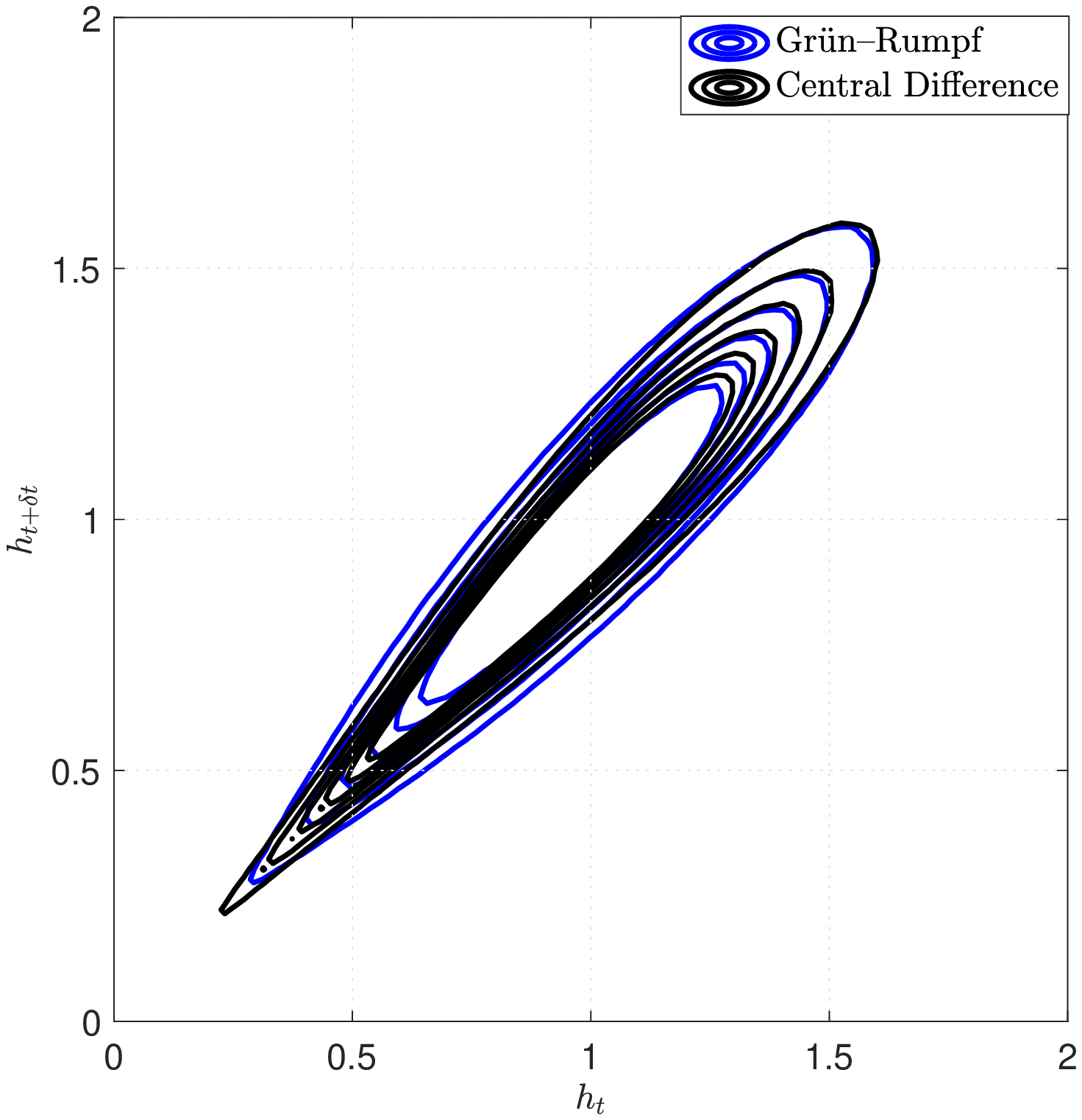}
    \subcaption{}
    \end{minipage}%
    \begin{minipage}{0.5\textwidth}
    \centering
    \includegraphics[width=1.1\linewidth]{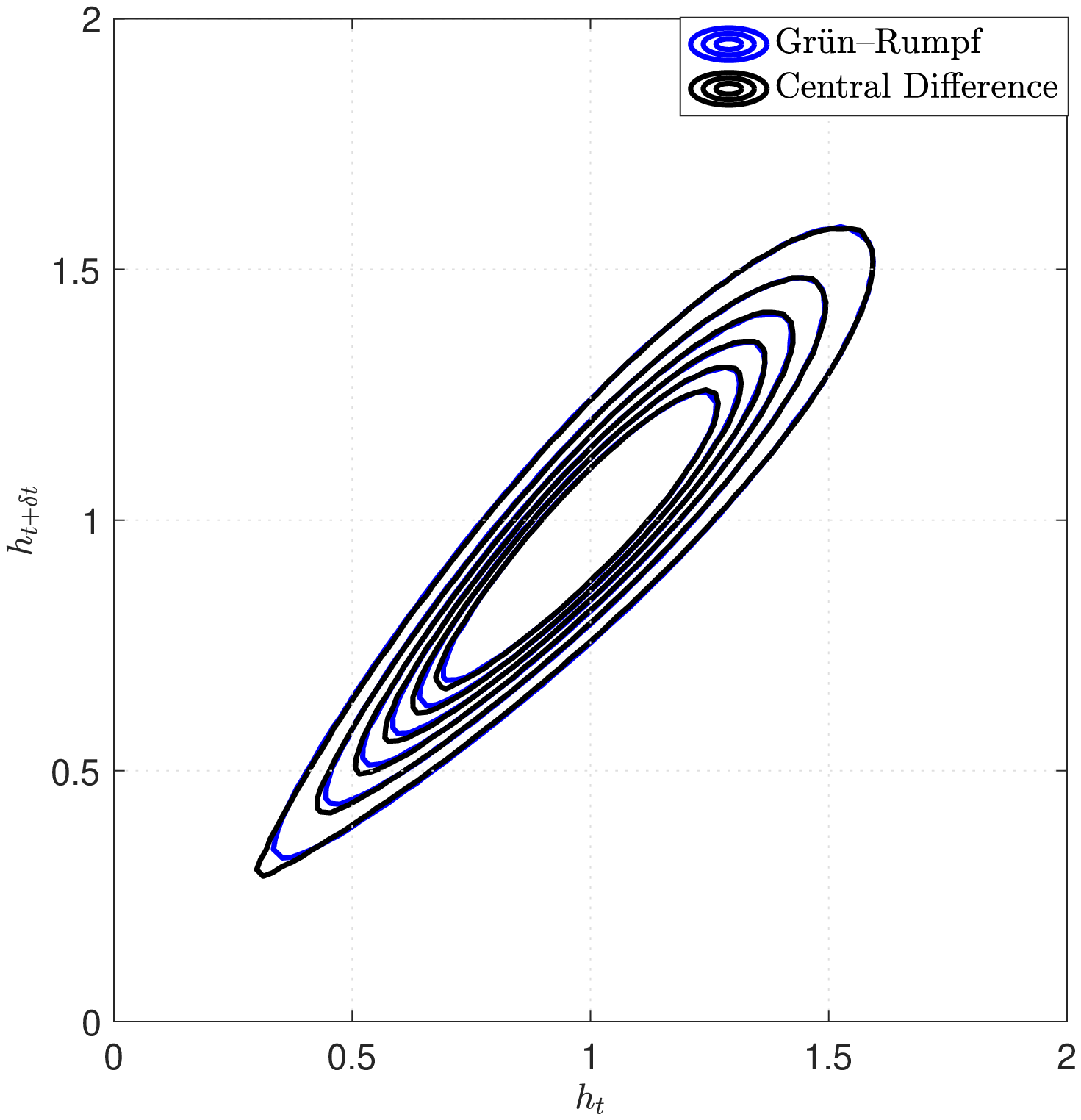}
    \subcaption{}
    \end{minipage}
        \begin{minipage}{0.5\textwidth}
    \centering
    \includegraphics[width=1.1\linewidth]{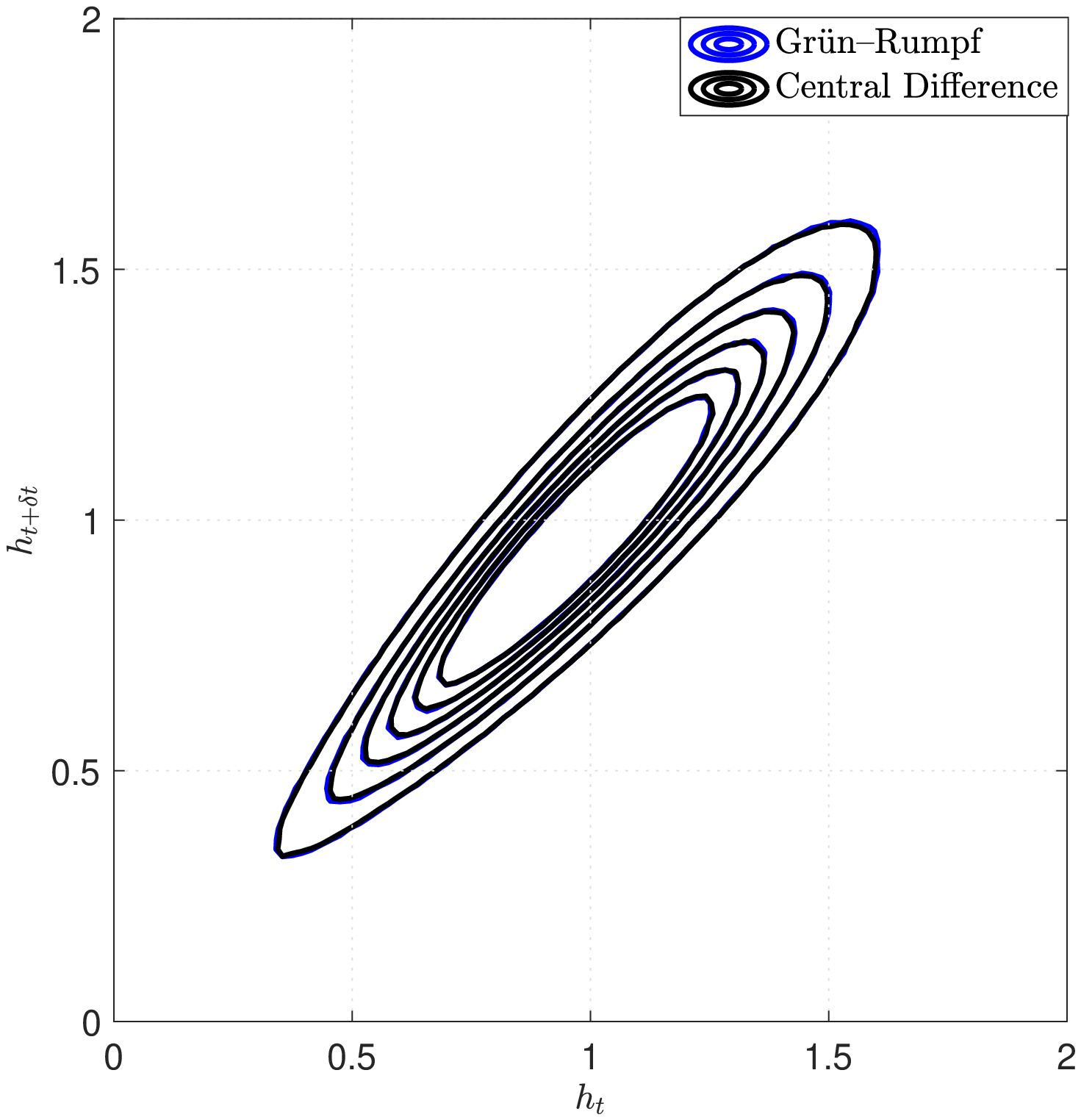}
    \subcaption{}
    \end{minipage}
    \caption{Comparisons of the level sets of the two-point (in time) distributions of the the \emph{Gr\"un--Rumpf} and the central difference  discretizations  for (A) $N=50$, (B) $N=100$, and (C) $N=200$.}
    \label{fig:pathspace2}
\end{figure}

As mentioned earlier in the paper, two different discretizations of a singular SPDE can converge to different limiting objects (cf. \cite{HairerMaasAoP2012}). Thus, it would not be unreasonable to expect that the \emph{Gr\"un--Rumpf} and central difference discretizations of the thin-film equation with thermal noise have different continuum limits. However, numerical evidence seems to indicate that, at least started at equilibrium, the path space measures of the two discretizations converge to the same object. 

\medskip

We check this by sampling from $\nu_N$ using~\cref{nualgo} and then integrating in time with $h_0 \sim \nu_N$ to some final time $T$. Repeating this process, we obtain a large number, $M \gg 1$, of samples. We can then  compute the two-point (in time) distributions of both discretizations, i.e. the joint law of $h_t$ and $h_{t + \delta t}$ for some $\Delta t \ll \delta t \ll T$, for different values of $N$.  One then observes that, as $N$ increases, the two discretizations seem to converge to each other. Note that since we start our simulations at the invariant measure and the underlying process is reversible the choice of $t \geq 0$ is irrelevant. We present the results of these experiments in~\cref{fig:pathspace1,fig:pathspace2}.

\section*{Acknowledgements}

This work was funded by the Deutsche Forschungsgemeinschaft (DFG, German Research Foundation) – SFB 1283/2 2021 – 317210226.

\appendix 

\section{The thin-film equation with linear mobility in Lagrangian coordinates}
% It is well known (cf. \cite[Theorem 2.18, p.74]{CV03}) that for linear mobility, i.e. $m = 1$, under the change of variables $h \mapsto X$ where $X$ is the inverse distribution function of $h$, i.e.
% \begin{align}\label{Xdef}
%     z = \int_0^{X(z)} h(x) \dx{x}
% \end{align}
% we have that
% \begin{align}
%     g_h\bra*{\dot{h}, \dot{h}} = \int_{0}^1 \bra*{\dot{X}}^2 \dx z =: g_X\bra*{\dot{X}, \dot{X}}.
% \end{align}
Let 
\begin{align}\label{Xdef}
    z = \int_0^{X(z)} h(x) \dx{x}
\end{align}
then taking the derivative twice with respect to $z$ of \eqref{Xdef} yields
\begin{align}\label{formulah}
    1 = h(X(z)) \frac{\dx}{\dx z} X(z)
\end{align}
as well as
\begin{align}\label{formuladxh}
    0 = \partial_xh(X(z)) \bra*{\frac{\dx}{\dx z}X(z)}^2 + h(X(z)) \frac{\dx^2}{\dx z^2} X(z).
\end{align}
Multiplying \eqref{formuladxh} with $h(X(z))^2$ and invoking \eqref{formulah} we end up with
\begin{align}\label{formuladxh2}
    \partial_x h(X(z)) = - h(X(z))^3  \frac{\dx^2}{\dx z^2} X(z).
\end{align}
Hence we compute for the Dirichlet energy
\begin{align}
    E(h)&:= \frac{1}{2} \int_0^1 (\partial_x h)^2 \dx x = \frac{1}{2} \int_0^1 (\partial_x h(X(z)))^2 \frac{\dx}{\dx z}X(z) \dx z \\
    &\stackrel{\eqref{formuladxh2}}{=} \frac{1}{2} \int_0^1 \bra*{h(X(z))^3 \frac{\dx^2}{\dx z^2} X(z)}^2 \frac{\dx}{\dx z}X(z) \dx z \\
    &\stackrel{\eqref{formulah}}{=} \frac{1}{2} \int_0^1\frac{\bra*{\frac{\dx^2}{\dx z^2} X(z)}^2}{\bra*{\frac{\dx}{\dx z}X(z)}^5} \dx z \\
    &=: E(X).
\end{align}
Moreover, for some $\delta X$ we compute
\begin{align}
    \mathrm{diff}E|_X.\delta X &= \frac{1}{2} \int_0^1 2 \frac{\frac{\dx^2}{\dx z^2} X(z)}{\bra*{\frac{\dx}{\dx z}X(z)}^5} \frac{\dx^2}{\dx z^2}(\delta X(z)) - 5 \frac{\bra*{\frac{\dx^2}{\dx z^2} X(z)}^2}{\bra*{\frac{\dx}{\dx z}X(z)}^6} \frac{\dx}{\dx z}(\delta X(z)) \dx z \\
    &= \int_0^1 \bra*{\frac{\dx^2}{\dx z^2}\bra*{\frac{\frac{\dx^2}{\dx z^2} X(z)}{\bra*{\frac{\dx}{\dx z}X(z)}^5}} + \frac{5}{2} \frac{\dx}{\dx z}\bra*{\frac{\bra*{\frac{\dx^2}{\dx z^2} X(z)}^2}{\bra*{\frac{\dx}{\dx z}X(z)}^6}}} \delta X(z) \dx z.
\end{align}
This, as usual, gives rise to the $L^2$-gradient flow
\begin{align}
    \partial_t X &= - \partial_z^2\bra*{\frac{\partial_z^2 X}{\bra*{\partial_z X}^5}} - \frac{5}{2} \partial_z \bra*{\frac{\bra*{\partial_z^2 X}^2}{\bra*{\partial_z X}^6}} \\
    &= \frac{1}{4} \partial_z^3 \bra*{\partial_z X}^{-4} - \frac{5}{8} \partial_z \bra*{\partial_z \bra*{\partial_z X}^{-2}}^2.
\end{align}
% Hence the gradient flow
% \begin{align}
%     \partial_t h = - \partial_x \bra*{h \partial_x^3h} = -\nabla E(h)
% \end{align}
% can be written as the $L^2$-gradient flow
% \begin{align}
%     \partial_t X = - \partial_z^2\bra*{\frac{\partial_z^2 X}{\bra*{\partial_z X}^5}} - \frac{5}{2} \partial_z \bra*{\frac{\bra*{\partial_z^2 X}^2}{\bra*{\partial_z X}^6}}
% \end{align}
% and by the fluctuation dissipation theorem we end up with
% \begin{align}\label{lagstfe}
%     \partial_t X &= - \partial_z^2\bra*{\frac{\partial_z^2 X}{\bra*{\partial_z X}^5}} - \frac{5}{2} \partial_z \bra*{\frac{\bra*{\partial_z^2 X}^2}{\bra*{\partial_z X}^6}} + \xi \\
%     &= \frac{1}{4} \partial_z^3 \bra*{\partial_z X}^{-4} - \frac{5}{8} \partial_z \bra*{\partial_z \bra*{\partial_z X}^{-2}}^2 + \xi
% \end{align}
% where $\xi$ is space-time white noise. The first term on the right hand side of \eqref{lagstfe} is clearly well-defined. For the second term on the right hand side of \eqref{lagstfe}, we notice that it is a ''KPZ''-like term followed by a derivative. Since the renormalization constant for the KPZ equation does not depend on the space variable (cf. \cite[Thm. 15.1, p. 223]{FH14}) we might expect that in our case it is killed off by the outer derivative. This heuristic is confirmed by numerical simulations even for $m=3$.

\section{Computing the change of coordinates}
% By defining $\zeta = \bar{\zeta} \circ \Psi$, $f_t = \bar{f}_t \circ \Psi$ as well as $\bar{h} = \Psi(h)$, we compute using the chain rule
% \begin{align}\label{compco1}
%     g_h\bra*{\nabla \zeta(h), \nabla f_t(h)} &= \mathrm{diff}\zeta|_h.\nabla f_t(h) \\ 
%     &= \partial_{\al} \bar{\zeta}(\bar{h}) \mathrm{diff}\varphi^{\al}|_h.\nabla f_t(h) \\
%     &= \partial_{\al} \bar{\zeta}(\bar{h}) \mathrm{diff}f_t|_h.\nabla \varphi^{\al}(h) \\
%     &= \partial_{\al} \bar{\zeta}(\bar{h}) \bar{\partial}_{\al'} \bar{f}_t(\bar{h}) \mathrm{diff}\varphi^{\al'}|_h.\nabla \varphi^{\al}(h) \\
%     &= \partial_{\al} \bar{\zeta}(\bar{h}) \partial_{\al'} \bar{f}_t(\bar{h}) g^{\al \al'}(\bar{h}).
% \end{align}
\subsection{The dual metric in coordinates}
Let the setting be as in the beginning of \cref{coord}. As usual, we define the musical isomorphism via
\begin{align}
    T^{*}\mathcal{M} \to T\mathcal{M}, \omega \to \omega^{\sharp}
\end{align}
where 
\begin{align}\label{riesz}
    \omega.\dot{h} = g\bra*{\omega^{\sharp}, \dot{h}}
\end{align}
for all $\dot{h} \in T\mathcal{M}$.
This gives rise to the dual metric $g'$ on $T^{*}\mathcal{M} \otimes T^{*}\mathcal{M}$ via
\begin{align}\label{dualmetric}
    g'\bra*{\omega, \omega'} := g\bra*{\omega^{\sharp}, \omega'^{\sharp}}
\end{align}
for all $\omega, \omega' \in T^{*}\mathcal{M}$. Let $g'^{\al \al'}$ and $g_{\al \al'}$ be the representation of $g'$ respectively $g$ in the coordinates $\bra*{\varphi^{\al}}_{\al}$ and let $\ell, \ell'$ be covectors and $\tau, \tau'$ be vectors that are related by 
\begin{align}\label{rel}
    \ell_{\al} = g_{\al \al'} \tau^{\al'}, \quad \ell'_{\al} = g_{\al \al'} \tau'^{\al'}.
\end{align}
Then by definition of \eqref{dualmetric} and by \eqref{rel}, we have
\begin{align}
     g_{\al \al'} \tau^{\al} \tau'^{\al'}  = g'^{\al \al'} \ell_{\al} \ell'_{\al'}
\end{align}
and thus we see that
$g'^{\al \al''} g_{\al'' \al'} = \delta_{\al \al'}$
such that finally
\begin{align}\label{ident}
    g'^{\al \al'} = g^{\al \al'}.
\end{align}
Moreover, by \eqref{dualmetric} and \eqref{ident}, we see that for $\zeta, \zeta'$ sufficiently smooth functions on $\mathcal{M}$ we have
\begin{align}\label{gradpairing}
    g\bra*{\nabla \zeta, \nabla \zeta'} = g'\bra*{\mathrm{diff} \zeta, \mathrm{diff} \zeta'} = g^{\al \al'} \partial_{\al} \zeta \partial_{\al'} \zeta'.
\end{align}
\subsection{Explicit formulae for partial derivatives}\label{parder}
For some function $f: \mathcal{M}_N \to \R$ we have
\begin{align}
     \partial_{i} f(h) = \frac{\dx}{\dx{\eps}}\Bigr|_{\eps = 0} f(h + \eps \hat{\varphi}_i)
\end{align}
as well as
\begin{align}
    \partial_{\al} f(h) = \frac{\dx}{\dx{\eps}}\Bigr|_{\eps = 0} f(h + \eps \bar{\varphi}_{\al}).
\end{align}
\section{Computation of the numerical mobility} \label{comp}
We restrict ourselves to mobility functions of the form $M(h) = h^m$. Then we compute
\begin{align}
    \grm_{\al \al}(h) &= \frac{1}{m-1} \frac{\bra*{h^{\al-}}^{1-m} - \bra*{h^{\al + }}^{1-m}}{h^{\al + } - h^{\al-}} \\
    &= \frac{1}{m-1} \frac{1}{h^{\al + } - h^{\al-}} \frac{\bra*{h^{\al + }}^{m-1} - \bra*{h^{\al-}}^{m-1}}{\bra*{h^{\al-}}^{m-1} \bra*{h^{\al + }}^{m-1}} \\
    &= \frac{1}{m-1} \frac{1}{h^{\al + } - h^{\al-}} \frac{\sum_{k=1}^{\infty} \binom{m-1}{k} \bra*{h_{\al-}}^{m-1-k} (h^{\al+} -h^{\al-})^ k}{\bra*{h^{\al-}}^{m-1} \bra*{h^{\al + }}^{m-1}} \\
    &= \frac{1}{m-1}\frac{\sum_{k=1}^{\infty} \binom{m-1}{k} \bra*{h^{\al-}}^{m-1-k} (h^{\al+} -h^{\al-})^{k-1}}{\bra*{h^{\al-}}^{m-1} \bra*{h^{\al + }}^{m-1}}.
\end{align}
In particular, this yields for $m=3$
\begin{align}
    \grm_{\al \al}(h) = \frac{1}{2} \frac{h^{\al-} + h^{\al + }}{\bra*{h^{\al-}}^2 \bra*{h^{\al+}}^2}
\end{align}
and hence
\begin{align}\label{metriccomp}
    \grm^{\al \al}(h) = 2 \frac{\bra*{h^{\al-}}^2 \bra*{h^{\al+}}^2}{h^{\al-} + h^{\al + }}.
\end{align}
Moreover, for the It\^o-correction term we are left with computing
% \begin{align}
%     \partial_i\bra*{A^T g^{-1}(h) A}^i \stackrel{\eqref{hattobar}}{=} \at^i_{\al'} \partial_{\al} g^{\al \al'}(h)
% \end{align}
% and by the chain rule, it holds that
\begin{align}\label{chainruleitocor}
    \partial_{\al'} \grm^{\al' \al}(h) = - g^{\gamma \gamma'}(h)\partial_{\gamma} \grm_{\gamma' \al'}(h)g^{\al'\al}(h)
\end{align}
and using \eqref{parder} we compute the derivative of the metric tensor via
\begin{align}\label{dergrm}
    \partial_{\gamma} \grm_{\gamma' \al'} (h) = - \delta_{\gamma' \al'}  \dashint_{I_{\al'}} \frac{M'(h) }{M(h)^2} \bar{\varphi}_{\gamma} \dx{x}.
\end{align}
By the diagonal structure of $g(h)$ it is enough to compute
\begin{align}\label{itocorcomp}
    \partial_{\al} \grm_{\al \al}(h) =  N^{\frac{3}{2}} \frac{ \frac{1}{M(h^{\al +})} + \frac{1}{M(h^{\al-})} - 2 \grm_{\al \al}(h)}{h^{\al + } - h^{\al-}}
\end{align}
where we used integration by parts which in the case $m = 3$ yields
\begin{align}
    \partial_{\al} \grm_{\al \al}(h) &= N^{\frac{3}{2}} \frac{\bra*{h^{\al-}}^{-3} + \bra*{h^{\al + }}^{-3} - \frac{h^{\al-} + h^{\al + }}{\bra*{h^{\al-}}^2 \bra*{h^{\al+}}^2}}{h^{\al + } - h^{\al-}} \\
    &= N^{\frac{3}{2}}  \frac{\bra*{h^{\al-}}^{-2}(\bra*{h^{\al-}}^{-1} - \bra*{h^{\al + }}^{-1}) + \bra*{h^{\al + }}^{-2} (\bra*{h^{\al + }}^{-1} - \bra*{h^{\al-}}^{-1})}{h^{\al + } - h^{\al-}}\\
    &= N^{\frac{3}{2}} \bra*{\frac{1}{\bra*{h^{\al-}}^3 h^{\al + }} - \frac{1}{\bra*{h^{\al + }}^3 h^{\al-}}}. \\
\end{align}
Hence, for $m=3$, we have
\begin{align} \label{itocnum}
    \partial_{\al'} g^{\al' \al}(h) = N^{\frac{3}{2}} 4 h^i h^{i+1} \frac{h^i - h^{i+1}}{h^i + h^{i+1}}.
\end{align}

\section{An integral inequality}
\begin{lem}\label{lemintineq}
    Let u(t) be positive and bounded for $t \in \pra*{0, T}$. Let $0 \leq \gamma <\infty$. Then, if
    \begin{align}\label{gamma1}
        u(t) \leq u(0) + C t + C\int_0^t u^{\gamma}(s) \dx s
    \end{align}
    for some constant $C$, we have for $\gamma = 1$
    \begin{align}
        u(t) \leq \bra*{u(0) + 1}e^{Ct}
    \end{align}
    and for $\gamma \neq 1$
    \begin{align}
        u(t) \leq\bra*{\bra*{u(0) + CT}^{1-\gamma} + (1-\gamma) C t}^{\frac{1}{1-\gamma}}.
    \end{align}
    % for some constant $C_{\gamma}$ depending on $\gamma$ and $C$.
\end{lem}
\begin{proof}
    For $\gamma = 1$ we note that we can write \eqref{gamma1} as
    \begin{align}
        u(t) + 1 \leq u(0) + 1 + C \int_0^t u(s) + 1 \dx s
    \end{align}
    and then apply Gronwall's inequality to get the assertion.
    
    If $\gamma < 1$ then we set $X(t) := \int_0^t u^{\gamma}(s) \dx s$ and hence 
    \begin{align}
        \frac{\dx}{\dx t} X(t) = u^{\gamma}(t) \stackrel{\eqref{gamma1}}{\leq} \bra*{u(0) + Ct + C X(t)}^{\gamma}
    \end{align}
    which implies
    \begin{align}\label{diffineq}
        \frac{\dx}{\dx t} \bra*{u(0) + CT + CX(t)} \leq C \bra*{u(0) + CT + C X(t)}^{\gamma}.
    \end{align}
    The differential inequality \eqref{diffineq} further yields
    \begin{align}
        \frac{\dx}{\dx t} \bra*{u(0) + CT + CX(t)}^{1 - \gamma} \leq C_{\gamma}
    \end{align}
    for $C_{\gamma} := (1-\gamma)C$ and since $X(0) = 0$ we have by integrating that
    \begin{align}
        \bra*{u(0) + CT + CX(t)}^{1 - \gamma} \leq \bra*{u(0) + CT}^{1 - \gamma} + C_{\gamma} t.
    \end{align}
    By taking the inverse and appealing again to the assumption \eqref{gamma1} we get the desired estimate.
\end{proof}
\bibliographystyle{abbrv}
\bibliography{refs}
\end{document}